\documentclass[twoside]{article}
\usepackage[utf8]{inputenc}

\usepackage[margin=1in]{geometry}

\title{On the accept-reject mechanism for Metropolis-Hastings algorithms}
\author{Nathan E. Glatt-Holtz, Justin A. Krometis, Cecilia F. Mondaini\\
\scriptsize{emails: negh@tulane.edu, jkrometi@vt.edu, cf823@drexel.edu}}
\date{}

\usepackage{amsfonts, amssymb, amsmath, amsthm,esint,multicol}
\usepackage[colorlinks=true, pdfstartview=FitV, linkcolor=blue,
            citecolor=blue, urlcolor=blue]{hyperref}
\usepackage[usenames]{color}
\definecolor{Red}{rgb}{0.7,0,0.1}
\definecolor{Green}{rgb}{0,0.7,0}
\usepackage{accents}
\usepackage{comment}
\usepackage{graphicx}
\usepackage[capitalize,nameinlink,noabbrev]{cleveref}
\usepackage{enumitem}

 \usepackage{mathtools} 

\usepackage{tabularx}

\numberwithin{equation}{section}

\newtheorem{Theorem}{Theorem}[section]
\newtheorem{Proposition}[Theorem]{Proposition}
\newtheorem{Lemma}[Theorem]{Lemma}
\newtheorem{Corollary}[Theorem]{Corollary}

\newtheorem{Definition}[Theorem]{Definition}
\newtheorem{Remark}[Theorem]{Remark}

\usepackage[section]{algorithm}      
\usepackage[noend]{algpseudocode}    

\newcommand{\be}{\begin{equation}}
\newcommand{\ee}{\end{equation}}

\newcommand{\RR}{\mathbb{R}}
\newcommand{\NN}{\mathbb{N}}

\newcommand{\indFn}[1]{1 \! \! 1_{#1}}

\newcommand{\Prb}{\mathbb{P}}

\newcommand{\bq}{\mathbf{q}}
\newcommand{\bz}{\mathbf{z}}
\newcommand{\btz}{\tilde{\mathbf{z}}}
\newcommand{\tbq}{\tilde{\mathbf{q}}}
\newcommand{\br}{\mathbf{r}}
\newcommand{\bu}{\mathbf{u}}
\newcommand{\bv}{\mathbf{v}}
\newcommand{\btq}{\tilde{\mathbf{q}}}
\newcommand{\btv}{\tilde{\bv}}
\newcommand{\btu}{\tilde{\mathbf{u}}}
\newcommand{\lbq}{\overline{\mathbf{q}}}
\newcommand{\lbv}{\overline{\mathbf{v}}}
\newcommand{\bqv}{\mathbf{z}}
\newcommand{\lbz}{\overline{\mathbf{z}}}
\newcommand{\bw}{\mathbf{w}}
\newcommand{\btw}{\tilde{\mathbf{w}}}

\newcommand{\gX}{\mathfrak{X}}
\newcommand{\gY}{\mathfrak{Y}}
\newcommand{\gM}{\nu}
\newcommand{\gN}{\rho}

\newcommand{\fm}{\phi}
\newcommand{\gm}{\psi}
\newcommand{\Vker}{\mathcal{V}}
\newcommand{\cM}{\mathcal{M}}

\newcommand{\Pker}{Q}
\newcommand{\Tker}{\bar{P}}
\newcommand{\pdms}{\eta}
\newcommand{\pdmsp}{\eta^\perp}
\newcommand{\msq}{\mu}
\newcommand{\msv}{\nu}
\newcommand{\Smap}{S}
\newcommand{\Fmap}{F}
\newcommand{\Gmap}{G}
\newcommand{\Bmap}{B_1}
\newcommand{\Cmap}{B_2}

\newcommand{\UPot}{\mathcal{U}}

\newcommand{\VPot}{\mathcal{K}}

\newcommand{\dT}{\delta}
\newcommand{\Pot}{\Phi}
\newcommand{\HmT}{\hat{\Psi}} 

\newcommand{\e}{\mathbf{e}}
\newcommand{\cC}{\mathcal{C}}
\newcommand{\cB}{\mathcal{B}}

\newcommand{\cN}{\mathcal{N}}
\newcommand{\cR}{\mathcal{R}}
\newcommand{\cS}{\mathcal{S}}
\newcommand{\cSS}{\mathcal{T}}
\newcommand{\cX}{\mathcal{X}}
\newcommand{\cK}{\mathcal{K}}

\newcommand{\arFn}{\alpha}
\newcommand{\harFn}{\hat\alpha}
\newcommand{\Sol}{\hat{S}}
\newcommand{\cSol}{\hat{\mathcal{S}}}

\newcommand{\Projq}{\Pi_1}

\newcommand{\Ham}{\mathcal{H}}

\newcommand{\bS}{\bar{S}}
\newcommand{\spq}{X}
\newcommand{\spv}{Y}

\newcommand{\Sstp}{\Gamma} 
\newcommand{\Si}{\Gamma^i} 

\newcommand{\da}{\delta_1}
\newcommand{\db}{\delta_2}
\newcommand{\agp}{f} 
\newcommand{\ndt}{\rho}
\newcommand{\fd}{N} 
\newcommand{\fu}{f_1}
\newcommand{\fv}{f_2}

\newcommand{\tJ}{\tilde{J}}
\newcommand{\tHam}{\widetilde{\mathcal{H}}}
\newcommand{\f}{\mathbf{f}}
\newcommand{\SA}{\hat{S}^{(A)}} 
\newcommand{\SB}{\hat{S}^{(B)}} 
\newcommand{\chS}{\mathcal{\hat{S}}}
\newcommand{\chT}{\mathcal{\hat{T}}}

\begin{document}
\markboth{On the accept-reject mechanism for Metropolis-Hastings algorithms.}
{N. E. Glatt-Holtz, J. A. Krometis, C. F. Mondaini}

\maketitle

\begin{abstract}
  This work develops a powerful and versatile framework for
  determining acceptance ratios in Metropolis-Hastings type Markov
  kernels widely used in statistical sampling problems.  Our approach
  allows us to derive new classes of kernels which unify random walk
  or diffusion-type sampling methods with more complicated `extended
  phase space' algorithms based around ideas from Hamiltonian
  dynamics.  Our starting point is an abstract result developed in the generality of measurable state spaces that addresses proposal kernels that possess a certain involution structure. Note that, while this
  underlying proposal structure suggests a scope which includes
  Hamiltonian-type kernels, we demonstrate that our abstract result
  is, in an appropriate sense, equivalent to an earlier general state
  space setting developed in \cite{Tierney1998} where the connection
  to Hamiltonian methods was more obscure.

  On the basis of our abstract results we develop several new classes
  of extended phase space, HMC-like algorithms.  Firstly we tackle the
  classical finite-dimensional setting of a continuously distributed
  target measure.  We then consider an infinite-dimensional framework
  for targets which are absolutely continuous with respect to a
  Gaussian measure with a trace-class covariance. Each of these
  algorithm classes can be viewed as `surrogate-trajectory' methods,
  providing a versatile methodology to bypass expensive gradient
  computations through skillful reduced order modeling and/or data
  driven approaches as we begin to explore in a forthcoming companion
  work, \cite{glatt2020LOLMCMC}.  On the other hand, along with the
  connection of our main abstract result to the framework in
  \cite{Tierney1998}, these algorithm classes provide a unifying
  picture connecting together a number of popular existing algorithms
  which arise as special cases of our general frameworks under
  suitable parameter choices. In particular we show that, in the
  finite-dimensional setting, we can produce an algorithm class which
  includes the Metropolis adjusted Langevin algorithm (MALA) and
  random walk Metropolis method (RWMC) alongside a number of variants
  of the HMC algorithm including the geometric approach introduced in
  \cite{GirolamiCalderhead2011}.  In the infinite-dimensional
  situation, we show that the algorithm class we derive includes the
  preconditioned Crank-Nicolson (pCN), $\infty$MALA and $\infty$HMC
  methods considered in \cite{beskos2008, Beskosetal2011,
    cotter2013mcmc} as special cases.
\end{abstract}

{\noindent
   }

\setcounter{tocdepth}{1}
\tableofcontents

\newpage

\section{Introduction}

A central concern in modern computational probability and statistics
is the development of effective sampling methods.  This is a
non-trivial task particularly for high-dimensional, non-canonical
probability distributions with elaborate correlation structures.
Indeed, across a great diversity of situations in the pure and applied
sciences and in engineering, it is crucial to be able to accurately
resolve observable quantities from complex statistical models which
naturally arise.  As such, sampling is a subject of ubiquitous
significance in a wide variety of application settings.

One of the most successful methodologies for sampling is the Markov
Chain Monte Carlo (MCMC) approach.  Starting from a given ``target''
probability distribution $\mu$ sitting on a state space $\spq$ one
aims to find a Markov kernel $P(\bq, d\tbq)$ which holds $\mu$ as an
invariant, that is
\begin{align}
  \int P(\bq, d\tbq)\mu(d\bq) = \mu(d\tbq).
    \label{eq:inv:msr:cond}
\end{align}
By iteratively sampling from such a kernel $P$ as
$\bq_n \sim P(\bq_{n-1}, d \br)$ one hopes that $P$ maintains
desirable mixing properties so that, for example,
\begin{align*}
    \lim_{N \to \infty} \frac{1}{N}\sum_{n =1}^N \phi(\bq_n)
    = \int \phi(\br)\mu(d\br)
\end{align*}
across a variety of observables $\phi: \spq \to \RR$.

Of course, deriving such kernels $P$ and determining their
effectiveness is an art, the subject of a wide and rapidly growing
literature.  Here the Metropolis-Hastings method,
\cite{metropolis1953equation, hastings1970monte}, has served as a core
foundation for the subject.  See also e.g. \cite{Li2008,
  robert2013monte, Betancourt2019} for more general background.  The
basic idea of this method is to formulate a ``proposal'' Markov kernel
$Q(\bq, d\tbq)$ which does not necessarily maintain the invariance
\eqref{eq:inv:msr:cond} but which is computationally feasible to
sample from.  One then introduces an ``accept-reject'' mechanism which
is used to correct for bias in $Q$ with respect to the given target
measure $\mu$. In this fashion one builds a `Metropolis-Hastings'
kernel $P$ from $Q$ which maintains \eqref{eq:inv:msr:cond} as
follows: Starting from a given current state $\bq_{n-1}$, one samples
a proposal state $\bar{\bq}_{n} \sim Q(\bq_{n-1}, d \br)$.  Next one
determines an acceptance probability $\arFn \in [0,1]$.  The next step in the
chain is then set to be $\bq_n := \bar{\bq}_{n}$ with probability $\arFn$ and $\bq_n := \bq_{n-1}$ otherwise.  Of course, the selection
of $Q$ and the derivation of the appropriate accept-reject mechanism
$\arFn$ depends heavily on the structure of $\mu$ and underlying state
space $\spq$ which $\mu$ sits upon.  Nevertheless, the
Metropolis-Hastings methodology encompasses many popular and effective
sampling methods, including random walk Monte Carlo (RWMC), the
Metropolis Adjusted Langevin (MALA) approach based on numerical
discretizations of appropriate stochastic dynamics
\cite{besag1994comments, roberts1996exponential}, and the Hamiltonian
(or Hybrid) Monte Carlo (HMC) algorithm \cite{duane1987hybrid,
  neal1993probabilistic}.

In this paper, we formulate a simple and quite flexible framework for
determining the acceptance ratio $\arFn$ developed in an abstract
setting applicable in the generality of measurable spaces.  The main
result of this work, \cref{thm:gen:rev:new}, unifies random walk or
diffusion-type approaches with more complicated ``extended phase
space'' algorithms like HMC.  In particular, while our framework
appears ``HMC-like'' at first glance, we show that it in fact
represents an alternative formulation of the measure-theoretic
approach to Metropolis-Hastings kernels introduced in
\cite{Tierney1998} where the connection to Hamiltonian algorithms is
more obscure. On the other hand, \cref{thm:gen:rev:new} allows us to
derive novel classes of HMC-like algorithms both in the finite and
infinite-dimensional settings.  These ``extended phase space''
algorithms provide a versatile methodology to bypass expensive
gradient computations through skillful reduced order modeling and/or
data driven approaches as we begin to explore in a forthcoming
companion work, \cite{glatt2020LOLMCMC}.  Moreover, our algorithms
provide a unified picture connecting a number of popular existing
algorithms which arise as special cases under suitable parameter
choices.  Altogether, the theoretical unity and reach of our main
result provides a basis for deriving novel sampling algorithms while
laying bare important relationships between existing methods.

\subsection{Background: Formulating the accept-reject mechanism}

Before outlining the main contributions of this work in more detail we
first lay out some further background on the Metropolis-Hastings
approach to give some context for our results here.

The original setting developed in \cite{metropolis1953equation,
  hastings1970monte}, addresses the case of a continuously distributed
target measure $\mu(d \br) = p(\br) d\br$ on a state space
$\spq = \RR^\fd$.  Here one considers proposal kernels which are also
of the continuous form $Q(\bq, d \br) = q(\bq, \br)d\br$.  A simple
calculation aimed at establishing detailed balance for the resulting Markov transition kernel $P$ as described above, namely that
\begin{align}\label{eq:det:bal}
	P(\bq, d \btq)\mu(d \bq) =P(\btq, d \bq)\mu(d \btq), 
\end{align}
a condition which is immediately seen to be sufficient for the invariance
\eqref{eq:inv:msr:cond}, yields the acceptance probability
\begin{align}
    \arFn(\bq,\tbq) := 
    1 \wedge \frac{ p(\tbq) q(\tbq, \bq) }{ p(\bq) q(\bq, \tbq) },
        \label{eq:hastings:A:R}
\end{align}
where $\bq$ is the current state and $\tbq$ is the proposed next move.

While \eqref{eq:hastings:A:R} encompasses a number of popular
algorithms including the RWMC and MALA methods, the determination of
the acceptance probably $\arFn$ can be a much more complicated and
delicate task in other cases of interest, particularly for popular
Hamiltonian (HMC) algorithms.  Moreover, one is often interested in
settings where the target measure sits on a more general state space.
Indeed, one class of Metropolis-Hastings algorithms motivating our
work here addresses the situation where $\spq$ is an
infinite-dimensional Hilbert space and one considers target measures
which are absolutely continuous with respect to a Gaussian probability
measure $\mu_0$ so that
\begin{align}
    \mu(d\bq) \propto 
    e^{- \Pot(\bq)}\mu_0(d\bq),
    \label{eq:gauss:ref:msr}
\end{align}
for an appropriate, $\mu_0$-integrable potential $\Pot: \spq \to
\RR$. This is an important and rich category of measures which arise
naturally in the Bayesian approach to PDE inverse problems,
\cite{stuart2010inverse,martin2012stochastic, bui2014analysis,
  dashti2017bayesian, PetraEtAl2014, BuiNguyen2016,
  dashti2017bayesian} and also in computational chemistry,
\cite{ReVa2005, hairer2005analysis, hairer2007analysis, HaStVo2009,
  HSV2011}.
  
One of our aims in this manuscipt is to provide new insight into the
derivation of a recently discovered collection of Metropolis-Hastings
algorithms \cite{beskos2008, cotter2013mcmc, Beskosetal2011,
  beskos2017geometric} that provide a basis to effectively sample from
such infinite-dimensional measures of the form
\eqref{eq:gauss:ref:msr}. The idea in these works,
\cite{beskos2008,cotter2013mcmc, Beskosetal2011}, is to appropriately
precondition stochastic or Hamiltonian dynamics related to the target
measure $\mu$ and then to make a delicate choice for the numerical
discretization of these equations so that one obtains an effective
proposal kernel $Q$. As usual, an appropriate $\arFn$ is then
introduced to correct for bias stemming from the original dynamics,
the numerical discretization of these dynamics, or both.  In
particular, this infinite-dimensional approach results in
an analogue of the random walk Monte Carlo algorithm with the so
called preconditioned Crank-Nicolson (pCN) algorithm as well as
infinite-dimensional formulations of the MALA \cite{beskos2008,
  cotter2013mcmc} and HMC \cite{Beskosetal2011, beskos2017geometric}
algorithms.  Each of these methods has shown great promise by
partially beating the ``curse of dimensionality'' as borne out by
recent theoretical developments \cite{hairer2014spectral,
  eberle2014error, glatt2020mixing, bou2020two} and by effectively
resolving certain challenging test problems \cite{BuiNguyen2016,
  beskos2017geometric, borggaard2020bayesian, glatt2020LOLMCMC}.

Regarding the derivation of the $\alpha$ for the pCN and for the
infinite-dimensional MALA algorithms from
\cite{beskos2008,cotter2013mcmc}, one can make use of the abstract
formulation due to Tierney \cite{Tierney1998} to determine $\arFn$.
Tierney's approach may be seen as an extension of
\cite{hastings1970monte} to general state spaces.  This is formulated
as follows: Given a proposal kernel $Q$ and a target measure $\mu$, if
the measures 
\begin{align}
  \pdms(d \bq, d\tbq) := Q(\bq, d \tbq) \mu(d \bq),
  \quad \pdmsp(d \bq, d\tbq) := \pdms(d \tbq, d\bq) 
    \label{eq:eta:rev:eta}
\end{align}
are mutually absolutely continuous then one can define the acceptance
probability $\arFn$ via the Radon-Nikodym derivative, namely
\begin{align}
  \arFn(\bq, \tbq)  :=  1 \wedge \frac{ d\pdmsp}{d \pdms}(\bq, \tbq)
  \label{eq:tieney:AR}
\end{align}
to achieve detailed balance \`a la \eqref{eq:det:bal}.  Here note that
in the appropriate finite-dimensional setting \eqref{eq:tieney:AR}
reduces to \eqref{eq:hastings:A:R}.

On the other hand, Tierney's elegant formulation in \cite{Tierney1998}
does not appear to cover HMC-type `extended phase space' algorithms in
an obvious way, even in the original finite-dimensional formulation
from \cite{duane1987hybrid}.  In HMC sampling, the proposal is
generated by interpreting the current state $\bq$ as a position
variable in a Hamiltonian system.  One identifies a Hamiltonian $\Ham$
such that the marginal of the Gibbs measure
\begin{align}
    \cM(d\bq ,d \bv) \propto e^{-\Ham(\bq, \bv)} d \bq d \bv
    \label{eq:Gibbs:formal}
\end{align}
onto position space corresponds to the desired target measure $\mu$.
Augmenting with a `momentum' (or sometimes `velocity') variable $\bv$
sampled from the $\bv$-marginal of $\cM$ and then integrating the
associate Hamiltonian dynamics in $(\bq,\bv)$-space via an
appropriately chosen approximate integrator $\hat{S}$, one obtains
\begin{equation}\label{eq:hmc:map}
    (\btq,\btv) := \hat{S}(\bq,\bv).
\end{equation}
The proposal $\btq$ is then given by
$\btq = \Projq \circ \hat{S}(\bq,\bv)$ where $\Projq$ represents
projection onto the position variable from the extended phase space.
Then the acceptance ratio $\harFn$ is specified as\footnote{For the
  infinite-dimensional formulation in \cite{Beskosetal2011},
  \eqref{eq:Gibbs:formal} is formal and made sense of with respect to
  a reference Gaussian measure while $\mathcal{H}$ is typically almost
  surely infinite.  As such, the ``change in energy''
  $\Delta \mathcal{H}(\bq, \bv) = \Ham(\bq,\bv)- \Ham(\hat{S}(\bq,\bv))$ in \eqref{eq:hmc:ar} is computed via a separate formula that
  accounts for a certain ``cancellation of infinities'' rather than by
  computing $\mathcal{H}$ directly. See \cref{sec:HMC:inf:dim} below
  and also \cite{Beskosetal2011} for further details.}
\begin{equation}\label{eq:hmc:ar}
  \harFn(\bq,\bv) 
  := 1 \wedge \exp \left[ \Ham(\bq,\bv)
                                      - \Ham\left(\hat{S}(\bq,\bv)\right) \right],
\end{equation}
where $\Ham$ is the Hamiltonian associated with the system.  

Note carefully that $\harFn$ in \eqref{eq:hmc:ar} depends not on the
current and proposed states $\bq,\btq$ as in \eqref{eq:tieney:AR}, but
rather on the current state $\bq$ and the initially sampled auxiliary
variable $\bv$.  As such, for Hamiltonian Monte Carlo algorithms, the
acceptance probability $\harFn$ is not accommodated by the mechanism
\eqref{eq:tieney:AR}, at least not in an obvious fashion. It is also
important to emphasize that the validity of this accept-reject
mechanism for the classical HMC algorithm \eqref{eq:hmc:ar} requires
one to formulate a numerical resolution $\hat{S}$ of the Hamiltonian
dynamics which respects certain delicate structural properties of the
original Hamiltonian system; namely one typically specifies $\hat{S}$
as a ``geometric integration scheme'', one which preserves volumes on
phase space and whose dynamics maintains certain `reversibility
properties' (see \eqref{eq:rev:mom:flip:inv} below).  The latter
reversibility condition reflects an indispensable underlying
involutive structure that we exploit here.

Starting from \cite{duane1987hybrid, neal1993probabilistic}, the core
ideas of this Hamiltonian approach have expanded into a profusion of
`extended phase space' methods.  This literature includes a number of
variations on the Hamiltonian and symplectic structure leading to the
formulation of \eqref{eq:hmc:map} as well as use of `surrogate
dynamics' methods to reduce the cost or complexity of numerically
expensive gradient computations which arise.  See e.g.
\cite{bernardo1998regression,rasmussen2003gaussian,Li2008,GirolamiCalderhead2011,neal2011mcmc,Beskosetal2011,hoffman2014no,meeds2014gps,lan2016emulation,AFM2017,zhang2017precomputing,zhang2017hamiltonian,beskos2017geometric,lu2017relativistic,li2019neural,radivojevic2020modified} and numerous other containing
references.  This large and rapidly expanding literature is reflective
of the success and effectiveness of the Hamiltonian approach as for
example can be seen in the wide adoption of the STAN software package,
\cite{GelmanLeeGuo2015, Stan2016}, in recent years.

It is therefore of great interest to provide a unified theoretical
foundation for these various works on extended phase space methods and,
if possible, to place them in the context of the original RWMC methods
dating all the way back to \cite{metropolis1953equation,
  hastings1970monte}.  A number of recent contributions preceding this
current work are notable in this regard.  For example the monograph
\cite{bou2018geometric} building on \cite{fang2014compressible}
provides a lucid and partially self-contained survey explaining
various underlying mechanisms involved in the derivation of
\eqref{eq:hmc:ar} in the original finite-dimensional setting.  On the
other hand recent work \cite{levy2017generalizing,
  radivojevic2020modified,neklyudov2020involutive,andrieu2020general} has identified the connection between HMC
methods and the involutive algorithms of Green
\cite{green1995reversible, geyer2003metropolis,
  geyer2011introduction}.  Regarding the infinite-dimensional version
of the HMC algorithm developed in \cite{Beskosetal2011,
  beskos2013advanced, beskos2017geometric} the finite-dimensional
setting is insufficient to derive an appropriate $\harFn$ as in
\eqref{eq:hmc:ar}.  In the original framing from \cite{Beskosetal2011}
this infinite-dimensional case was justified via a finite-dimensional
approximation scheme reminiscent of an extended body of work on the
invariance of the Gibbs measure for Hamiltonian PDEs, see
\cite{bourgain1994periodic} and more recently
\cite{nahmod2019randomness, benyi2019probabilistic} for a
comprehensive survey of this research direction.  On the other hand
subsequent work \cite{beskos2013advanced, beskos2017geometric},
provides notable insights into the underlying structural
considerations at play in the formulation of the infinite dimensional
setting by making delicate use of the Cameron-Martin Theorem.

\subsection{Overview of Our Contribution}

We turn finally to describe our contributions herein.  The main result
of this work is \cref{thm:gen:rev:new}, which provides a single,
simple formulation that subsumes all of the above algorithms and, we
expect, many other methods of interest not directly addressed here.
The result is developed in the generality of measurable state spaces
$(\spq, \Sigma_\spq)$ to reversibly sample from a given target
probability measure $\mu$ defined on $\Sigma_\spq$.

To proceed we consider Markov kernels formulated as follows: select an
additional measurable space $(\spv, \Sigma_\spv)$ and form the
extended phase space $\spq \times \spv$. Fix any proposal map
$S: \spq \times \spv \to \spq \times \spv$ which is an involution
operation, namely such that $S^2 = I$, where we denote here and throughout $S^2 \coloneqq S \circ S$, and any reference Markov kernel
$\Vker: \spq \times \Sigma_{\spv} \to [0,1]$.  We obtain a proposal
kernel $Q(\bq,d \btq)$ from $S$ and $\Vker$ by sampling
$\bv \sim \Vker(\bq, d \bw)$ and then taking a proposed step as
$\btq = \Pi_1 \circ S(\bq, \bv)$, where $\Pi_1$ is the projection onto
the ``position variable'', namely $\Pi_1(\btq, \btv) = \btq$.  Thus,
in measure theoretic language, we have
$Q(\bq, d \tbq) = (\Pi_1 \circ S(\bq, \cdot))^* \Vker(\bq, d \tbq)$
where $f^{*}\nu$ denotes the \emph{pushforward} of a measure $\nu$ by
a function $f$; see \cref{subsec:prelim:meas} below for details.

For such data $\mu$, $S$ and $\Vker$, our result then determines a
suitable acceptance probability as a function of $\bq$ and $\bv$ given
by
\begin{align}\label{eq:ar:abs:intro}
	\harFn(\bq, \bv) := 1 \wedge 
	\left(\frac{d S^* \cM}{d \cM}(\bq,\bv) \right),
	\quad \text{ where }
	\quad
	\cM( d \bq, d\bv) :=  \Vker(\bq, d \bv) \mu( d \bq)
\end{align}
and where $d S^* \cM/d \cM$ denotes the Radon-Nikodym derivative of
the pushforward measure $S^* \cM$ with respect to $\cM$.  Note that
this acceptance ratio is well-defined when $S^* \cM$ is absolutely
continuous with respect to $\cM$ but this condition can be relaxed, see \cref{rem:noabscont} below.
By adopting this measure theoretic language of pushforwards and
Radon-Nikodym derivatives, as reviewed for our purposes in
\cref{subsec:prelim:meas}, \cref{thm:gen:rev:new} affords a simple
proof based around the intuition of using $\harFn$ to balance inflows
and outflows between any two states $\bq$ and $\btq$ of the Markov
chain.

We observe that \cref{thm:gen:rev:new} has a direct interpretation as
an abstraction of the HMC method, but one which provides the crucial
insight that the role of the `Gibbs measure' as in
\eqref{eq:Gibbs:formal} and that of the `numerical integrator' as in
\eqref{eq:hmc:map} can be largely disconnected insofar as achieving
the detailed balance condition, \eqref{eq:det:bal}, is concerned.  In
this analogy $\cM$, given as in \eqref{eq:ar:abs:intro}, specifies the
Gibbs measure.  Regarding the involution $S$, we notice
that the solution map $\hat{S}$ of a Hamiltonian system or any
reasonable numerical resolution thereof is not typically an
involution.  However, by taking $S := R \circ \hat{S}$, where
$R(\bq, \bv) = (\bq,-\bv)$ is the momentum flip operation, one does
indeed obtain an involution for a large class of ``geometric''
integration schemes $\hat{S}$.  Thus, so long as $S^*\cM$ is
absolutely continuous with respect to $\cM$, the algorithmic elements
$\mu$, $\Vker$ and $S = R \circ \hat{S}$ are otherwise unrelated
insofar as the scope of \cref{thm:gen:rev:new} is concerned.

With \cref{thm:gen:rev:new} in hand, we proceed to demonstrate the
reach and theoretical unity that this result provides by detailing how
it can be used to derive and to analyze a variety of specific
Metropolis-Hastings type algorithms.  We first observe that we can use
\cref{thm:gen:rev:new} to recover the complete framework of the more
traditional Metropolis-Hastings techniques up to and including the
abstract generality of Tierney's formulation in \cite{Tierney1998}.
Indeed, in \cref{sec:tierney}, we show that the formalisms in
\cite{Tierney1998} and in \cref{thm:gen:rev:new} may be ultimately
viewed as having an equivalent scope.  On the other hand the
`HMC-like' character of \cref{thm:gen:rev:new} described in the
previous paragraph suggests an immediate connection to Hamiltonian
type extended phase space methods.  In \cref{sec:approx:ham:methods},
we use this observation as a crucial starting point for deriving a
variety of HMC-like sampling methods culminating in the derivation of
\cref{alg:gHMC:FD:1}, \cref{alg:gHMC:FD:2}, \cref{alg:gHMC:FD:3} and
\cref{alg:method} provided below.  
The final \cref{sec:class:ex} outlines multiple ways that a selection of important sampling methods widely considered in the literature may be subsumed under \cref{thm:gen:rev:new}. Here to obtain our selection of $\Vker$ and $S$ for a given target $\mu$, we draw on various results in \cref{sec:approx:ham:methods} and
\cref{sec:tierney}.

Regarding the equivalence of \cite{Tierney1998} and
\cref{thm:gen:rev:new} we first consider in \cref{subsec:inv:F}
proposal kernels of the form
\begin{align}\label{eq:push:Q:ker}
	Q(\bq, d \tbq) = F(\bq, \cdot)^* \Vker(\bq, d \tbq), 
\end{align}
for some $F: \spq \times \spv \to \spq$ such that $F(\bq, \cdot)$ is
invertible for each fixed $\bq$.  We then demonstrate that there
exists a unique involution $S$ whose projection onto the position
space is $F$. With this $S$ in hand, we show that the main result in
\cite{Tierney1998}, namely the formulation \eqref{eq:tieney:AR},
follows as a special case of \cref{thm:gen:rev:new}.  See
\cref{prop:inv:S}, \cref{thm:Tierney} below for precise details.

Note that the scope of \eqref{eq:push:Q:ker} trivially encompasses the
results in \cite{Tierney1998} as seen, by selecting, for any given
proposal kernel $Q$, $F(\bq,\bv) := \bv$ and $\Vker := Q$.  On the
other hand, numerous relevant examples including RWMC and MALA and
even their Hilbert space counterparts,
\cite{beskos2008,cotter2013mcmc}, can be recovered from nontrivial
formulations of $F$ and $\Vker$.  This leads to the interesting
observation that algorithms of interest can be recovered through
multiple, non-equivalent applications of \cref{thm:gen:rev:new}.

Conversely, in the other direction, we observe in
\cref{subsec:alt:proof} that the results from \cite{Tierney1998} can
be suitably employed to provide a second, independent, proof of
\cref{thm:gen:rev:new}.  Here, from the given data of
\cref{thm:gen:rev:new}, namely the target $\mu$, the involution $S$
and kernel $\Vker$, we proceed by considering a suitable deterministic
proposal kernel specified by $S$ acting on the extended phase space.  The
results in \cite{Tierney1998} are then applied for $\mu = \cM$ on
this extended phase space.  In this way we obtain a kernel which, when
appropriately integrated against $\Vker$, yields the desired kernel
specified by \cref{thm:gen:rev:new}.  These details are given in
\cref{subsec:alt:proof} below.

In \cref{sec:approx:ham:methods} we use \cref{thm:gen:rev:new} and the
intuition it provides to develop some classes of concrete,
extended-phase space algorithms.  We may view our methods as
introducing additional `degrees of freedom' for the selection of
parameters used for tuning HMC methods. Crucially the approach
includes functional parameters which allow greater latitude in
selecting the numerical integration procedure used for approximating
Hamiltonian dynamics.  In particular our formulation provides a
flexible means of using reduced order modeling or data driven
approaches to avoid expensive gradient computations which arise.  As
already mentioned above, we begin to explore such applications for
these algorithms in a forthcoming companion work
\cite{glatt2020LOLMCMC}.  Note furthermore that our methods provide a
unified view for an extensive existing literature around so called
`surrogate trajectory methods' \cite{bernardo1998regression,
  rasmussen2003gaussian, Li2008, neal2011mcmc, meeds2014gps,
  lan2016emulation, zhang2017precomputing, zhang2017hamiltonian,
  beskos2017geometric, li2019neural, radivojevic2020modified}.  On the
other hand, in \cref{sec:class:ex}, we show that a number of classical
variants of HMC, MALA, RWMC as well as the Hilbert space methods
$\infty$HMC, pCN, $\infty$MALA all fall as special cases of these
algorithms introduced in \cref{sec:approx:ham:methods}, under an
appropriate choice of algorithmic parameters.

In \cref{sec:ghm:main} we begin by developing the finite-dimensional
setting of a continuously distributed target measure.  To this end we
provide a brief but self-contained tour of some geometric numerical
methods for Hamiltonian systems highlighting how these methods connect
back to the setting of \cref{thm:gen:rev:new}.  In particular this
presentation emphasizes the fundamental role that volume preserving
methods, palindromic splitting structures and reversibility play in
determining the involutive mapping $S$ at the center of
\cref{thm:gen:rev:new}, while suggesting further scope for deriving
other Hamiltonian-type sampling methods in the finite-dimensional
setting.  Notably we address situations where the integrator is not
assumed to have a gradient structure in \cref{alg:gHMC:FD:1} and
generalize the use of some implicit integrators developed in
\cite{GirolamiCalderhead2011} culminating in \cref{alg:gHMC:FD:3}.
Furthermore \cref{prop:Ham:props:gen} and \cref{prop:HMC:Leap:Frog}
suggest possible variations on the usual setting of momentum flip
reversibility which may be of some use in future algorithmic
developments.

In \cref{sec:HMC:inf:dim} we turn to address the infinite-dimensional
Gaussian setting where our target measure has the form
\eqref{eq:gauss:ref:msr}.  Here we derive a quite general surrogate
trajectory method in \cref{alg:method}.  Note that this algorithm
includes $\infty$HMC, pCN, $\infty$MALA as well as geometric
variations from \cite{beskos2017geometric} as important special cases.
Our Hilbert space approach takes as its starting point the method
introduced in \cite{Beskosetal2011}.  In this Gaussian base measure
setting, the implied Hamiltonian from \cite{Beskosetal2011} has the
form $\Ham := \Ham_1 + \Ham_2$ with
\begin{align}\label{split:ham:intro}
  \Ham_1(\bq, \bv) := \frac{1}{2}|\cC^{-1/2} \bq|^2
               + \frac{1}{2}|\cC^{-1/2} \bv|^2,
  \qquad	
  \Ham_2(\bq, \bv) := \Phi(\bq),
\end{align}
The crucial insight in \cite{Beskosetal2011} is to use an appropriate
preconditioning operator $J$ along with a particular Strang splitting
defined around $\Ham_1$ and $\Ham_2$ to derive a numerical integrator
of the form
$\hat{S} = (\Xi^{(1)}_{\delta/2} \circ \Xi_\delta^{(2)} \circ
\Xi^{(1)}_{\delta/2})^n$ for the algorithmic parameters $\delta > 0$,
the size of the numerical time step, and $n = T/\delta$ where $T >0$
is the total integration time.  Here, for any $t > 0$,
\begin{align}
   \Xi_t^{(1)}(\bq, \bv) = (\bq, \bv - t \cC D \Phi(\bq)),
	\quad   
  \Xi_t^{(2)}(\bq, \bv) = (\cos(t)\bq + \sin(t) \bv, - \sin(t) \bq + \cos(t)\bv).
  \label{eq:intro:inf:int}
\end{align}
This delicate choice of splitting and preconditioning results, in the
language of our framework in \cref{thm:gen:rev:new}, in an involution
$S = R \circ \hat{S}$, where again $R$ is the momentum flip operation,
such that $S^* (\mu \otimes \mu_0)$ is absolutely continuous with
respect to $\mu \otimes \mu_0$ or, in more heuristic but concrete
terms, in a ``cancellation of infinities" as would appear in
\eqref{eq:hmc:ar}.

Our generalization of \cite{Beskosetal2011} to \cref{alg:method}
centers on the observation that we can replace the term $D\Phi$ in
\eqref{eq:intro:inf:int} with essentially any reasonable function
$\agp: X \to X$ and replace the velocity proposal kernel
$\Vker = \mu_0$ suggested by \eqref{split:ham:intro} with any $\Vker$
such that $\Vker(\bq,d \bv)$ is absolutely continuous with respect to
$\mu_0(d \bv)$ for any $\bq$.  Indeed, in this more general setting, we
still obtain an involution $S$ which maintains the absolute continuity
of $S^* \cM$ with respect to
$\cM(d\bq,d\bv) \propto e^{-\Phi(\bq)}\mu_0(d\bq) \Vker(\bq,d\bv)$
required by \cref{thm:gen:rev:new}. As already alluded to above, the
proof of invariance of the target $\mu$ in \cite{Beskosetal2011}
follows an involved spectral approximation approach analogous to
\cite{bourgain1994periodic}.  Here in \cref{thm:rev:HMC} we proceed
similarly to \cite{beskos2013advanced} and provide a different and
more direct proof based on the Cameron-Martin theorem which reduces
the problem to computing
$d S^* (\mu_0 \otimes \mu_0)/d (\mu_0 \otimes \mu_0)$ and making use
of a simple identity for iterated pushforward maps, given as
\eqref{RN:push:comp} below.  Indeed, our involution $S$ involves
repeated applications of the maps $\Xi_t^{(1)}$, $\Xi_t^{(2)}$; the
latter map $\Xi_t^{(2)}$ holds $\mu_0 \otimes \mu_0$ invariant, while
the Radon-Nikodym derivative associated with the pushforward of
$\mu_0 \otimes \mu_0$ by $\Xi_t^{(1)}$ can be computed via the
Cameron-Martin Theorem.

\medskip
\noindent{\bf An additional note on references:} During final
preparation of the initial draft of this manuscript, the coauthors
attended a seminar \cite{sanzserna2020haussdorf} wherein an as yet
unpublished result due to C. Andrieu was described that appeared
similar in some ways to \cref{thm:gen:rev:new}.  Since our initial
submission, the work \cite{andrieu2020general} has now appeared
publicly.  This contribution includes material, discovered completely
independently by Andrieu and his co-authors, which closely parallels
some of the ideas considered here including \cref{thm:gen:rev:new}.

\section{General Formulation of the Accept-Reject Mechanism}
\label{sec:main:result}

This section is devoted to the main result of the paper which we
present as \cref{thm:gen:rev:new} in \cref{subsec:main:res}.  Before
stating and proving this main result we briefly recall some measure
theoretic terminologies and facts in \cref{subsec:prelim:meas}. For
more details regarding this preliminary material, we refer the reader
to e.g. \cite{Bogachev2007,Folland1999,Aliprantis2013}.

\subsection{Preliminaries on Measure Theory}
\label{subsec:prelim:meas}

Let $(\gX, \Sigma_\gX)$ and $(\gY, \Sigma_\gY)$ be measurable
spaces. Given a measurable function $\fm: \gX \to \gY$ and a measure
$\gM$ on $(\gX, \Sigma_\gX)$, the \emph{pushforward of $\gM$ by $\fm$},
denoted as $\fm^* \gM$, is defined as the measure on $\gY$ given by
\begin{align*}
  \fm^* \gM (A) \coloneqq \gM(\fm^{-1}(A))
  \quad \mbox{ for any } A \in \Sigma_\gY.
\end{align*}
Clearly, for any measurable functions $\fm_1, \fm_2$ mapping between
appropriate spaces, we have that
\begin{align}
  (\fm_1 \circ \fm_2)^* \gM = \fm_1^*(\fm_2^*\gM).
  \label{eq:pushf:comp}
\end{align}
We recall moreover that, given a $(\fm^*\gM)$-integrable function
$\gm: \gY \to \RR$, i.e.  $\gm \in L^1(\fm^* \gM)$, it follows that the
composition $\gm \circ \fm : \gX \to \RR$ is in $L^1(\gM)$ and the
following change of variables formula holds
\begin{align}
  \label{eq:CoV:push}
  \int_\gY \gm(\bw) \fm^* \gM(d \bw)
       = \int_\gX \gm (\fm (\bw)) \gM(d \bw).
\end{align}

Let us observe that if $\bar{\bw}$ is a random variable sampled from a
probability measure $\gM$ then $\fm(\bar{\bw})$ is distributed as
$\fm^* \gM$.  Furthermore it is worth noticing that, in the special
case when $\gX = \RR^N$ and $\gM$ is any Borel measure on $\RR^N$
which is absolutely continuous with respect to Lebesgue measure,
namely when
\begin{align*}
  \gM(d \bw) = p(\bw) d \bw
\end{align*}
for some density function $p: \RR^N \to \RR$, then for any
diffeomorphism $\fm : \RR^N \to \RR^N$, we have
\begin{align}
  \fm^* \gM(d \bw) = p(\fm^{-1}(\bw)) |\det \nabla \fm^{-1} (\bw)| d \bw.
  \label{eq:push:forward:den}
\end{align}

Next recall that, given measures $\gM$ and $\gN$ on a measurable space
$(\gX, \Sigma_\gX)$, we say that $\gM$ is \emph{absolutely
  continuous with respect to $\gN$}, and write $\gM \ll \gN$, if
$\gM(A) = 0$ whenever $\gN(A) = 0$, for $A \in \Sigma_\gX$. If $\gM$ and
$\gN$ are two sigma-finite measures on $(\gX, \Sigma_\gX)$ such that
$\gM \ll \gN$ then there exists a $\gN$-almost unique function
$d\gM/d\gN \in L^1(\gN)$ such that
\begin{align}\label{eq:RN}
  \gM(A) = \int_A \frac{d\gM}{d\gN}(\bw) \gN(d \bw),
  \quad A \in \Sigma_\gX,
\end{align}
called the \emph{Radon-Nikodym derivative} of $\gM$ with respect to
$\gN$.  In particular, if $\gM_1$, $\gM_2$ and $\gN$ are sigma-finite
measures on $(\gX, \Sigma_\gX)$ with $\gM_1 \ll \gN$ and $\gM_2 \ll \gN$,
namely
\begin{align*}
    \gM_1(d\bw) = \fm_1(\bw) \gN(d\bw), \quad \gM_2(d \bw) = \fm_2(\bw)\gN(d\bw),
\end{align*}
where $\fm_1 = d\gM_1/d\gN$ and $\fm_2 = d\gM_2/d\gN$, and if $\fm_2 > 0$
$\gN$-a.e. then $\gM_1 \ll \gM_2$ and
 \begin{align}
  \frac{d \gM_1}{d \gM_2}(\bw) = \frac{\fm_1(\bw)}{\fm_2(\bw)}
  \quad \mbox{ for $\gN$-a.e. } \bw \in \gX.
  \label{eq:RN:Der:form}
\end{align}
It also immediately follows from \eqref{eq:RN} that, given
sigma-finite measures $\gM$, $\gN$ and $\gamma$ on
$(\gX, \Sigma_\gX)$ such that $\gM \ll \gN$ and $\gN \ll \gamma$,
then $\gM \ll \gamma$ and
\begin{align}\label{chain:rule:meas}
  \frac{d\gM}{d\gamma}(\bw)
  = \frac{d\gM}{d\gN}(\bw) \frac{d\gN}{d\gamma} (\bw)
  \quad \mbox{ for $\gamma$-a.e. } \bw \in \gX. 
\end{align}
Moreover, given a measurable and invertible mapping
$\phi: \gX \to \gX$ with measurable inverse
$\phi^{-1}: \gX \to \gX$\footnote{Here it is worth pointing out that
  when $\gX$ is a Polish space then, for any measurable and
  invertible mapping $\phi: \gX \to \gX$, its inverse $\phi^{-1}$ is
  always a measurable mapping, see e.g. \cite[Theorem
  12.29]{Aliprantis2013}}, and sigma-finite measures $\gM$ and $\gN$
on $(\gX, \Sigma_\gX)$ with $\gM \ll \gN$, it follows that
$\phi^*\gM \ll \phi^*\gN$ and
\begin{align}\label{RN:pushfwd}
	\frac{d \phi^*\gM}{d \phi^* \gN}(\bw) =
  \frac{d\gM}{d\gN}(\phi^{-1}(\bw))
  \quad \mbox{ for ($\phi^*\gN$)-a.e } \bw \in \gX.
\end{align}
This follows by noticing that, for any ($\phi^*\gM$)-integrable
function $\psi: \gX \to \RR$, we have from \eqref{eq:CoV:push} and
\eqref{eq:RN} that
\begin{align*}
  \int_\gX \psi(\bw) \phi^*\gM(d\bw)
  = \int_\gX \psi(\phi(\bw)) \gM(d\bw)
  = \int_\gX \psi(\phi(\bw)) \frac{d\gM}{d\gN}(\bw) \gN(d\bw)
  = \int_\gX \psi(\bw) \frac{d\gM}{d\gN}(\phi^{-1}(\bw)) \phi^*\gN(d\bw).
\end{align*}

Finally we observe that, given a sigma-finite measure $\gM$ on
$(\gX, \Sigma_\gX)$, and a sequence $\fm_i: \gX \to \gX$,
$i = 1, \ldots, n$, of measurable and invertible functions with
measurable inverses $\fm_i^{-1}: \gX \to \gX$ and such that
$\fm_i^* \gM \ll \gM$ for each $i = 1, \ldots, n$, then
$(\fm_n \circ \cdots \circ \fm_1)^* \gM \ll \gM$ and
\begin{align}\label{RN:push:comp}
  \frac{d (\fm_n \circ \cdots \circ \fm_1)^*\gM}{d \gM} (\bw)
  = \frac{d\fm_n^* \gM}{d\gM}(\bw) \prod_{i=1}^{n-1}
  \frac{d \fm_i^* \gM}{d \gM}((\fm_n\circ \cdots \circ \fm_{i+1})^{-1}(\bw))
  \quad \mbox{ for $\gM$-a.e. } \bw \in \gX,
\end{align}
To derive this identity, notice that for every
$((\fm_n \circ \ldots \circ \fm_1)^*\gM)$-integrable function
$\psi: \gX \to \RR$ we have
 \begin{align*}
   &\int \psi(\bw) \left( \fm_n \circ \dots \circ \fm_2 \circ \fm_1\right)^{*} \gM(d\bw)\\
   &=
     \int \psi(\fm_n \circ \dots \circ \fm_2 (\bw)) \fm_1^{*}\gM(d\bw)
     =
     \int \psi(\fm_n \circ \dots \circ \fm_2 (\bw)) 
        \frac{d\fm_1^{*}\gM}{d\gM}(\bw)\gM(d\bw) 
   \\
   &=
     \int \psi(\fm_n \circ \dots \circ \fm_3(\bw)) 
     \frac{d\fm_1^{*}\gM}{d\gM}\left(\fm_2^{-1}(\bw)\right)\fm_2^{*}\gM(d\bw) 
   =
   \int \psi(\fm_n \circ \dots \circ \fm_3(\bw)) 
   \frac{d\fm_1^{*}\gM}{d\gM}\left(\fm_2^{-1}(\bw)\right)
   \frac{d\fm_2^{*}\gM}{d\gM}\left(\bw\right)\gM(d\bw) 
   \\&=
   \int \psi(\fm_n \circ \dots \circ \fm_4(\bw)) 
   \frac{d\fm_1^{*}(\gM)}{d\gM}\left(\fm_2^{-1}\circ \fm_3^{-1}(\bw)\right)
   \frac{d\fm_2^{*}\gM}{d\gM}\left(\fm_3^{-1}(\bw)\right)\fm_3^{*}\gM(d\bw) 
   \\&= \cdots = \int \psi(\bw) \left(\prod_{i=1}^{n-1}
   \frac{d \fm_i^* \gM}{d \gM}((\fm_n\circ \cdots \circ \fm_{i+1})^{-1}(\bw)) \right) 
   \frac{d \fm_n^*\gM}{d\gM}(\bw) \gM(d \bw).
\end{align*}

\subsection{The Main Result}\label{subsec:main:res}

As alluded to above in the introduction, our main theorem shows how to
define an acceptance probability that yields a reversible sampling
algorithm when the proposal kernel is given in terms of an involution
$S$ defined on an extended parameter space, namely $\spq \times \spv$ 
for measurable spaces $(\spq, \Sigma_\spq)$ and $(\spv, \Sigma_\spv)$. More
specifically, denoting by $\Pi_1: \spq \times \spv \to \spq$ the projection
mapping onto the first component, i.e.
\begin{align}
  \label{eq:proj:map:def}
  \Pi_1 (\bq,\bv) = \bq \quad \text{ for all } (\bq, \bv) \in \spq \times \spv,
\end{align}
we consider the case of proposal kernels given as
$Q(\bq, d \btq) = (\Pi_1 \circ S(\bq, \cdot))^*\Vker(\bq,\cdot) (d \btq)$, 
for some Markov kernel $\Vker: \spq \times \Sigma_\spv \to [0,1]$.

To avoid dealing with further technical measure theoretical details,
in the statement below and throughout the manuscript we apply the
general results from \cref{subsec:prelim:meas} concerning sigma-finite
measures in the particular context of probability measures.

\begin{Theorem}\label{thm:gen:rev:new}
  Let $(\spq, \Sigma_\spq)$ and $(\spv, \Sigma_\spv)$ be measurable
  spaces. Let $\mu$ be a probability measure on $\spq$, and let
  $\Vker: \spq \times \Sigma_\spv \to [0,1]$ be a Markov kernel. Let
  $\cM$ be the probability measure on $\spq \times \spv$ defined as
  \begin{align}
    \label{def:ext:target}
    \cM(d \bq, d\bv) = \Vker(\bq, d \bv) \mu(d\bq).
  \end{align}
  Suppose there exists a measurable mapping
  $S: \spq \times \spv \to \spq \times \spv$ satisfying the following
  properties:
    \begin{enumerate}[label={(P\arabic*)}]
        \item\label{P1c} $S$ is an involution, i.e. $S^2 = I$;
        \item\label{P2c} $S^* \cM$ is absolutely continuous with respect to $\cM$.
    \end{enumerate}
 Let $\harFn: \spq \times \spv \to \RR$ be the function defined by
 \begin{align}
   \label{def:arFn:b}
   \harFn(\bq, \bv) :=
   1 \wedge \frac{d S^*\cM}{d \cM} (\bq,\bv), \quad (\bq, \bv) \in \spq \times \spv, 
 \end{align}
 and let $P: \spq \times \Sigma_\spq \to [0,1]$ be the Markov
 transition kernel defined as
 \begin{align}
  \label{eq:ext:MH:ker:diff}
  P(\bq, d\btq)
  = \int_\spv \harFn(\bq, \bv)
       \delta_{\Projq \circ S(\bq, \bv)} (d\btq) \Vker (\bq, d\bv)
       + \delta_{\bq}(d\btq) \int_\spv (1 - \harFn(\bq, \bv)) \Vker (\bq, d\bv),
 \end{align}
 for $\bq \in X$.
 Then, $P$ satisfies \textit{detailed balance} with respect to $\mu$, i.e.
 \begin{align}
   P(\bq, d \btq) \mu(d\bq) = P(\btq, d\bq) \mu(d\btq),
   \label{eq:rev:mu:P:def}
 \end{align}
 so that, in particular, $\mu$ is $P$ invariant.
\end{Theorem}

Before presenting the proof, we try to provide some intuition. Under
the assumptions, if $S$ maps $(\bq,\bv)$ to $(\btq,\btv)$, then $S$
will map $(\btq,\btv)$ back to $(\bq,\bv)$. That is, if a proposal
$\btq$ can be generated from the Markov kernel starting from $\bq$,
then a proposal $\bq$ can similarly be generated starting from
$\btq$. By appropriately selecting $\arFn$ to balance these two flows
-- from $(\bq,\bv)$ to $S(\bq,\bv)$ and vice versa -- we can achieve
detailed balance and therefore invariance. We now turn to the
details.

\begin{proof}
  It suffices to show that for every bounded measurable function
  $\varphi: \spq \times \spq \to \RR$ we have
 \begin{align}
 \label{eq:testfc:rev:b}   
   \int_\spq \int_\spq \varphi(\bq, \btq) P(\bq, d \btq) \mu(d\bq)
   = \int_\spq \int_\spq \varphi(\bq, \btq) P(\btq, d \bq) \mu(d\btq).
 \end{align}
 From the definitions of $P$ and $\cM$ in \eqref{eq:ext:MH:ker:diff},
 \eqref{def:ext:target}, and Fubini it follows that
\begin{align*}
\int_\spq \int_\spq \varphi(\bq, &\btq) P(\bq, d \btq) \mu (d\bq) \\
 &= \int_\spq \int_\spv \varphi(\bq, \Projq \circ S(\bq, \bv))
                             \harFn(\bq, \bv) \cM(d\bq, d \bv)
   + \int_\spq \int_\spv \varphi(\bq, \bq) (1 - \harFn(\bq, \bv))
             \cM(d\bq, d \bv) \\
&=: (I) + (II).
\end{align*}
Analogously, for the right-hand side of \eqref{eq:testfc:rev:b} we have
\begin{align*}
\int_\spq \int_\spq \varphi(\bq, &\btq) P(\btq, d \bq) \mu(d\btq) 
\\
&= 
\int_\spq \int_\spv \varphi(\Projq \circ S(\btq, \bv), \btq) \harFn(\btq, \bv) \cM(d\btq, d \bv)
+ \int_\spq \int_\spv \varphi(\btq, \btq) (1 - \harFn(\btq, \bv)) \cM(d\btq, d \bv)
\\
&=: (III) + (IV).
\end{align*}
Clearly, $(II) = (IV)$. We now show that $(I) = (III)$. 

Invoking assumption \ref{P1c}, we obtain
\begin{align*}
	(I) &=
              \int_\spq \int_\spv \varphi(\bq, \Pi_1 \circ S (\bq, \bv))
                       \, \harFn(\bq, \bv) \, \cM(d\bq, d \bv)\\
	&=
   \int_\spq \int_\spv \varphi(\Pi_1 \circ S^2 (\bq, \bv),
                        \Pi_1 \circ S (\bq, \bv) )
               \,  \harFn(S^2 (\bq, \bv)) \, \cM(d\bq, d \bv).
\end{align*}
Thus, by Fubini, the change of variables formula, \eqref{eq:CoV:push},
and then invoking assumption \ref{P2c}, we write
\begin{align}
  (I) &=
        \int_{\spq \times \spv} \varphi(\Pi_1 \circ S(\bq, \bv), \bq)
        \, \harFn(S(\bq, \bv)) \, S^* \cM(d \bq, d\bv)
  \notag\\
      &= \int_{\spq \times \spv} \varphi(\Pi_1  \circ S(\bq, \bv), \bq)
        \, \harFn(S(\bq, \bv)) \, \frac{d S^* \cM}{ d \cM} (\bq, \bv)
        \,  \cM (d \bq, d\bv).
        \label{eq:I}
\end{align}
Hence, in order to conclude that $(I) = (III)$, it suffices to show that 
\begin{align}\label{eq:RN:a0}
\harFn(S (\bq, \bv)) \, \frac{d S^*\cM}{d \cM} (\bq, \bv)
= 
\harFn(\bq, \bv)
\quad \mbox{ for  $\cM$-a.e. }  (\bq, \bv) \in \spq \times \spv. 
\end{align}
Since $d S^*\cM / d \cM (\bq, \bv) \geq 0$
$\cM$-a.e., from the definition of $\harFn$ in
\eqref{def:arFn:b} it follows that
\begin{align}\label{eq:RN:a}
	\harFn(S (\bq, \bv)) \, \frac{d S^*\cM}{d \cM} (\bq, \bv)
	&= 
	\frac{d S^*\cM}{d \cM} (\bq, \bv)
	\wedge
	\left( \frac{d S^*\cM}{d \cM} (\bq, \bv)
	\frac{d S^*\cM}{d \cM} (S (\bq, \bv))
	\right) 
\end{align}
for $\cM$-a.e. $(\bq, \bv) \in \spq \times \spv$. Now from
\eqref{RN:push:comp} and assumption \ref{P1c} we have that
\begin{align}\label{eq:RN:S:new}
     \frac{d S^*\cM}{d \cM} (\bq, \bv)
	\frac{d S^*\cM}{d \cM} (S (\bq, \bv))
	=
	\frac{d (S^2)^*\cM}{d \cM} (\bq, \bv) 
	= 1
	\quad \mbox{ for  $\cM$-a.e. }  (\bq, \bv) \in \spq \times \spv. 
\end{align}
Plugging \eqref{eq:RN:S:new} into \eqref{eq:RN:a}, we obtain
\eqref{eq:RN:a0}. This concludes the proof.
\end{proof}

The MCMC sampling scheme resulting from \cref{thm:gen:rev:new} is
summarized in \cref{alg:method:abs}.

\begin{algorithm}[H]
\caption{}
\begin{algorithmic}[1]\label{alg:method:abs}
    \State Choose $\bq_0 \in X$
    \For{$k \geq 0$}
        \State Sample $\bv_k \sim \Vker(\bq_0, \cdot)$
        \State Propose $\lbq_{k+1} = \Projq \circ S(\bq_k,\bv_k)$
        \State Set $\bq_{k+1} =\lbq_{k+1}$ with probability
                   $\harFn(\bq_{k}, \bv_{k})$ given by \eqref{def:arFn:b},
                   otherwise $\bq_{k+1} = \bq_{k}$ 
    \EndFor
\end{algorithmic}
\end{algorithm}

We conclude this section with a number of remarks clarifying the
scope of \cref{thm:gen:rev:new}.

\begin{Remark}
  If in \cref{thm:gen:rev:new} we assumed in addition that $\spq$ and
  $\spv$ are Radon spaces (i.e. separable metric spaces on which every
  probability measure is tight), then we could take $\cM$ to be any
  probability measure on $\spq \times \spv$ with first marginal $\mu$,
  i.e. $\Pi_1^* \cM = \mu$. Indeed, in this case it follows from the
  disintegration theorem that there exists a Markov kernel
  $\Vker: \spq \times \Sigma_Y \to [0,1]$ such that $\cM$ is written
  as in \eqref{def:ext:target}, see e.g. \cite[Theorem
  5.3.1]{AGL2008}.
\end{Remark}

\begin{Remark}
  Assumptions \ref{P1c} and \ref{P2c} in \cref{thm:gen:rev:new} imply
  that the measures $S^*\cM$ and $\cM$ are in fact mutually absolutely
  continuous. Indeed, if $E \subset \spq \times \spv$ is such that
  $S^*\cM (E) = 0$, then by definition of pushforwards we have
  $\cM(S^{-1}(E)) = 0$. Thus assumption \ref{P2c} implies that
  $S^* \cM(S^{-1}(E)) = 0$, and with \ref{P1c} we obtain
 \[
    0 = S^* \cM(S^{-1}(E)) = 
    \cM(S^{-1}(S^{-1}(E))) = \cM(S^2(E)) = \cM(E).
 \]
 Hence, $\cM(E) = 0$, so that $\cM \ll S^*\cM$.
\end{Remark}

\begin{Remark}[Generalizations of \cref{thm:gen:rev:new}]
  \label{rem:noabscont} We notice that the statement of
  \cref{thm:gen:rev:new} in fact holds under a more general
  form. Namely, similarly as in \cite{Tierney1998}, we could disregard
  assumption \ref{P2c} and define the acceptance probability $\harFn$
  in \eqref{def:arFn:b} as
  $\harFn(\bq, \bv) = 1 \wedge \frac{d S^*\cM}{d\cM}(\bq, \bv)$ for
  all $(\bq, \bv)$ in a measurable subset
  $\mathcal{O} \subset \spq \times \spv$ where
  $S^*\cM_{\mathcal{O}} \ll \cM_{\mathcal{O}}$, and $0$
  otherwise, where here $\Gamma_{\mathcal{O}}$ denotes the restriction of a measure $\Gamma$ to the set $\mathcal{O}$. However, it is worth pointing out that in practice this
  set $\mathcal{O}$ could be empty, in which case $\harFn \equiv 0$
  and the corresponding MCMC algorithm would always reject the
  proposals, a clearly undesirable behavior.

  Moreover, instead of defining $\harFn$ explicitly as in
  \eqref{def:arFn:b}, we could have taken $\harFn$ to be any
  measurable function for which \eqref{eq:RN:a0} holds (particularly
  on the set $\mathcal{O}$, under the setting of the previous
  paragraph). Indeed, \eqref{def:arFn:b} gives just one such
  example. It is worth pointing however that, as noticed in
  \cite[Section 3]{Tierney1998}, the standard choice given in
  \eqref{def:arFn:b} yields the maximum acceptance probability among
  all possible choices of $\harFn$ satisfying \eqref{eq:RN:a0}, since
\[
\harFn(S (\bq, \bv)) \, \frac{d S^*\cM}{d \cM} (\bq, \bv)
= 
\harFn(\bq, \bv) 
\leq 1 \wedge  \frac{d S^*\cM}{d \cM} (\bq, \bv),
\quad (\bq, \bv) \in \spq \times \spv.
\]

Finally, assumption \ref{P1c} can be replaced by the following more
general condition:
\begin{enumerate}[label={($P$\arabic*')}]
\item\label{P1p} $(S^2)^*\cM = \cM$, $\Pi_1 \circ S^2 = \Pi_1$, and
  $S$ is invertible with measurable inverse
  $S^{-1}: \spq \times \spv \to \spq \times \spv$.
\end{enumerate}
Indeed, we arrive at the same identity \eqref{eq:I} by noticing that
\begin{align*}
  (I) &= \int_\spq \int_\spv \varphi(\bq, \Pi_1 \circ S (\bq, \bv)) 
            \harFn(\bq, \bv) \cM(d\bq, d \bv)
      =
        \int_\spq \int_\spv \varphi(\bq, \Pi_1 \circ S (\bq, \bv)) 
              \harFn(\bq, \bv) (S^2)^*\cM(d\bq, d \bv) \\
      &=
        \int_\spq \int_\spv \varphi(\Pi_1 \circ S (\bq, \bv), \Pi_1 \circ S^2 (\bq, \bv))  
              \harFn(S(\bq, \bv)) S^* \cM(d\bq, d \bv) \\
      &= 
        \int_\spq \int_\spv \varphi(\Pi_1 \circ S(\bq, \bv), \bq) 
           \harFn(S(\bq, \bv)) S^*\cM(d\bq, d \bv).
\end{align*}
Moreover the conclusion in \eqref{eq:RN:S:new} still holds since with
the assumption $(S^2)^*\cM = \cM$ it follows that
$S^*\cM = (S^{-1})^*(S^2)^*\cM = (S^{-1})^* \cM$. Therefore, invoking
\eqref{RN:push:comp}, we obtain
\begin{align*}
	\frac{d S^* \cM}{d \cM}(\bq, \bv) \frac{d S^* \cM}{d \cM} (S(\bq, \bv)) 
	=
	\frac{d (S^{-1})^* \cM}{d \cM}(\bq, \bv) \frac{d (S^{-1})^* \cM}{d \cM}(S(\bq, \bv)) 
	=
	\frac{d ((S^{-1})^2)^*\cM}{d \cM}(\bq, \bv) = 1 
\end{align*}
for $\cM$-a.e. $(\bq, \bv )\in \spq \times \spv$, where the last
identity holds since
$((S^{-1})^2)^*\cM = ((S^{-1})^2)^*(S^2)^*\cM = \cM$.

Note that, all of the applications of \cref{thm:gen:rev:new} developed in
\cref{sec:approx:ham:methods}, \cref{sec:class:ex} below fall within
the particular involutive structure of \cref{thm:gen:rev:new}, \ref{P1c}. It
would thus be interesting to identify an example of MCMC algorithm
that would require the more general framework from \ref{P1p}.
\end{Remark}

\begin{Remark}\label{rmk:R:hatS}
  For many MCMC algorithms, the associated proposal kernel is defined
  from a mapping $S$ as in \cref{thm:gen:rev:new} that is in turn
  given in terms of a numerical integrator of a suitably chosen
  dynamics. Denoting such an integrator by a measurable mapping
  $\Sol: \spq \times \spv \to \spq \times \spv$, one commonly seeks a
  corresponding measurable mapping
  $R: \spq \times \spv \to \spq \times \spv$ such that
\begin{align}\label{R:Sol:rev}
 R \circ \Sol = \Sol^{-1} \circ R.
\end{align}
In the context of vector spaces $\spq, \spv$, the condition
\eqref{R:Sol:rev} is referred to in the theory of geometric
integrators as the \emph{reversibility of $\Sol$ with respect to $R$},
and $R$ is often taken as a linear invertible mapping (see
e.g. \cite[Definition V.1.2]{Hairer2006book}). Frequently, one can
find $R$ commensurate with \eqref{R:Sol:rev} which is also an
involution, i.e. $R^2 = I$. For example, in the case of the HMC
algorithm, where $\Sol$ represents a numerical integrator of a
suitable Hamiltonian dynamics, $R$ is commonly taken as the
``momentum-flip'' involution $R(\bq, \bv) = (\bq, - \bv)$, see
\cref{sec:approx:ham:methods} and \cref{subsubsec:fd:HMC},
\cref{subsubsec:pHMC} below.

We thus define $S \coloneqq R \circ \Sol$, and conclude from
\eqref{R:Sol:rev} and $R^2 = I$ that $S^2 = I$, i.e. $S$ is an
involution, so that assumption \ref{P1c} of \cref{thm:gen:rev:new} is
satisfied. In fact, if instead of $R^2 = I$ we assume more generally
that $(R^2)^*\cM = \cM$, $\Pi_1 \circ R^2 = \Pi_1$, and $R$ is
invertible with measurable inverse, then clearly \eqref{R:Sol:rev}
implies that $S^2 = (R \circ \Sol)^2 = R^2$, so that $S$ satisfies the
more general condition \ref{P1p} in \cref{rem:noabscont}. Here again
it would be interesting to identify an example of MCMC algorithm for
which this more general setting for $R$ is required.
\end{Remark}

\begin{Remark}[Connection with the Metropolis-Hastings-Green Algorithm]
\label{rmk:alg:Rn:Preliminary}
As already mentioned in the introduction, \cref{alg:method:abs} may be
seen as a generalization to abstract state spaces of the so called
Metropolis-Hastings-Green algorithm \cite{green1995reversible,
  geyer2003metropolis, geyer2011introduction}.  We can make this
connection explicit as follows.  Let $X=\RR^\fd$ and fix any
continuously distributed target probability measure
$\mu(d\bq) = p(\bq)d\bq$.  We consider any Markovian proposal kernel
$\Vker: \RR^N \times \mathcal{B}(\RR^M) \to [0,1]$ of the form
$\Vker(\bq, d \bv) = q(\bq, \bv) d \bv$ where
$q: \RR^{N \times M} \to [0,\infty)$ is such that
$\int_{\RR^M} q(\bq, \bv) d \bv =1$ for any $\bq \in \RR^N$.  Finally
select any $S: \RR^{N+M} \to \RR^{N + M}$ which is $C^1$ and is an
involution, namely we assume that $S \circ S(\bz) = \bz$ for every
$\bz \in \RR^{N + M}$.  Then, referring back to \eqref{def:arFn:b} and
recalling \eqref{eq:push:forward:den}, \eqref{eq:RN:Der:form}, we
obtain a Markov kernel $P$ of the form
\eqref{eq:ext:MH:ker:diff} with the acceptance probability given by
\begin{align}
      \harFn(\bq, \bv)
      = 1 \wedge \frac{\rho(S(\bq, \bv)) | \det \nabla S(\bq, \bv)|}{\rho(\bq, \bv)}
       \quad \text{ where } \rho(\bq, \bv) := p(\bq)q(\bq,\bv).
          \label{eq:Gen:HMC:AR}
\end{align}
Thus with these specifications for $\mu$, $\Vker$ and $S$ we see that
the algorithm derived in \cite{green1995reversible} (see also
\cite[Section 1.2]{geyer2003metropolis}) falls out as a special case
of \cref{alg:method:abs}.
    
\cite{green1995reversible} also considered the possibility of
combining proposals from a collection of different sampling mechanisms
where the proposal mechanism employed is selected at random according
to a state dependent probability.  This `multi-kernel' approach
developed in \cite{green1995reversible} can be generalized to an
abstract context similar to \cref{thm:gen:rev:new} in a fashion which
recovers the results in \cite{green1995reversible} as a special case.
To see this we proceed as follows.  Fix a measurable space
$(\spq, \Sigma_\spq)$ along with a collection of auxiliary measurable
spaces $(\spv_j, \Sigma_{\spv_j})$ defined for $j = 1, \ldots, L$. Our
target measure $\mu$ is any probability measure on $\spq$.  On each
extended phase space $X \times Y_j$ for $j = 1, \ldots, L$ we assume
that we have defined a Markov proposal kernel $\Vker_j(\bq, d \bv_j)$
and a measurable involution $S_j : X \times Y_j \to X \times Y_j$ such
that $S_j^* \cM_j$ is absolutely continuous with respect to $\cM_j$
where $\cM_j(d \bq, d\bv) = \Vker_j(\bq, d \bv) \mu(d\bq)$.  Finally
we suppose we have defined, for each $j = 1,\ldots, L$, a measurable function
$\kappa_j: X \to [0,1]$ in such a way that
$\sum_{j =1}^L \kappa_j (\bq) = 1$ for each $\bq \in X$.  With these
ingredients in hand we define
$P_j: \spq \times \Sigma_\spv \to [0,1]$, for $j = 1, \ldots L$ to be
the Markov transition kernels given as
 \begin{align*}
  P_j(\bq, d\btq)
  = \int_\spv \harFn_j(\bq, \bv_j)
       \delta_{\Projq \circ S_j(\bq, \bv_j)} (d\btq)  \Vker_j (\bq, d\bv_j)
       +  \delta_{\bq}(d\btq) \int_\spv (1 - \harFn_j(\bq, \bv_j)) \Vker_j (\bq, d\bv_j),
 \end{align*}
 with
 \begin{align*}
   \harFn_j(\bq, \bv) :=
   1 \wedge \left( \frac{\kappa_j (\Projq \circ S_j(\bq,\bv_j)) } {\kappa_j (\bq)} 
   \frac{d S^*_j \cM_j}{d  \cM_j} (\bq,\bv_j) \right), 
   \quad (\bq, \bv_j) \in \spq \times \spv_j.
 \end{align*}
 We now define a compound kernel $P: \spq \times \Sigma_\spq \to [0,1]$
 as $P(\bq, d\btq) = \sum_{j =1}^L \kappa_j(\bq)P_j(\bq, \tbq)$.
 Following the proof of \cref{thm:gen:rev:new} mutatis mutandis we obtain
 that $P$ is reversible with respect to $\mu$.  Moreover, in the special case when
 $\mu(d \bq) = p(\bq) d\bq$, $Y_j = \RR^{M_j}$, 
 $\Vker_j(\bq, d \bv_j) = q_j(\bq,  \bv_j)d \bv_j$, we recover the formulation
 in \cite{green1995reversible} as described in  
 \cite[Section 1.3]{geyer2003metropolis}.

Note that the connection between the Metropolis-Hasting-Green algorithm
and Hamiltonian Monte Carlo sampling, where the Hamiltonian dynamic
or more precisely a well chosen numerical discretization of the dynamics
provides an associated involution, seems to be more 
recent.  See  \cite{levy2017generalizing, neklyudov2020involutive}.
Indeed, below in \cref{sec:ghm:main} we show how the usage of
Hamiltonian dynamics can allow one to derive a broad class of
involutions $S$ and thus to recover numerous formulations of
Hamiltonian Monte Carlo (essentially) as special cases of the
general observation leading to \eqref{eq:Gen:HMC:AR}.
\end{Remark}

\section{Connection to the Tierney Framework}
\label{sec:tierney}

In this section, we connect our main result, \cref{thm:gen:rev:new},
to the formulation of reversible Metropolis-Hastings algorithms on
general state spaces $\cX$ laid out previously in \cite{Tierney1998}.
We show that the two frameworks connect or overlap in several
different and interesting ways.  On the one hand, in
\cref{subsec:inv:F} with \cref{thm:Tierney}, we prove that
\cref{thm:gen:rev:new} subsumes the main result in \cite{Tierney1998}.
In fact we demonstrate that this connection can often be made in multiple, 
non-equivalent ways; namely, in cases of interest,
a variety of different choices for $S$ and $\Vker$ in
\eqref{eq:ext:MH:ker:diff} can ultimately yield the same desired
Markov kernel $\Tker$ specified in Tierney's formulation.  In
particular note that \cref{thm:Tierney} is employed in
\cref{sec:class:ex} below to provide further insights for the derivation
of various classical MCMC algorithms.  On the other hand, in
\cref{subsec:alt:proof}, we establish that the main results in
\cite{Tierney1998} can be used to develop a proof of
\cref{thm:gen:rev:new} distinct from the one presented above in
\cref{sec:main:result}.

\subsection{Overview of Tierney's formulation}

Before turning to the main results in this section let us begin by
recalling some of the framework and notations from
\cite{Tierney1998}. As above in \cref{sec:main:result} we let
$(\cX, \Sigma_{\cX})$ be any measurable space. Starting with a target
probability measure $\mu$ on $\cX$ and a proposal Markov kernel
$\Pker: \cX \times \Sigma_{\cX} \to [0,1]$, \cite{Tierney1998}
considers the Metropolis-Hastings type Markov kernel defined as
\begin{align} \label{eq:kern:tierney}
  \Tker (\bq,d\btq) = \arFn(\bq,\btq)\Pker(\bq,d\btq)
           + \delta_\bq(d\btq) \int_\spq
           (1-\arFn(\bq,\br))\Pker(\bq,d\br),
           \quad \bq \in \cX,
\end{align}
with the acceptance ratio $\arFn$ given by
\begin{align}\label{eq:ar:tierney}
  \arFn(\bq,\btq) :=
  1 \wedge \frac{d\pdmsp}{d\pdms}(\bq,\btq),
  \quad \bq, \btq \in \cX.
\end{align}
Here
the measures $\pdms,\pdmsp$ are given by
\begin{align}\label{eq:meas:tierney}
  \pdms(d\bq,d\btq) = \mu(d\bq)Q(\bq,d\btq)
  \qquad \pdmsp(d\bq,d\btq) = \pdms(d\btq,d\bq)
     =\mu(d\btq)Q(\btq,d\bq)
\end{align}
and are assumed to be mutually absolutely
continuous, so that in particular $\arFn$ in \eqref{eq:ar:tierney} 
is well-defined.\footnote{In fact \cite{Tierney1998} allows for $\pdms,\pdmsp$ to
  be mutually absolutely continuous only on a subset
  $\mathcal{O} \subset \spq \times \spq$, in which case $\arFn$ is
  defined to be zero on $\mathcal{O}^c$. A similar generalization
  holds for \cref{thm:gen:rev:new}; see \cref{rem:noabscont}. But for
  simplicity of presentation we restrict ourselves to the case of
  mutual absolute continuity here.}
Tierney then shows, \cite[Theorem 2]{Tierney1998}, that \eqref{eq:kern:tierney},
 \eqref{eq:ar:tierney} yields a Markov kernel satisfying detailed balance  
with respect to $\mu$, namely, cf. \eqref{eq:rev:mu:P:def},
$\Tker (\bq,d\btq)\mu(d \bq) = \Tker (\btq,d\bq)\mu(d \btq)$.

As an important preliminary observation we notice that the formulation
in \cite{Tierney1998}, that is \eqref{eq:kern:tierney},
\eqref{eq:ar:tierney}, can be recovered from \cref{thm:gen:rev:new} in
a straightforward fashion as follows:
\begin{Remark}\label{rmk:flip:it:baby}
  Given the inputs $\Pker$, $\mu$ we take $\spq := \cX$, $\spv := \cX$
  and set
\begin{align} \label{eq:flip:it:baby}
  \Vker(\bq,d\bv) :=Q(\bq,d\bv),
  \qquad S(\bq,\bv) := (\bv,\bq).
\end{align}
Clearly $S$ is an involution and it is easy to see that, with these
choices (cf. \eqref{def:ext:target}), $\cM = \pdms$ and
$S^*\cM = \pdmsp$ so that $dS^*\cM/ d \cM = d\pdmsp/d\pdms$.
Therefore \eqref{eq:ext:MH:ker:diff} reduces to
\eqref{eq:kern:tierney} under this $\Vker$ and $S$.
\end{Remark}

\subsection{Reduction to the Tierney Formulation}
\label{subsec:inv:F}

The observation in \cref{rmk:flip:it:baby}, recovering
\eqref{eq:kern:tierney} from \eqref{eq:ext:MH:ker:diff}, can be extended
more broadly.  Here a starting point is to notice that we are in fact
writing the proposal kernel in \eqref{eq:kern:tierney} as
$\Pker(\bq, d \btq) = F(\bq,\cdot)^*\Vker(\bq, \cdot)(d \btq)$ but in
a trivial fashion where $\Pker = \Vker$ and $F(\bq,\bv) = \bv$. The
insight is that a desired $Q$, along with the involution $S$ required for
\eqref{eq:ext:MH:ker:diff} which recovers the kernel
\eqref{eq:kern:tierney}, can often be identified from other,
non-trivial, choices for $F$ and $\Vker$.

We formulate this generalization as follows. Let $(\spq, \Sigma_\spq)$
and $(\spv, \Sigma_\spv)$ be measurable spaces.  We consider proposals
$\btq \sim Q(\bq, d \btq)$ made from some starting point
$\bq \in \spq$ which are generated in the following fashion:
\begin{enumerate}
\item Draw $\bv \sim \Vker(\bq, \cdot)$, for some Markov kernel
  $\Vker: \spq \times \Sigma_\spv \to [0,1]$.
\item Compute $\btq = \Fmap(\bq,\bv)$ from a (measurable) deterministic map
  $\Fmap:\spq \times \spv \to \spq$.
\end{enumerate}
In other words we are considering the proposal kernel $Q$ in the
formulation
\begin{align}
  \Pker(\bq,d\btq)
  = \int_\spv \delta_{\Fmap(\bq,\bv)}(d\btq) \Vker(\bq,d\bv)
  = \Fmap(\bq, \cdot)^* \Vker(\bq, \cdot)(d \btq).
  \label{ass:kern:tierney}
\end{align}
We assume throughout what follows that, for a given sample $\bq \in \spq$
and proposal $\btq \in \spq$, $\Fmap$ can be inverted to obtain the
$\bv \in \spv$ such that $\Fmap(\bq,\bv)=\btq$, namely,
\begin{align}\label{ass:kern:inv}
  \text{for each
  $\bq \in \spq$,
  the map $\Fmap(\bq,\cdot):\spv \to \spq$ is
  one-to-one.}
\end{align}

\begin{Remark}
  One may formulate $Q$ as in \eqref{ass:kern:tierney} maintaining
  \eqref{ass:kern:inv} in a nontrivial fashion for MCMC methods with
  additive noise such as the RWMC, pCN, and MALA schemes, all of which
  fall into the framework \eqref{eq:kern:tierney} outlined in
  \cite{Tierney1998}.  In other words, for each of these examples, we
  may determine $F$ and $\Vker$ in a form distinct from
  \eqref{eq:flip:it:baby}; see, for example, \eqref{eq:RWMC:prop},
  \eqref{eq:MALA:prop:comp}, \eqref{def:F:pCN}, and \eqref{def:F:mala}
  in \cref{sec:class:ex} below.  By contrast, \eqref{ass:kern:inv}
  does not encompass HMC, for which there may be more than one $\bv$
  leading to the same proposal.
\end{Remark}

We now show in \cref{thm:Tierney} that in the formulation
\eqref{ass:kern:tierney} under \eqref{ass:kern:inv} we may obtain a suitable
involution $S$ which yields an equivalence between the Markov
transition kernels in \eqref{eq:ext:MH:ker:diff},
\cref{thm:gen:rev:new}, and \eqref{eq:kern:tierney}, from
\cite{Tierney1998}.  As a preliminary step we show how to construct
this involution $\Smap$ in $(\bq,\bv)$-space corresponding to any
$\Fmap$ satisfying \eqref{ass:kern:inv}.

\begin{Proposition}\label{prop:inv:S}
  Let $\spq$, $\spv$ be any sets and let
  $\Fmap: \spq \times \spv \to \spq$ be a mapping satisfying
  \eqref{ass:kern:inv}, i.e. such that for each fixed $\bq \in \spq$,
  $\Fmap(\bq, \cdot): \spv \to \spq$ is one-to-one.  Then, there
  exists a unique mapping
  $\Smap: \spq \times \spv \to \spq \times \spv$ such that
  $\Projq \circ \Smap = \Fmap$ and $\Smap^2 = I$, given by
  \begin{align}\label{def:inv:S}
	\Smap(\bq, \bv) = (\Fmap(\bq, \bv), \Fmap(\Fmap(\bq, \bv),\cdot)^{-1}(\bq))
          \quad \mbox{ for all } (\bq, \bv) \in \spq \times \spv.
  \end{align}
\end{Proposition}
\begin{proof}
  First, let us verify that $\Smap$ defined in \eqref{def:inv:S}
  satisfies the required properties. Clearly,
  $\Projq \circ \Smap = \Fmap$, so it remains to show that
  $\Smap^2 = I$. 
  Define the maps $\Bmap: \spq \times \spv \to \spq \times \spq$ and 
  $\Cmap: \spq \times \spq \to \spq \times \spq$ by
    \begin{equation}
        \Bmap(\bq,\bv) = (\bq,\Fmap(\bq,\bv)), 
        \qquad   
        \Cmap(\bq,\btq) = (\btq,\bq)
        \label{eq:tierney:BCmaps}
    \end{equation}
    for all $\bq \in \spq$, $\btq \in \spq$ and $\bv \in \spv$. 
    Note that $\Cmap^2=I$ trivially and by assumption \eqref{ass:kern:inv} we have
    \begin{align*}
        \Bmap^{-1}(\bq,\btq)=(\bq,\Fmap(\bq,\cdot)^{-1}(\btq)) 
        \quad \mbox{ for all } (\bq, \btq) \in \spq \times \spq.
    \end{align*}
    Then it is not difficult to check that $\Smap$ as in
    \eqref{def:inv:S} can be written as
    \begin{equation}
        \Smap(\bq,\bv) = \Bmap^{-1} \circ \Cmap \circ \Bmap(\bq,\bv) 
        \quad \mbox{ for all } (\bq,\bv) \in \spq \times \spv.
        \label{eq:jkbs:S}
    \end{equation}
    Then clearly $S^2=I$ since $\Cmap^2 = I$.

  Now suppose that $\bS: \spq \times \spv \to \spq \times \spv$ is any
  mapping satisfying the required properties, i.e.
  $\Projq \circ \bS = \Fmap$ and $\bS^2 = I$. Let
  $\Gmap: \spq \times \spv \to \spv$ such that
  \begin{align*}
    \bS(\bq, \bv) = (\Fmap(\bq, \bv) , \Gmap(\bq, \bv))
    \quad \mbox{ for all } (\bq, \bv) \in \spq \times \spv.
  \end{align*}
  Thus,
  \begin{align*}
	\bS^2 (\bq, \bv) 
	= \bS(\bS(\bq, \bv)) 
    = (\Fmap(\Fmap(\bq, \bv), \Gmap(\bq, \bv)),
       \Gmap(\Fmap(\bq, \bv), \Gmap(\bq, \bv))).    
  \end{align*}
  Since $\bS^2 = I$, it follows in particular that
  \begin{align*}
    \Fmap(\Fmap(\bq, \bv), \Gmap(\bq, \bv)) = \bq,
  \end{align*}
  which implies 
  \begin{align*}	
	\Gmap(\bq, \bv) =  	\Fmap(\Fmap(\bq, \bv), \cdot)^{-1}(\bq).
  \end{align*}
  Therefore, $\bS = \Smap$. This concludes the proof.
\end{proof}

With \cref{prop:inv:S} in hand we now turn to our equivalence result:
\begin{Theorem}\label{thm:Tierney}
  Let $\spq$ and $\spv$ be measurable spaces. Take $\mu$ to be a probability measure
  on $\spq$ and consider an associated proposal kernel $\Pker$
  satisfying \eqref{ass:kern:tierney}, \eqref{ass:kern:inv},
  i.e. 
  $\Pker(\bq, d \btq)= (\Fmap(\bq, \cdot))^* \Vker(\bq, \cdot)(d
  \btq)$, where $\Vker: \spq \times \Sigma_\spv \to [0,1]$ is a Markov
  kernel and $\Fmap: \spq \times \spv \to \spq$ is a measurable mapping such that $\Fmap(\bq, \cdot): \spv \to \spq$ is
  one-to-one for each fixed $\bq \in \spq$ and, additionally, its inverse $F(\bq,\cdot)^{-1}$ is measurable\footnote{As remarked in \cref{subsec:prelim:meas}, here we notice again that that if $\spq$ and $\spv$ are Polish spaces, i.e. separable and completely metrizable, then the fact that, for each fixed $\bq \in \spq$, $F(\bq, \cdot)$ is a measurable and one-to-one mapping between Polish spaces automatically implies that its inverse is measurable, see e.g. \cite[Theorem 12.29]{Aliprantis2013}.}. We define the
  probability measures $\pdms, \pdmsp$ on $\spq \times \spq$ as in
  \eqref{eq:meas:tierney} and let
  $\Smap: \spq \times \spv \to \spq \times \spv$ be the unique
  mapping satisfying $\Pi_1 \circ \Smap = \Fmap$ and $\Smap^2 = I$,
  given by \eqref{def:inv:S} in \cref{prop:inv:S}.  Then, denoting by
  $\cM$ the probability measure on $\spq \times \spv$ defined as in
  \eqref{def:ext:target}, we have
  \begin{itemize}
  \item[(i)] the measure $\pdmsp$ is absolutely continuous with
    respect to $\pdms$ if and only if $\Smap^*\cM$ is
    absolutely continuous with respect to $\cM$.
  \item[(ii)] Moreover, under either of the equivalent
    circumstances in (i), we have
  \begin{align}\label{RN:eta:S}
   \frac{d \pdmsp}{d \pdms}(\bq, \Fmap(\bq, \bv))
    = \frac{d \Smap^*\cM}{d \cM}
    (\bq, \bv)
  \end{align}
  for
  $\cM\text{-a.e. } (\bq, \bv) \in \spq \times \spv$
  or, equivalently,
  \begin{align}\label{RN:eta:S:inv:a}
    \frac{d \pdmsp}{d \pdms}(\bq, \btq)
    = \frac{d \Smap^*\cM}{d\cM}(\bq,F(\bq, \cdot)^{-1}(\btq))
  \end{align}
  for $\eta\text{-a.e. } (\bq, \btq) \in \spq \times \spq$.
\item[(iii)] Furthermore, under either of these equivalent absolute
  continuity conditions, the Markov kernels given by
  \eqref{eq:ext:MH:ker:diff} and \eqref{eq:kern:tierney} coincide,
  where $\arFn$ and $\harFn$ are as defined in \eqref{eq:ar:tierney}
  and \eqref{def:arFn:b}, respectively.  These acceptance functions
  $\arFn$ and $\harFn$ maintain the relationships
  \begin{align}\label{RN:eta:S:inv}
    \arFn(\bq, \btq) = \harFn(\bq, F(\bq, \cdot)^{-1}(\btq)), \quad
    \harFn(\bq, \bv) = \arFn(\bq, F(\bq, \bv))
  \end{align}
  for $\bq \in \spq$, $\btq \in \spq$ and $\bv \in \spv$.
\end{itemize}

\begin{Remark}
Reiterating the discussion at the beginning of this subsection, notice that, 
modulo the requirement that $\spq = \spv$, 
we see that \cref{thm:Tierney} reduces to \cref{rmk:flip:it:baby} by taking 
$F(\bq,\bv) = \bv$ and $\Vker = \Pker$.
\end{Remark}

\end{Theorem}
\begin{proof}
  Let $\Bmap: \spq \times \spv \to \spq \times \spq$ and
  $\Cmap: \spq \times \spq \to \spq \times \spq$ be defined as in
  \eqref{eq:tierney:BCmaps}. Then it is not hard to see from
  \eqref{eq:meas:tierney} that
\begin{equation}
    \pdmsp(d\bq,d\btq) = \Cmap^{*} \pdms(d\bq,d\btq)
    \label{eq:jkbs:eta:C}
\end{equation}
and similarly from \eqref{ass:kern:tierney} that
\begin{equation}
    \pdms(d\bq,d\btq) = \Bmap^{*} \cM (d\bq,d\btq)
    \label{eq:jkbs:eta:B}.
\end{equation}
Then combining \eqref{eq:jkbs:S}, \eqref{eq:jkbs:eta:C}, and
\eqref{eq:jkbs:eta:B} yields
\begin{equation}
    \pdmsp(d\bq,d\btq) 
    = (\Cmap\circ\Bmap)^{*} \cM (d\bq,d\btq)
    = (\Bmap\circ\Smap)^{*} \cM (d\bq,d\btq).
    \label{eq:jkbs:eta:BS}
\end{equation}

Let us first show that $\pdmsp \ll \pdms$ implies
$\Smap^* \cM \ll \cM$. Indeed, let $A$ be a measurable subset of
$\spq \times \spv$ such that $\cM(A) = 0$. 
By the assumptions on $F$ it follows that $\Bmap$ is a
one-to-one measurable mapping with measurable inverse.
Then $\Bmap(A)$
is also a Borel set in $\spq \times \spq$.
From \eqref{eq:jkbs:eta:B}, we thus obtain
\[
    0 = \cM(A) = \cM(\Bmap^{-1} \circ \Bmap (A)) = \Bmap^* \cM(\Bmap(A)) = \pdms (\Bmap(A)).
\]
Since $\pdmsp \ll \pdms$, this implies that $\pdmsp(\Bmap(A)) =
0$. With \eqref{eq:jkbs:eta:BS}, we deduce
\[
  0 = \pdmsp(\Bmap(A)) = (\Bmap \circ \Smap)^* \cM(\Bmap(A)) = \Smap^*
  \cM(\Bmap^{-1}\circ \Bmap(A)) = \Smap^* \cM(A).
\]
This shows that $\Smap^* \cM \ll \cM$. The reciprocal claim, that
$\Smap^* \cM \ll \cM$ implies $\pdmsp \ll \pdms$, follows similarly by
invoking \eqref{eq:jkbs:eta:B} and \eqref{eq:jkbs:eta:BS}. This
concludes the proof of item (i).

Now assuming any of the equivalent circumstances in (i), notice that,
for any bounded and measurable function
$\varphi: \spq \times \spq \to \RR$, application of
\eqref{eq:jkbs:eta:BS} and \eqref{eq:jkbs:eta:B} yields
\begin{align*}
    \int\limits_{\spq \times \spv} \! \! \!  \varphi(\bq, \bv) \Smap^* \cM(d\bq,d\bv)
  &= \!  \int\limits_{\spq \times \spq}
    \! \! \!  \varphi(\Bmap^{-1}(\bq, \btq)) (\Bmap \circ \Smap)^* \cM (d\bq,d\btq) 
    = \! \int\limits_{\spq \times \spq}
    \! \! \!  \varphi(\Bmap^{-1}(\bq, \btq)) \pdmsp(d\bq,d\btq) \\
  &= \!  \int\limits_{\spq \times \spq}
    \! \!\!   \varphi(\Bmap^{-1}(\bq, \btq)) \frac{d \pdmsp}{d\pdms} (\bq, \btq)\pdms(d\bq,d\btq) 
    =\!  \int\limits_{\spq \times \spq}
    \! \! \!  \varphi(\Bmap^{-1}(\bq, \btq)) \frac{d \pdmsp}{d\pdms} (\bq, \btq)\Bmap^* \cM (d\bq,d\btq) \\
  &= \! \int\limits_{\spq \times \spv}
    \! \! \!  \varphi(\bq, \bv) \frac{d \pdmsp}{d\pdms} (\Bmap(\bq, \bv)) \cM (d\bq,d\bv), 
\end{align*}
so that
\begin{equation*}
    \frac{d \Smap^*\cM}{d \cM } (\bq, \bv)
    =\frac{d \pdmsp}{d\pdms} (\Bmap(\bq, \bv))
    =\frac{d \pdmsp}{d\pdms} (\bq, \Fmap(\bq,\bv))
\end{equation*}
for $\cM$-a.e. $(\bq, \bv) \in \spq \times \spv$. Clearly, this
implies \eqref{RN:eta:S:inv:a}. Indeed, if
$E \subset \spq \times \spq$ is the set of points
$(\bq, \btq) \in \spq \times \spq$ where \eqref{RN:eta:S:inv:a} does
not hold, then it is not difficult to see that \eqref{RN:eta:S} does
not hold for every $(\bq, \bv) \in \Bmap^{-1}(E)$. But since
\eqref{RN:eta:S} holds $\cM$-a.e., then
$\Bmap^* \cM(E) = \cM(\Bmap^{-1}(E)) = 0$. Hence, from
\eqref{eq:jkbs:eta:B}, $\pdms(E) = 0$, so that \eqref{RN:eta:S:inv:a}
holds for $\eta$-a.e. $(\bq, \btq) \in \spq \times \spq$. Similarly,
\eqref{RN:eta:S:inv:a} implies \eqref{RN:eta:S}, so that these are
indeed equivalent.

Finally, concerning the coincidence of the
kernels $P$ and $\Tker$ in (iii), beginning from \eqref{eq:kern:tierney} and
applying \eqref{ass:kern:tierney} followed by \eqref{RN:eta:S}, we have
\begin{align} 
    \Tker(\bq,d\btq) &= \arFn(\bq,\btq)Q(\bq,d\btq) + \delta_\bq(d\btq)
                   \int_\spq (1-\arFn(\bq,\br))\Pker(\bq,d\br)
  \notag\\
      &= \arFn(\bq,\btq)\Fmap(\bq, \cdot)^* \Vker(\bq,\cdot)(d \btq)
                   + \delta_\bq(d\btq) \int_\spq
        \left(1-\arFn(\bq,\br)\right)\Fmap(\bq, \cdot)^*
                    \Vker(\bq,\cdot)(d \br)
  \notag\\
    &= \int_\spq \delta_\br(d\btq) \arFn(\bq,\br)\Fmap(\bq, \cdot)^*
       \Vker(\bq,\cdot)(d \br)
      + \delta_\bq(d\btq) \int_\spq
      \left(1-\arFn(\bq,\br)\right)\Fmap(\bq, \cdot)^* \Vker(\bq,\cdot)(d \br)
  \notag\\
    &= \int_\spq \delta_{\Fmap(\bq, \bv)}(d\btq) \arFn(\bq,F(\bq,\bv)) \Vker(\bq,d \bv)
      + \delta_\bq(d\btq) \int_\spq
      \left(1-\arFn(\bq,F(\bq,\bv))\right)  \Vker(\bq,d \bv)
  \notag\\
    &= \int_\spq \delta_{\Projq \circ \Smap(\bq, \bv)}(d\btq) \harFn(\bq, \bv) \Vker(\bq, d \bv)
      + \delta_\bq(d\btq) \int_\spq \left(1-\harFn(\bq, \bv)\right)
       \Vker(\bq,d \bv),
      \label{eq:Ker:ext}
\end{align}
which is \eqref{eq:ext:MH:ker:diff}.  The proof is now complete.
\end{proof}

\subsection{An alternative proof of \cref{thm:gen:rev:new}}
\label{subsec:alt:proof}

Turning to our second task in this section we now show how the
formulation in \cite{Tierney1998} can be employed to develop a second
independent proof of \cref{thm:gen:rev:new}.  Here we proceed by
defining an appropriate proposal kernel $\Pker$ on a product space
$\spq \times \spv$, with $(\spq, \Sigma_\spq)$ and
$(\spv, \Sigma_\spv)$ being any measurable spaces, and then taking an
appropriate marginal of the corresponding transition kernel $\Tker$ as
in \eqref{eq:kern:tierney}.

We start with a more general result which is actually independent of
this particular product structure.

\begin{Theorem}\label{thm:tierney:ext}
  Let $(\cX, \Sigma_\cX)$ be a measurable space and let $\cM$ be any
  probability measure on $\cX$. Suppose there exists a measurable
  mapping $S: \cX \to \cX$ satisfying the following properties
\begin{enumerate}[label={(P\arabic*)}]
    \item\label{P1:t} $S$ is an involution, i.e. $S^2 = I$;
    \item\label{P2:t} $S^* \cM$ is absolutely continuous with respect
      to $\cM$.
\end{enumerate}
Consider the Markov kernel $Q$ on $\cX$ defined as
\begin{align}\label{def:Q:tX}
    Q(\bu, d \btu) = \delta_{S(\bu)}(d \btu), \quad \bu \in \cX,
\end{align}
and let $\eta$, $\eta^\perp$ be the measures on $\cX \times \cX$ given
as
\begin{align}\label{def:eta:tX}
    \eta(d \bu, d \btu) = Q(\bu, d\btu) \cM(d \bu), \quad
    \eta^\perp(d \bu, d \btu) = \eta(d \btu, d \bu).
\end{align}
Then $\eta$ and $\eta^\perp$ are mutually absolutely continuous, with
\begin{align}\label{ar:tierney:S:ext}
  \frac{d\eta^\perp}{d\eta}(\bu, \btu) = \frac{d S^*\cM}{d \cM}(\bu)
  \quad \mbox{ for $\eta$-a.e. } (\bu, \btu) \in \cX \times \cX.
\end{align}
Consequently, the Markov kernel $\Tker$ on $\cX$ defined as in
\eqref{eq:kern:tierney} for $Q$, $\arFn$, $\eta$ and $\eta^\perp$ as
in \eqref{def:Q:tX}, \eqref{eq:ar:tierney} and \eqref{def:eta:tX},
respectively, written here as
\begin{align}\label{Tker:tierney:ext}
  \Tker(\bu, d \btu) = \arFn(\bu, \btu) \delta_{S(\bu)}(d \btu)
  + \delta_{\bu}(d \btu) \left[1 - \arFn(\bu, S(\bu)) \right],
\end{align}
satisfies detailed balance with respect to $\cM$.
\end{Theorem}
\begin{proof}
  Let $\varphi: \cX \times \cX \to \RR$ be any bounded and measurable
  function. Recalling the definition of $\eta$ and $\eta^\perp$ from
  \eqref{def:eta:tX}, we have
\begin{align*}
    \int_{\cX \times \cX} \varphi(\bu, \btu) \eta^\perp (d\bu, d \btu)
    &= \int_{\cX \times \cX} \varphi(\bu, \btu) \eta (d\btu, d \bu)
    = \int_{\cX \times \cX} \varphi(\bu, \btu) \delta_{S(\btu)} (d\bu) \cM (d \btu) \\
    &= \int_{\cX} \varphi(S(\btu), \btu) \cM (d \btu) 
\end{align*}
Now invoking properties \ref{P1c} and \ref{P2c} of $S$, we deduce that
\begin{align*}
    \int_{\cX \times \cX} \varphi(\bu, \btu) \eta^\perp (d\bu, d \btu)
    &= \int_{\cX} \varphi(S(\btu), S^2(\btu)) \cM (d \btu) 
    = \int_{\cX} \varphi(\bu, S(\bu)) S^*\cM (d \bu) \\
    &= \int_{\cX} \varphi(\bu, S(\bu)) \frac{d S^*\cM}{d \cM}(\bu) \cM (d \bu) 
      = \int_{\cX \times \cX} \varphi(\bu, \btu)
           \frac{d S^*\cM}{d \cM}(\bu) \delta_{S(\bu)}(d \btu) \cM (d \bu) \\
    &= \int_{\cX \times \cX} \varphi(\bu, \btu) \frac{d S^*\cM}{d\cM}(\bu)
      \eta(d \bu, d \btu).
\end{align*}
This shows that $\eta^\perp \ll \eta$ and that
\eqref{ar:tierney:S:ext} holds. Similarly, we can show that
$\eta \ll \eta^\perp$, so that in fact $\eta$ and $\eta^\perp$ are
mutually absolutely continuous.

The observation that $\Tker$ defined in \eqref{eq:kern:tierney} can be
written as \eqref{Tker:tierney:ext} for $Q$ as in \eqref{def:Q:tX} is
clear. Finally, the fact that $\Tker$ is reversible with respect to
$\cM$ follows as a consequence of the general result given in
\cite[Theorem 2]{Tierney1998}.  The proof is complete.
\end{proof}

Applying \cref{thm:tierney:ext} combined with a suitable
marginalization of the Markov kernel $\Tker$ in
\eqref{Tker:tierney:ext} now yields the result of
\cref{thm:gen:rev:new}.

\begin{Corollary}\label{cor:marg}
  Let $(\spq, \Sigma_\spq)$ and $(\spv, \Sigma_\spv)$ be measurable
  spaces. Let $\mu$ be a probability measure on $\spq$, and
  $\Vker: \spq \times \Sigma_\spv \to [0,1]$ be a Markov
  kernel. Consider $\cX = \spq \times \spv$ and suppose there exists a
  measurable mapping $S: \cX \to \cX$ satisfying properties
  \ref{P1c}-\ref{P2c} from \cref{thm:gen:rev:new} (with $\cM$
  specified as in \eqref{def:ext:target} relative to inputs $\Vker$
  and $\mu$ given here). Let $\Tker$ be the Markov kernel on $\cX$
  defined as in \eqref{Tker:tierney:ext}, and define $P$ to be the
  Markov kernel on $\spq$ given as
\begin{align}\label{Tker:marg}
  P(\bq, d\btq) \coloneqq
  \int_\spv \int_\spv \Tker((\bq, \bv), d \btq \, d \btv) \Vker(\bq, d \bv) 
  \quad \mbox{ for all } \bq \in \spq
  \mbox{ and } A \in \Sigma_\spq.
\end{align}
or equivalently as
$P(\bq, A) \coloneqq \int_\spv \Tker((\bq, \bv), A \times \spv)
\Vker(\bq, d \bv)$ for all $\bq \in \spq$ and $A \in
\Sigma_\spq$. Then $P$ satisfies detailed balance with respect to $\mu$. Moreover,
$P(\bq, \cdot)$ coincides with the definition given in
\eqref{eq:ext:MH:ker:diff} for $\mu$-a.e. $\bq \in \spq$.
\end{Corollary}
\begin{proof}
  The statement that $P$ defined in \eqref{Tker:marg} satisfies detailed balance
  with respect to $\mu$ follows immediately from the fact that $\Tker$
  satisfies detailed balance with respect to $\cM$ defined in
  \eqref{def:ext:target}, as a consequence of
  \cref{thm:tierney:ext}. For the second claim, notice that for each
  fixed $\bq \in \spq$ and $A \in \Sigma_\spq$ we have according to
  the definition of $\Tker$ in \eqref{Tker:tierney:ext} that
\begin{align}\label{Tker:marg:1}
    P(\bq, A) 
    =& \int\limits_\spv \int\limits_{A \times \spv} \arFn((\bq,\bv),(\btq, \btv))
       \delta_{S(\bq, \bv)}(d\btq, d \btv) \Vker(\bq,d\bv) \notag\\
     &+ \int\limits_\spv \int\limits_{A \times \spv}
       \delta_{(\bq, \bv)}(d \btq, d \btv) \left[ 1 - \arFn((\bq, \bv), S(\bq, \bv))\right] \Vker(\bq,d\bv) 
    \eqqcolon \,\, (I_1) + (I_2).
\end{align}
For the first term, we write
\begin{align}
  (I_1) &=
          \int_\spv \int_{\spq \times \spv} \indFn{A\times \spv}(\btq,\btv)
          \arFn((\bq, \bv),(\btq, \btv)) \delta_{S(\bq, \bv)}(d\btq, d \btv) \Vker(\bq,d\bv) \notag\\
    &= \int_\spv \indFn{A\times \spv}(S(\bq, \bv)) \arFn((\bq, \bv),S(\bq, \bv)) \Vker(\bq,d\bv) \notag\\
    &= \int_\spv \int_A \delta_{\Pi_1 \circ S(\bq, \bv)}(d \btq) \arFn((\bq, \bv),S(\bq, \bv)) \Vker(\bq,d\bv),
\end{align}
while for the second term
\begin{align}\label{Tker:marg:2}
  (I_2) = \int_A \delta_\bq(d\btq)
  \int_\spv \left[ 1 - \arFn((\bq, \bv),S(\bq, \bv)) \right] \Vker(\bq,d\bv).
\end{align}
Moreover, it is not difficult to show that \eqref{ar:tierney:S:ext} implies
\[
    \frac{d \eta^\perp}{d \eta}((\bq, \bv), S(\bq, \bv)) = \frac{d S^*\cM}{d \cM}(\bq, \bv) 
    \quad \mbox{ for $\cM$-a.e. } (\bq, \bv) \in \spq \times \spv,
\]
so that 
\begin{align}\label{arFn:marg}
  \arFn((\bq, \bv), S(\bq, \bv)) = \harFn(\bq, \bv)
  \quad \mbox{ for $\cM$-a.e. } (\bq, \bv) \in \spq \times \spq,
\end{align}
for $\harFn$ as defined in \eqref{def:arFn:b}.

Consequently, from \eqref{Tker:marg:1}-\eqref{Tker:marg:2} and
\eqref{arFn:marg}, we conclude that for $\mu$-a.e. $\bq \in \spq$ and
every $A \in \Sigma_\spq$
\begin{align*}
    P(\bq, A) = \int_\spq \int_A \delta_{\Pi_1 \circ S(\bq, \bv)}(d \btq) \harFn(\bq, \bv) \Vker(\bq,d\bv)
    + \int_A \delta_\bq(d\btq) \int_\spq \left[ 1 - \harFn(\bq, \bv) \right] \Vker(\bq,d\bv),
\end{align*}
which coincides with the Markov kernel defined in
\eqref{eq:ext:MH:ker:diff}. This finishes the proof.
\end{proof}

\begin{Remark}
  In regard to the results of \cref{thm:tierney:ext} and
  \cref{cor:marg}, we notice that a slightly different approach to
  constructing a Markov transition kernel on a product space
  $\spq \times \spv$ is provided in \cite[Algorithm
  1]{fang2014compressible} (see also \cite[Section
  5.3]{bou2018geometric}), described as follows. Assume, similarly as
  in \cref{rmk:R:hatS}, that the mapping $\Smap$ from
  \cref{thm:tierney:ext} is given as $R \circ \hat{S}$, with
  $R, \hat{S} : \spq \times \spv \to \spq \times \spv$ being two
  measurable mappings such that $S = R \circ \hat{S}$ satisfies
  assumptions \ref{P1c}-\ref{P2c} from \cref{thm:gen:rev:new} and,
  additionally, that $R^*\cM = \cM$. Then, the Markov transition
  kernel resulting from \cite[Algorithm 1]{fang2014compressible} can
  be written as
\begin{align}\label{def:Tkerp}
  \Tker'(\bu, d \btu) = \arFn'(\bu) \delta_{\hat{S}(\bu)}(d \btu)
  + \delta_{R(\bu)}(d\btu) \left[ 1 - \arFn'(\bu)\right],
\end{align}
where
\begin{align}\label{def:alphap}
    \arFn'(\bu) = 1 \wedge \frac{d (\hat{S}^{-1})^* \cM}{d \cM}(\bu) \quad \mbox{ for } \bu \in \spq \times \spv,
\end{align}
with $\cM$ as defined in \eqref{def:ext:target}. Here assumptions
\ref{P1c}-\ref{P2c} together with $R^*\cM = \cM$ imply that the
Radon-Nikodym derivative in \eqref{def:alphap} is well-defined, since
\[
(\hat{S}^{-1})^*\cM = (S^{-1}\circ R)^* \cM = (S^{-1})^* R^* \cM = (S^{-1})^* \cM = S^* \cM.
\]
In fact, in the setting from \cite{fang2014compressible}, $\spq$ and
$\spv$ are taken as finite-dimensional spaces, and $\arFn'$ is written
in terms of the density of $\cM$ with respect to the corresponding
Lebesgue measure in $\spq \times \spv$.

Most importantly, we notice that in contrast to $\Tker$ defined in
\eqref{Tker:tierney:ext}, the transition kernel $\Tker'$ in
\eqref{def:Tkerp} does \emph{not} in general satisfy detailed balance
with respect to the probability measure $\cM$. Indeed, it is shown in
\cite[Section V]{fang2014compressible} that $\Tker'$ satisfies a
modified form of the standard detailed balance
condition. Nevertheless, if we assume in addition that $R$ is such
that $\Pi_1 \circ R = \Pi_1$ then by taking the marginalization of
$\Tker'$ as in \eqref{Tker:marg} it is not difficult to show that we
obtain a Markov kernel on $\spq$ that also coincides a.e. with the one
from \eqref{eq:ext:MH:ker:diff}, and thus satisfies detailed balance
with respect to the first marginal of $\cM$, i.e. $\mu = \Pi_1^*\cM$.
\end{Remark}

\section{Approximate Hamiltonian Monte Carlo methods}
\label{sec:approx:ham:methods}

This section derives and studies some `extended phase space' sampling
methods.  We first consider the classical finite dimensional setting
involving a continuously distributed target measure in
\cref{sec:ghm:main} leading to \cref{alg:gHMC:FD},
\cref{alg:gHMC:FD:1}, \cref{alg:gHMC:FD:2} and \cref{alg:gHMC:FD:3}.
We then turn to the infinite dimensional Gaussian-Hilbert space
framework in \cref{sec:HMC:inf:dim} culminating in \cref{alg:method}.

The developments presented here are ultimately based on an application
of \cref{thm:gen:rev:new} while drawing on an extended library of HMC
samplers sitting on the foundation of a rich variety of numerical
methods employed for the effective discretization of Hamiltonian
systems.  We provide an essentially self-contained presentation of
some elements of this extensive and disparate literature laying out a
toolbox which can be expanded upon to derive further algorithms in the
future.

On the one hand the results below can be seen as representing a class
of surrogate trajectory methods, reflective of a growing body of
literature; cf. \cite{bernardo1998regression,rasmussen2003gaussian,Li2008,neal2011mcmc,meeds2014gps,lan2016emulation,AFM2017,beskos2017geometric,zhang2017hamiltonian,zhang2017precomputing,li2019neural,radivojevic2020modified}. These works seek to
partially avoid expensive gradient computations dictated by classical
HMC formulations. In this connection it is notable that our methods
allow for a variety of gradient-free approximations where the dynamics
may not be symplectic but are merely volume-preserving (see
\cref{def:sym:Plec} below).  From a slightly different perspective we
may see the results in this section as providing a large `parameter
space' of possible samplers which include many popular and recently
discovered methods as important special cases.  We make explicit the
parameter choices connecting back to existing methods below in
\cref{sec:class:ex}, further illustrating and enriching the unifying
outlook provided by \cref{thm:gen:rev:new}.

\subsection{The finite-dimensional case}
\label{sec:ghm:main}

We begin with the classical setting where the target measure $\mu$
sits on $\RR^N$.  Recall that, in this finite-dimensional context, our
goal is to sample measures $\mu$, which are presumed to be
continuously distributed.  For convenience, we also assume that $\mu$
is strictly positive, but see \cref{rem:noabscont}.  Thus, we may
write $\mu$ in the potential form
\begin{align}
  \mu(d \bq) = \frac{1}{Z_\UPot}e^{-\UPot(\bq)}d \bq, \quad Z_\UPot
  = \int_{\RR^\fd} e^{-\UPot(\bq)}d \bq,
  \label{eq:trg:Msr:fnt:dim}
\end{align}
where $\UPot: \RR^\fd \to \RR$ is any measurable function such that
$e^{-\UPot(\bq)} \in L^1(\RR^\fd)$.  Typically we will additionally
suppose that $\UPot \in C^1$.  

In order to develop our results we need to introduce some elements
from the theory of geometric integration and from Hamiltonian
dynamical systems more broadly. We have tried to keep our discussion
here as elementary and as self-contained as possible but it includes a
number of results which are covered in much more detail and with a
much wider scope elsewhere.  See e.g.
\cite{neal2011mcmc,Hairer2006book,LeimkuhlerReich2004,bou2018geometric} and we refer to
\cite{marsden1995introduction,jose2000classical,arnol2013mathematical} for the broader context of Hamiltonian
systems.

\subsubsection{Hamiltonian and Surrogate Extended Phase Space Dynamics}
The starting point of the HMC approach involves selecting a
Hamiltonian function $\Ham: \RR^{2\fd} \to \RR$ such that the marginal
of the associated Gibbs measure
\begin{align}
\cM(d \bq, d \bv) 
:= \frac{1}{Z_\Ham} e^{-\Ham(\bq, \bv)} d \bq d\bv,
\quad Z_\Ham = \int_{\RR^{2\fd}} e^{-\Ham(\bq, \bv)} d \bq d\bv,
\label{eq:gibbs:Msr}
\end{align}
with respect to the ``position'' variable $\bq$ coincides with the
target measure $\mu$ defined in \eqref{eq:trg:Msr:fnt:dim}. As such,
it is natural to consider a Hamiltonian given by
\begin{align}\label{eq:Ham}
	\Ham(\bq, \bv) = \UPot(\bq) + \VPot(\bq,\bv) + \ln
  Z_\VPot(\bq),
  \quad Z_\VPot(\bq) \coloneqq \int_{\RR^\fd} e^{-\VPot(\bq,\bv)}d \bv
\end{align}
for some $C^1$ function $\VPot: \RR^{2\fd} \to \RR$ such that
$\bv \mapsto e^{-\VPot(\bq,\bv)} \in L^1(\RR^\fd)$, for each fixed
$\bq \in \RR^\fd$. Here the term $\ln Z_{\VPot}(\bq)$ is included
precisely to ensure that the marginal of $\cM$ with respect to $\bq$
coincides with $\mu$. Therefore, $\cM$ from \eqref{eq:gibbs:Msr} can
be written as
\begin{align}\label{eq:smp:Msr:fnt:dim}
  \cM(d\bq, d \bv) = \Vker(\bq, d \bv) \mu(d\bq),
  \quad \Vker(\bq, d \bv) = \frac{1}{Z_\VPot(\bq)} e^{-\VPot(\bq,\bv)}d \bv,
\end{align}
with $Z_{\Ham} = Z_{\UPot}$ from \eqref{eq:trg:Msr:fnt:dim}. Denoting by
$\cB(\RR^\fd)$ the $\sigma$-algebra of Borel sets in $\RR^\fd$, it
follows by construction that
$\Vker: \RR^\fd \times \cB(\RR^\fd) \to [0,1]$ is a Markov kernel, and
$\cM$ defines a probability measure on $\RR^{2\fd}$.

Note that typically one considers
\begin{align}
  \label{eq:typ:sample:pot}
  \VPot(\bq,\bv) = \frac{1}{2}\langle M(\bq)^{-1} \bv, \bv \rangle 
\end{align}
for an appropriately chosen symmetric positive definite ``mass
matrix'' $M$, so that $\VPot(\bq,\bv)$ corresponds to the negative
log-density of the $\RR^{\fd}$-valued gaussian distribution
$N(0,M(\bq))$. Classically, $M$ is $\bq$ independent and often simply
taken to be the identity (but see the infinite dimensional formulation
in \cref{sec:HMC:inf:dim} and \cref{subsec:app:Hilb} below).  On the
other hand, the Riemannian manifold HMC method introduced in
\cite{GirolamiCalderhead2011} considers cases where we introduce a
dependence on $\bq$ in $M$ in \eqref{eq:typ:sample:pot}, thus
providing an important motivating example for allowing
`position-dependence' in the formulation of the kinetic portion of the
Hamiltonian in \eqref{eq:Ham}.  Note also that `non-Gaussian' choices
for $\VPot$ are also relevant.  See, for example, the relativistic HMC
algorithm developed in \cite{lu2017relativistic}.  Both of these HMC
variants are briefly described in \cref{subsubsec:fd:HMC} below where
they are connected back to the generalized frameworks we consider
here.

Having determined $\Ham$ as in \eqref{eq:Ham} one now considers the
associated Hamiltonian dynamics for the pair
$\bqv = (\bq, \bv) \in \RR^{2\fd}$ as
\begin{align}
  \label{eq:Ham:dyn:comp:fm}
  \frac{d \bqv}{dt}  = J^{-1} \nabla \Ham(\bqv),
  \quad \bqv(0) = (\bq_0,\bv_0),
\end{align}
where $J$ is any $2 \fd \times 2 \fd$ real matrix which is
antisymmetric and invertible. Here the typical form for $J$ is
\begin{align}
  \label{eq:cls:J}
  J :=
  \begin{pmatrix}
    0&    -I\\
    I&     0
  \end{pmatrix}
\end{align}
but other ``non-canonical'' choices for $J$ are relevant for
sampling. For example, in the infinite-dimensional version of HMC
derived in \cite{Beskosetal2011}, which we consider in
\cref{sec:HMC:inf:dim}, \cref{subsubsec:pHMC} below, $J$ is used to
``precondition'' the dynamics, see \eqref{eq:non:con:form:J:PC}.
Other possibilities for $J$ are studied for example with the so called
magnetic HMC methods introduced in \cite{tripuraneni2017magnetic}.

Since $\Ham$ is invariant under the flow associated to
\eqref{eq:Ham:dyn:comp:fm}, the Gibbs measure $\cM$ given as
\eqref{eq:gibbs:Msr} is invariant with respect to this flow.  This
implies that the Markov transition kernel associated to the dynamics
of the $\bq$ variable in \eqref{eq:Ham:dyn:comp:fm}, given by
\begin{align}\label{Mkv:ker:exact:fd:HMC}
  P^t(\bq_0,A) \coloneqq \mathbb{P}(\bq(t;\bq_0,\bv_0) \in A),
  \quad \bv_0 \sim \Vker(\bq_0,\cdot),
\end{align}
defined for some $t \geq 0$, for all $\bq_0 \in \RR^{\fd}$ and any
Borel set $A \subset \RR^{\fd}$, holds the $\bq$-marginal of $\cM$,
i.e. the target $\mu$, as an invariant measure.  For a fixed
integration time $T > 0$, $P^T$ as given in
\eqref{Mkv:ker:exact:fd:HMC} defines the Markov kernel for what is
known as the exact HMC algorithm, which is of theoretical interest as
an idealization of HMC, \cite{glatt2020mixing, bou2020two}.  However
this $P^T$ is of much less practical significance because it is
typically impossible to exactly resolve the solution operator for
\eqref{eq:Ham:dyn:comp:fm}. Instead, one resorts to a skillfully
chosen numerical approximation $\hat{S}(\bq_0,\bv_0)$ for the solution
of \eqref{eq:Ham:dyn:comp:fm} at the time $T$ that is commensurate
with the setting of \cref{thm:gen:rev:new}.

In view of obtaining a wider class of HMC-like algorithms to
sample from $\mu$ in \eqref{eq:trg:Msr:fnt:dim}, we replace
\eqref{eq:Ham:dyn:comp:fm} with the following general dynamics:
\begin{align}\label{approx:dyn:fd}
  \frac{d\bq}{dt} = \fu(\bq, \bv), \quad
  \frac{d\bv}{dt} = \fv(\bq, \bv), \quad
  (\bq(0), \bv(0)) = (\bq_0, \bv_0),
\end{align}
for suitably regular functions $\fu: \RR^{2\fd} \to \RR^\fd$ and
$\fv: \RR^{2\fd} \to \RR^\fd$.  Here an underlying idea for
considering the more general dynamic \eqref{approx:dyn:fd} is that we
may aim to replace the right-hand side of \eqref{eq:Ham:dyn:comp:fm},
$J^{-1}\nabla \Ham(\bqv)$, with an artfully chosen approximation
$(\fu(\bq, \bv),\fv(\bq, \bv))$, one that is computationally cheaper
to evaluate while maintaining essential features of
$J^{-1}\nabla \Ham(\bqv)$.  Therefore \eqref{approx:dyn:fd} is the
starting point for a methodology to resolve the target $\mu$ with a
lower overall computational cost.  In this connection notice that, in
contrast to \eqref{eq:Ham:dyn:comp:fm}, the system
\eqref{approx:dyn:fd} for general functions $\fu$ and $\fv$ may not be
a Hamiltonian system, and also may not be expected to hold $\cM$ as an
invariant measure.

Below we illustrate some classes of MCMC algorithms resulting from
such $\fu, \fv$ which still preserve the target measure $\mu$ as
invariant, as long as the integrator $\hat{S}$ and accept-reject
function $\harFn$ are chosen appropriately. In \cref{prop:Ham:res:1}
below, we show how \cref{thm:gen:rev:new} can be used to derive an
appropriate accept-reject mechanism assuming certain natural
structural properties of the map $\hat{S}$. Then, in a series of
subsections, we introduce three classes of algorithms developed around
different considerations for approximating \eqref{eq:Ham:dyn:comp:fm}
with \eqref{approx:dyn:fd}; see \eqref{eq:apx:SD:1},
\eqref{eq:Ham:spl:1} and \eqref{eq:apx:SD:2} below.

\subsubsection{Case 0: General extended phase space methods}
Before introducing our first and most general algorithm class, let us
first recall some basic definitions from the theory of Hamiltonian
dynamical systems which we need here and below.

\begin{Definition}\label{def:sym:Plec}
  Let $\hat{S}: \RR^{2\fd} \to \RR^{2\fd}$ be a $C^1$ map.
  \begin{itemize}
  \item[(i)] Fix any linear invertible map
    $R: \RR^{2\fd} \to \RR^{2\fd}$. We say that $\hat{S}$ is
    \emph{reversible with respect to $R$} (or simply
    $R$-\emph{reversible}) if $\hat{S}$ is itself invertible and
    \begin{align}
      \label{eq:rev:mom:flip:inv}
      R \circ \hat{S} (\bqv) = \hat{S}^{-1} \circ R (\bqv)
    \end{align}
    for every $\bqv \in \RR^{2\fd}$.
  \item[(ii)] We say that $\hat{S}$ is \emph{symplectic}, with respect
    to an invertible (typically antisymmetric) matrix $J$, if
    \begin{align}
      \label{eq:symp:def}
      (\nabla \hat{S}(\bqv))^* J \nabla
      \hat{S}(\bqv) = J
    \end{align} for every $\bqv \in \RR^{2\fd}$, where
    $A^*$ denotes the conjugate transpose of a matrix $A$.\footnote{Equivalently, 
      one may consider the symplectic form $\Omega(\btz, \lbz) := \langle \btz, J \lbz \rangle$,
      for $\btz, \lbz \in \RR^{2 \fd}$
      and assert that $\hat{S}$ is symplectic if $\Omega$ is invariant under the 
      pullback by $\hat{S}$, i.e. we have that $\hat{S}^* \Omega= \Omega$.  Here
      recall that, in this context, $(\hat{S}^*\Omega)_{\bw}(\btz, \lbz)  
      := \Omega(\nabla \hat{S}(\bw)\btz, \nabla \hat{S}(\bw)\lbz)$, for
      $\btz, \lbz, \mathbf{w} \in \RR^{2 \fd}$.  We observe that, in the case where $J$ is 
      canonical, namely when $J$ is of the form \eqref{eq:cls:J}, then we may write 
      $\Omega = d\bq \wedge d \bv := \sum_{j=1}^\fd d q_j \wedge d v_j$
      with $d$ the exterior derivative and $\wedge$ the wedge product so 
      that $dq_j \wedge d v_j(\btz, \lbz) = \tilde{q}_j \bar{v}_j -  \bar{q}_j \tilde{v}_j$
      for $\btz = (\btq, \btv), \lbz = (\lbq, \lbv)$.
      See e.g. \cite{loring2011introduction} for basic definitions.  As identified
      in e.g. \cite[Chapter 4.1]{LeimkuhlerReich2004} we therefore have
      that $\hat{S}$ is symplectic with respect to the canonical
      form $\Omega$ when 
      $d \bq \wedge d\bv = d \lbq \wedge d\lbv$ where
      $(\bq,\bv) = \hat{S} (\lbq,\lbv)$.
    }
  \end{itemize}
\end{Definition}

Let us collect a few elementary properties of symplectic and
$R$-reversible mappings under \cref{def:sym:Plec} whose
proofs are immediate.
\begin{Lemma}
  \label{lem:cons:symp:rev}
  \mbox{}
  \begin{itemize}
  \item[(i)] Under \eqref{eq:rev:mom:flip:inv} it follows that if
    additionally $R$ is an involution, i.e. $R \circ R = I$, then
    $R \circ \hat{S}$ is an involution, i.e.
    \begin{align}
      \label{eq:Mom:flip:form}
      R \circ \hat{S} \circ R \circ \hat{S} (\bqv) = \bqv
    \end{align}
    for every $\bqv \in \RR^{2\fd}$.
  \item[(ii)] If $\hat{S}$ is symplectic, with respect to any
    invertible matrix $J$, then $\hat{S}$ is volume preserving in
    $\RR^{2\fd}$, namely
    \begin{align}
      \label{eq:vol:pres}
      | \det \nabla \hat{S}(\bqv)| = 1
    \end{align}
    for every $\bqv \in \RR^{2\fd}$.
  \item[(iii)] If $\hat{S}_1$, $\hat{S}_2$ are two symplectic
    mappings, then their composition $\hat{S}_1 \circ \hat{S}_2$ is
    also symplectic.  Similarly, under the weaker condition that
    $\hat{S}_1$, $\hat{S}_2$ are both volume preserving, \`a la
    \eqref{eq:vol:pres}, then so too $\hat{S}_1 \circ \hat{S}_2$ must
    be volume preserving.
  \end{itemize}
\end{Lemma}

It is immediately clear that these two properties introduced in
\cref{def:sym:Plec} together with \cref{lem:cons:symp:rev} are
tailor-made for \cref{thm:gen:rev:new}.  See also
\cref{rmk:alg:Rn:Preliminary} and the identity \eqref{eq:Gen:HMC:AR}.
We formalize this as follows.
\begin{Proposition}\label{prop:Ham:res:1}
  Fix any $C^1$ potential functions $\UPot: \RR^\fd \to \RR$ and
  $\VPot: \RR^{2\fd} \to \RR$ so that we can define a probability
  measure $\mu(d\bq) = Z_\UPot^{-1} e^{-\UPot(\bq)} d\bq$ and a Markov
  kernel
  $\Vker(\bq, d \bv) = Z_\VPot(\bq)^{-1}e^{-\VPot(\bq,\bv)} d\bv$ as
  in \eqref{eq:trg:Msr:fnt:dim} and \eqref{eq:smp:Msr:fnt:dim},
  respectively. Consider the associated Hamiltonian
  $\Ham = \UPot + \VPot + \ln Z_\UPot$ as defined in \eqref{eq:Ham}.
  Let $\hat{S} : \RR^{2\fd} \to \RR^{2\fd}$ be a $C^1$ mapping which
  is reversible with respect to a linear involution $R$ as in
  \cref{def:sym:Plec}(i).  Then
  \begin{itemize}
  \item[(i)] the kernel $P$ defined as in \eqref{eq:ext:MH:ker:diff}
    with $S = R \circ \hat{S}$ and with $\harFn$ defined as
    \begin{align}\label{def:arFn:fin:dim}
      \harFn (\bq, \bv)
      = 1 \wedge \left[\exp( - \Ham(R \circ \hat{S}(\bq, \bv))
      + \Ham(\bq,\bv)) |\det \nabla \hat{S}(\bq, \bv)|
      \right]
    \end{align}
    satisfies detailed balance with respect to $\mu$ as in \eqref{eq:rev:mu:P:def}.
  \item[(ii)] If we furthermore assume that $\hat{S}$ is symplectic, a
    la \eqref{eq:symp:def}, or merely volume-preserving as in
    \eqref{eq:vol:pres} then $\harFn$ reduces to
    \begin{align}
	\harFn(\bq, \bv)
	= 1 \wedge \left[\exp( - \Ham(R \circ \hat{S}(\bq, \bv))
	+ \Ham(\bq,\bv))
	\right].
	\label{eq:ar:let:it:flip}
    \end{align}
  \item[(iii)] On the other hand, if we assume that $\Ham$ is
    invariant under $R$, namely
    \begin{align}
	\Ham(\bq,\bv) = \Ham\left( R(\bq,\bv)\right),
	\label{eq:Ham:R:inv}
    \end{align}
    then $\harFn$ becomes
    \begin{align}
	\harFn(\bq, \bv)
	= 1 \wedge \left[\exp( - \Ham( \hat{S}(\bq, \bv))
	+ \Ham(\bq,\bv)) |\det \nabla \hat{S}(\bq, \bv)|
	\right].
	\label{eq:AR:BR:SanSerb}
    \end{align}
  \item[(iv)] Finally if both $\hat{S}$ is volume-preserving and
    \eqref{eq:Ham:R:inv} holds we can take $\harFn$ as
	\begin{align}
	\harFn(\bq, \bv)
	= 1 \wedge \exp( - \Ham( \hat{S}(\bq, \bv)) + \Ham(\bq,\bv)) .
	\label{eq:AR:trad}
	\end{align}
	\end{itemize}
\end{Proposition}
\begin{proof}
  Let $\cM$ be the probability measure on $\RR^{2\fd}$ defined as in
  \eqref{eq:gibbs:Msr}, \eqref{eq:smp:Msr:fnt:dim}, namely
  $\cM(d\bq, d \bv) = \Vker(\bq, d \bv)$
  $\mu(d\bq) = Z_{\Ham}^{-1} e^{- \Ham(\bq, \bv)} d \bq d \bv$. Since
  $\cM$ has a strictly positive density with respect to the Lebesgue
  measure and, by \cref{lem:cons:symp:rev} (i), $S = R \circ \hat{S}$
  is an involution, then the result follows directly from
  \cref{thm:gen:rev:new} once we compute $dS^*\cM/d \cM$.  Here we use
  \eqref{eq:push:forward:den} and \eqref{eq:RN:Der:form} to obtain
  that
  \begin{align}\label{RN:SM:M:fd}
    \frac{d S^* \cM}{d \cM}(\bq, \bv)
    &= \exp \left( - \Ham(S^{-1}(\bq, \bv)) + \Ham(\bq, \bv)\right)
            |\det \nabla S^{-1} (\bq, \bv)| \notag \\
    &= \exp \left( - \Ham(S(\bq, \bv)) + \Ham(\bq, \bv) \right)
      |\det \nabla R(\hat{S}(\bq, \bv))| |\det \nabla \hat{S}(\bq, \bv)|.
  \end{align}
  Since $R$ is a linear involution then clearly
  $|\det \nabla R(\bqv)| = 1$ for any $\bqv \in \RR^{2\fd}$, so that
  \eqref{def:arFn:fin:dim} follows from \eqref{def:arFn:b} and
  \eqref{RN:SM:M:fd}. Now \eqref{eq:ar:let:it:flip} follows
  immediately from \cref{lem:cons:symp:rev} (ii), and the remaining
  identities \eqref{eq:AR:BR:SanSerb} and \eqref{eq:AR:trad} are
  clear, completing the proof.
\end{proof}

We summarize the algorithm resulting from \cref{prop:Ham:res:1} as
follows.
\begin{algorithm}[H]
  \caption{(Generic Extended Phase Space Algorithm to sample from
    $\mu(d\bq) = Z_\UPot^{-1} e^{-\UPot(\bq)} d\bq$)}
  \begin{algorithmic}[1]\label{alg:gHMC:FD}
    \State Select the algorithm parameters:
    \begin{itemize}
    \item[(i)] The momentum kernel
      $\Vker(\bq,d \bv) = Z_\VPot(\bq)^{-1} e^{-\VPot(\bq,\bv)} d \bv
      \in \Pr(\RR^\fd)$, for each fixed $\bq \in \RR^{\fd}$.
    \item[(ii)] The integrator $\Sol$, and the linear involution $R$
      such that $\Sol$ is $R$-reversible (see \eqref{eq:rev:mom:flip:inv},
      \eqref{eq:Mom:flip:form}).
    \end{itemize}
    \State Choose $\bq_0 \in \RR^\fd$.
    \For{$k \geq 0$}
    \State Sample $\bv_k \sim \Vker(\bq_k, \cdot)$.
    \State Propose $\lbq_{k+1} := \Projq \circ R \circ \Sol(\bq_k,\bv_k)$, where
    $\Projq(\bq,\bv) = \bq$.
    \State Set $\bq_{k+1} :=\lbq_{k+1}$ with probability $\harFn(\bq_{k}, \bv_{k})$
    for $\harFn$ given by \eqref{def:arFn:fin:dim}, and otherwise take
    $\bq_{k+1} := \bq_{k}$.
          \EndFor
	\end{algorithmic}
\end{algorithm}

\subsubsection{Algebraic Considerations
  for Numerical Splittings}
With the above formulations in place we next present some approaches
to constructing suitable numerical integrators $\hat{S}$ from the
general system \eqref{approx:dyn:fd} that are commensurate with the
setting of \cref{prop:Ham:res:1}.  Here the reversibility condition
\eqref{eq:rev:mom:flip:inv} leading to \eqref{eq:Mom:flip:form} is
indispensable.  Our starting point for constructing such reversible
integrators out of a well chosen discretization of
\eqref{approx:dyn:fd} are dictated by the following basic algebraic
observations which will be useful here and again further on in
\cref{sec:HMC:inf:dim}:
\begin{Lemma}\label{prop:inv:R:S}
  Let $\cR: \cX \to \cX$ be a mapping on a set $\cX$.
  \begin{itemize}
  \item[(i)] Suppose that $\cSS_j, \bar{\cSS}_j: \cX \to \cX$ are such
    that $\cR \circ \cSS_j = \bar{\cSS}_j \circ \cR$ for
    $j = 1, \ldots n$.  Then, taking $\cSS := \cSS_1 \circ \cdots \circ
    \cSS_n$ and $\bar{\cSS} := \bar{\cSS}_1 \circ \cdots \circ \bar{\cSS}_n$,
    we have that $\cR \circ \cSS = \bar{\cSS} \circ \cR$.
  \item[(ii)] If $\cS_j: \cX \to \cX$, $j = 1, \ldots, n$, are
    invertible mappings satisfying
    $\cR \circ \cS_j = \cS_j^{-1} \circ \cR$, then the ``palindromic''
    composition $\cS$ defined as
    \begin{align}\label{eq:palandromic}
      \cS \coloneqq
      \cS_1 \circ \cS_2 \circ \cdots \circ \cS_{n-1}
      \circ \cS_n \circ \cS_{n-1} \circ \cdots \circ \cS_2 \circ \cS_1
    \end{align}
    satisfies $\cR \circ \cS = \cS^{-1} \circ \cR$.
	\end{itemize}
\end{Lemma}
\begin{proof}
  The first item, (i), is obvious.  For (ii) notice that, due to the
  palindromic structure of $\cS$, its inverse $\cS^{-1}$ maintains the
  same palindromic structure, with
  \begin{align*}
    \cS_1^{-1} \circ \cS_2^{-1} \circ \cdots \circ
    \cS_{n-1}^{-1} \circ \cS_n^{-1} \circ \cS_{n-1}^{-1}
    \circ \cdots \circ \cS_2^{-1} \circ \cS_1^{-1} = \cS^{-1},
  \end{align*}
  so that $\cR \circ \cS = \cS^{-1} \circ \cR$ follows directly from item (i).
\end{proof}

Next let us introduce some definitions.
\begin{Definition}\label{def:adj:num:op}
  Fix $\dT_0 > 0$ and suppose that for each $\dT \in (-\dT_0, \dT_0)$
  we have an invertible map $\chS_\dT: \RR^{2\fd} \to \RR^{2 \fd}$.
  \begin{itemize}
  \item[(i)] We define the \emph{adjoint} of $\chS_\dT$, denoted as
    $\chS_\dT^*$, according to $\chS_{\dT}^* \coloneqq (\chS_{-\dT})^{-1}$.
  \item[(ii)] We say $\chS_\dT$ is \emph{symmetric} (or
     self-adjoint) if $\chS_{\dT}^* = \chS_{\dT}$.
  \end{itemize}
\end{Definition}
\noindent The following desirable properties around this adjoint
operation are immediate:
\begin{Lemma}\label{lem:adj:props}
  \mbox{}
  \begin{itemize}
    \item[(i)]   Given maps collections of invertible maps $\chS, \chT$ as in
      \cref{def:adj:num:op} we have $\chS^{**} = \chS$ and also
      $(\chS\chT)^* = \chT^* \chS^*$.
      In particular $\chS^* \chS$ is symmetric.
    \item[(ii)] If $\chS_\dT$ is symmetric for some $\dT \in \RR$
      then $\chS_\dT^{-1} = \chS_{-\dT}$.
  \end{itemize}
\end{Lemma}

We now have everything in hand to implement algorithms around the
dynamics \eqref{approx:dyn:fd} discretized in a suitable form to apply
\cref{prop:Ham:res:1} in three specific cases.

\subsubsection{Case 1: Separable Surrogate Dynamics}
\label{ssec:sep:dyn}

One particular case of \eqref{approx:dyn:fd} we consider is when $\fu$
depends only on $\bv$ and $\fv$ depends only on $\bq$. Namely, in this
case, \eqref{approx:dyn:fd} reduces to
\begin{align}\label{sep:dyn}
	\frac{d\bq}{dt} = \fu(\bv), 
	\quad \frac{d\bv}{dt} = \fv(\bq),
	\quad  (\bq(0), \bv(0)) = (\bq_0, \bv_0),
\end{align}
for some $C^1$ functions $\fu, \fv: \RR^{\fd} \to \RR^\fd$.    Here we have 
in mind the situation where our Hamiltonian $\Ham$ and matrix
$J$ in \eqref{eq:Ham:dyn:comp:fm} in have the separable form
\begin{align}\label{Ham:sep:fd}
	\Ham (\bq,\bv) := \UPot(\bq) + \VPot(\bv),
       \qquad \text{ and } \quad
       J \coloneqq 
	\begin{pmatrix}
		0 & - A \\
		A^* & 0
	\end{pmatrix},
\end{align}
for some invertible matrix $A \in \RR^{\fd \times \fd}$, so 
that $\fu$ and $\fv$ would serve as suitable approximations
of the form
\begin{align}
	\label{eq:apx:SD:1}
	\fu(\bv) \approx (A^*)^{-1} \nabla \VPot(\bv),  
	\qquad  
	\fv(\bq) \approx -A^{-1}  \nabla \UPot(\bq).
\end{align}

In this case, \eqref{sep:dyn} we may consider the classical
\emph{leapfrog} integrator.  This scheme is defined by splitting the
dynamics \eqref{sep:dyn} into
\begin{align}\label{split:dyn:fd}
	\frac{d\bq}{dt} = 0, \quad \frac{d\bv}{dt} = \fv(\bq), \quad 
	\mbox{ and } \quad \frac{d\bq}{dt} = \fu(\bv), \quad \frac{d\bv}{dt} =0,
\end{align}
with analytical solutions given explicitly for any time $t$ and 
initial datum $(\bq_0, \bv_0) \in  \RR^{2\fd}$ as 
\begin{align}\label{def:Xi1:Xi2:fd}
	\Xi_t^{(1)}(\bq_0,\bv_0) = (\bq_0, \bv_0 + t \fv(\bq_0)), 
	\quad \mbox{ and } \quad 
	\Xi_t^{(2)}(\bq_0, \bv_0) = (\bq_0 + t \fu(\bv_0), \bv_0),
\end{align}
respectively. Given time steps $\da, \db > 0$ and a number of
iterations $n \in \NN$, a leapfrog-type integrator for \eqref{sep:dyn}
is then defined according to the following Strang splitting
\begin{align}\label{def:leapfrog:int}
  \cSol(\bq_0, \bv_0) 
  = \cSol_{n,\da, \db} (\bq_0, \bv_0) 
  \coloneqq \left(  \Xi^{(1)}_{\da} \circ \Xi^{(2)}_{\db}
                   \circ \Xi^{(1)}_{\da} \right)^n (\bq_0, \bv_0),
  \qquad \mbox{ for all } (\bq_0, \bv_0) \in \RR^{2\fd}.
\end{align}

Under minimal condition on $\fu, \fv$ we have the following result
placing the map \eqref{def:leapfrog:int} defined from \eqref{sep:dyn}
in the setting of \cref{prop:Ham:res:1}:

\begin{Theorem}\label{cor:sep:dyn}
  Consider the dynamics in \eqref{split:dyn:fd}, resolved as
  \eqref{def:Xi1:Xi2:fd}, for any given $C^1$ functions
  $\fu,\fv: \RR^{\fd} \to \RR^{\fd}$.  For any time steps
  $\da, \db > 0$ and any number of iterations $n \in \NN$, we take
  $\cSol_{n,\da, \db}$ to be the integrator defined in
  \eqref{def:leapfrog:int}.  Then
\begin{itemize}
\item[(i)] $\cSol_{n,\da, \db}$ is volume-preserving, i.e.
  $|\det \nabla \cSol_{n,\da, \db}(\bz)| = 1$ for all
  $\bz \in \RR^{2\fd}$.
\item[(ii)]  If we furthermore 
assume that
\begin{align}
	\label{eq:Mom:flp:rv:smp}
	\fu(-\bv) = - \fu(\bv) \qquad \text{ for every } \bv \in \RR^\fd
\end{align}
then $\cSol_{n,\da, \db}$ is reversible with respect to the momentum
flip involution $R$ in the sense of \cref{def:sym:Plec},
\eqref{eq:rev:mom:flip:inv}.  Here $R$ is given by
\begin{align}\label{def:flip:R}
      R(\bq, \bv) = (\bq, - \bv), \quad
      (\bq, \bv) \in
 \RR^{2\fd}.
    \end{align}
\end{itemize}
\end{Theorem}

\begin{proof}
  Start with the first item, (i).  From \eqref{def:Xi1:Xi2:fd} it
  follows that
  \begin{align*}
    \nabla \Xi^{(1)}_t(\bq, \bv) =
    \begin{pmatrix}
      I & 0 \\
      t \nabla \fv(\bq) & I
    \end{pmatrix},
     \quad \mbox{ and }
     \quad \nabla \Xi^{(2)}_t(\bq,\bv) =
	\begin{pmatrix}
		I & t \nabla \fu(\bv) \\
		0 & I
	\end{pmatrix},
\end{align*}
so that, clearly,
$|\det \nabla \Xi^{(1)}_t(\bz)| = |\det \nabla \Xi^{(2)}_t(\bz)| = 1$
for every $\bz \in \RR^{2\fd}$ and each $t \geq 0$.  From
\eqref{def:leapfrog:int} it thus follows that
$|\det \nabla \cSol_{n,\da, \db}(\bz)| = 1$ for all $\bz \in \RR^{2\fd}$.
Regarding the reversibility claim in item (ii) we observe that
$\Xi^{(j)}$ is symmetric in the sense of \cref{def:adj:num:op}, (ii)
for $j = 1,2$.  On the other hand it is also direct to check that
$R \Xi^{(j)}_t = \Xi^{(j)}_{-t} R$ for any $t \geq 0$ and $j =1,2$.
Here note the use of the condition \eqref{eq:Mom:flp:rv:smp} for
$\Xi^{(2)}$.  Thus with \cref{lem:adj:props}, (ii) and
\cref{prop:inv:R:S}, (ii), noting the palindromic structure in
\eqref{def:leapfrog:int}, we obtain the desired reversibility claim for
$\cSol_{n,\da, \db}$, completing the proof.
\end{proof}

We summarize and present the algorithm directly resulting from
\cref{cor:sep:dyn} and \cref{prop:Ham:res:1}, (ii).  As we make
explicit in \cref{sec:con:class:alg:fd} below this algorithm includes
the classical Random Walk Monte Carlo (RWMC), Metropolis Adjusted
Langevin (MALA), and Hamiltonian Monte Carlo (HMC) algorithms as
special cases under suitable parameter choices.

\begin{algorithm}[H]
  \caption{(Generalized Leapfrog for surrogate dynamics to sample
    $\mu(d\bq) = Z_\UPot^{-1} e^{-\UPot(\bq)} d\bq$)}
\begin{algorithmic}[1]\label{alg:gHMC:FD:1}
  \State Select the algorithm parameters:
  \begin{itemize}
  \item[(i)] The proposal distribution
     $\Vker(\bq, d \bv) = Z_\VPot(\bq)^{-1} e^{-\VPot(\bq,\bv)} d \bv \in
    \Pr(\RR^\fd)$ for each $\bq \in \RR^\fd$.
  \item[(ii)] The surrogate functions $\fu, \fv: \RR^\fd \to \RR^\fd$
    defining $\Xi_t^{(1)}(\bq,\bv) := (\bq, \bv + t \fv(\bq))$, and
    $\Xi_t^{(2)}(\bq, \bv) := (\bq + t \fu(\bv), \bv)$, where $\fu$
    satisfies $\fu(-\bv) = -\fu(\bv)$ for all $\bv \in \RR^{\fd}$.
  \item[(iii)] The time step sizes $\delta_1, \delta_2 > 0$.
  \item[(iv)] The number of iterations $n$.
  \end{itemize}
   \State Choose $\bq_0 \in \RR^\fd$.
    \For{$k \geq 0$}
        \State Sample $\bv_k \sim \Vker(\bq_k, d \bv)$.
        \State Propose
        $(\lbq_{k+1}, \lbv_{k+1}) :=
        \left(  \Xi^{(1)}_{\da} \circ \Xi^{(2)}_{\db} \circ \Xi^{(1)}_{\da} \right)^n (\bq_k, \bv_k)$.
	\State Set $\bq_{k+1} :=\lbq_{k+1}$ with probability
        $1 \wedge \left[\exp( - \Ham(\lbq_{k+1}, -\lbv_{k+1})
          + \Ham(\bq_k,\bv_k)) \right]$,
        otherwise take $\bq_{k+1} := \bq_{k}$.
    \EndFor
\end{algorithmic}
\end{algorithm}

\begin{Remark}
  \cref{alg:gHMC:FD:1} derived from \eqref{sep:dyn} can be further
  generalized in several directions.  For example
  \eqref{def:leapfrog:int} can be replaced with any palindromic
  splitting involving the maps $\Xi^{(1)}, \Xi^{(2)}$ defined in  \eqref{def:Xi1:Xi2:fd}.  
  The reversibility condition \eqref{eq:Mom:flp:rv:smp} as suitable for the momentum
  flip operation $R$ in \eqref{def:flip:R} can be replaced with a more
  general condition \eqref{cond:fu:fv} as we explore in
  \cref{prop:HMC:Leap:Frog} below.
\end{Remark}

\subsubsection{Case 2: Splitting into Hamiltonian Sub-dynamics}
\label{ssec:ham:sub:dyn}

A second case of interest arises when we consider a surrogate of the
form
\begin{align}\label{eq:Ham:split}
	J^{-1} \nabla \Ham(\bqv) \approx \sum_{j=1}^m J^{-1}_j \nabla \Ham_j(\bqv)
\end{align}
for \eqref{eq:Ham:dyn:comp:fm} where each $J_j$ is
antisymmetric and invertible and each $\Ham_j: \RR^\fd \to \RR$ is
suitably smooth.  Corresponding to each element in this sum,
\eqref{eq:Ham:split}, we consider mappings $\Xi^{(j)}$ where
\begin{align}\label{eq:Ham:spl:1}
	\{\Xi_t^{(j)}\}_{t \geq 0} 
	\; \text{ is the solution map for } \;
	\frac{d \bqv}{dt} = J^{-1}_j \nabla \Ham_j(\bqv),
\end{align}
defined so long as each of the associated Hamiltonian systems admits a
globally defined dynamic.  The idea here is that we might formulate
the approximation in \eqref{eq:Ham:split} so that each
$\{\Xi_t^{(j)}\}_{t \geq 0}$ has an explicitly solvable form while in
any case preserving crucial structural properties of Hamiltonian
systems commensurate with the setting of \cref{prop:Ham:res:1}.  As
such, \eqref{eq:Ham:split} would suggest that a suitable composition
of the maps $\Xi^{(j)}$ would yield a reasonable approximation for
\eqref{eq:Ham:dyn:comp:fm}.  Keeping in mind \cref{prop:inv:R:S}, we
select for some $l \in \NN$, $n \in \NN$, any
$j_1, \ldots, j_l \in \{1, \ldots, m\}$, any
$\delta_1, \ldots, \delta_l > 0$, and define
\begin{align}\label{eq:ham:spt:gen:alg}
 \cSol(\bq_0, \bv_0)  \coloneqq 
 	\left(  \Xi^{(j_1)}_{\delta_1} \circ \cdots \circ \Xi^{(j_l)}_{\delta_l} 
  \circ  \Xi^{(j_l)}_{\delta_l} \circ  \cdots \circ\Xi^{(j_1)}_{\delta_1}
  \right)^n (\bq_0, \bv_0),
  \quad \mbox{ for all } (\bq_0, \bv_0) \in \RR^{2\fd}.
\end{align}

Before providing conditions under which \eqref{eq:ham:spt:gen:alg}
yields a suitable class of sampling algorithms we first recall some
basic properties of Hamiltonian dynamical systems that will be needed.
\begin{Proposition}\label{prop:Ham:props:gen}
  Suppose that $\{\Xi_t\}_{t \geq 0}$ is the solution operator of
  \begin{align}
    \label{eq:Ham:gen:split}
    \frac{d\bqv}{dt}  = J^{-1} \nabla \tilde{\Ham}(\bqv)
  \end{align}
  for a matrix $J$ which is antisymmetric and invertible and a $C^2$
  Hamiltonian function $\tilde{\Ham}: \RR^{2\fd} \to \RR$ such that
  the dynamics \eqref{eq:Ham:gen:split} are uniquely and globally defined.
  Then,
  \begin{itemize}
  \item[(i)]   for every $t \geq 0$, $\Xi_t$ is symplectic (with
    respect to $J$) as in
    \cref{def:sym:Plec}, (ii).
  \item[(ii)] Suppose we have a linear involution $R: \RR^{2\fd} \to \RR^{2\fd}$
    such that
    \begin{align}
      \label{eq:r:j:sym:cond}
      R J^{-1}R^* = - J^{-1}
    \end{align}
    and
    \begin{align}
      \label{eq:H:r:sym:cond}
      \tilde{\Ham}(R \bqv) = \tilde{\Ham}(\bqv)
      \quad \text{ for every }\bqv \in \RR^{2\fd}.
    \end{align}
    Then, for every $t \geq 0$, $\Xi_t$ is reversible with respect to
    $R$ in the sense of \cref{def:sym:Plec}(i).  
  \item[(iii)] In particular, consider
  \begin{align*}
    J =
    \begin{pmatrix}
      0& -A\\
      A^*&  0
    \end{pmatrix},
  \end{align*}
  for an invertible $A$ and we assume that $\tilde{\Ham}$ is symmetric
  in its second variable, namely
  $\tilde{\Ham}(\bq, \bv) = \tilde{\Ham}(\bq, -\bv)$, for any
  $(\bq, \bv) \in \RR^{2\fd}$.  Then, for every $t \geq 0$, $\Xi_t$ is
  reversible with respect to the momentum flip involution $R$ as in
  \eqref{def:flip:R}.
  \end{itemize}
\end{Proposition}
\begin{proof}
  For the first item let
  $B(t)(\bqv) := (\nabla \Xi_t(\bqv))^* J \nabla \Xi_t(\bqv)$ and
  notice that $B(0)(\bqv) = J \bqv$ while $d B(\bqv)/ dt = 0$.

  Regarding the second item, given any solution $\bqv (t)$ of
  \eqref{eq:Ham:gen:split} we consider
  $\tilde{\bqv}(t) := R \bqv (-t)$.  Observe that,
  with the assumption that $R$ is an involution,
  \begin{align*}
    \frac{d \tilde{\bqv}(t)}{dt} = - R J^{-1} \nabla \tilde{\Ham}(\bqv(-t))
                                    =- R J^{-1} \nabla \tilde{\Ham}(R \tilde{\bqv}(t))
  \end{align*}
  Now, from \eqref{eq:H:r:sym:cond}, we have
  \begin{align*}
    R^* \nabla \tilde{\Ham}(R \bar{\bqv}) = \nabla \tilde{\Ham} (\bar{\bqv})
    \quad \text{ for every } \bar{\bqv} \in \RR^{2\fd}
  \end{align*}
  and so with the fact that $R^*$ is an involution and our assumption
  \eqref{eq:r:j:sym:cond} we conclude that $\tilde{\bqv}(t)$ must also
  obey \eqref{eq:Ham:gen:split}. From this symmetry observation and
  the uniqueness of solutions of \eqref{eq:Ham:gen:split}, we infer
  \begin{align*}
     \Xi_t( R\bqv_0) =  R\Xi_{-t}( \bqv_0) \quad \text{ for any } t
    \in \RR, \text{ and any } \bqv_0 \in \RR^{2\fd}.
  \end{align*}
  Thus
  \begin{align*}
   ( R \circ \Xi_t \circ   R \circ \Xi_t)( \bqv_0)
    =R \Xi_t (  R \Xi_t(\bqv_0))
    = \Xi_{-t}(\Xi_t(\bqv_0)) = \bqv_0
  \end{align*}
  as desired for the second item.  The third item follows from the
  second with a direct computation showing that
  \eqref{eq:r:j:sym:cond} holds for these specific choices of $J$ and $R$.
  The proof is complete.
\end{proof}
\begin{Remark}
  See \cite[Lemma 2]{tripuraneni2017magnetic} for variations on the
  theme of \cref{prop:Ham:props:gen}, (ii) with some interesting
  applications in deriving further HMC-type sampling algorithms.
\end{Remark}

\begin{Theorem}\label{thm:Ham:sur:splitting}
  Fix a collection of $C^2$ functions
  $\Ham_j : \RR^{2\fd} \to \RR$ along with corresponding
  antisymmetric and invertible $2\fd \times 2 \fd$ real matrices $J_j$
  for $j = 1, \ldots, m$.  Assume the Hamiltonian dynamics associated
  to each of these pairs, \`a la \eqref{eq:Ham:gen:split}, are uniquely
  and globally defined.  Let $\{\Xi^{(j)}_t\}_{t \in \RR}$ be the
  associated solution maps as in \eqref{eq:Ham:spl:1}.  Furthermore
  suppose that there exists a linear involution $R$ such that
  \begin{align}\label{eq:mut:rev:cond}
    R J^{-1}_jR^* = - J^{-1}_j,   \text{ and } \,
    \Ham_j \circ R = \Ham_j,\, \text{ for each }
    j = 1, \ldots, m.
  \end{align}
  Select any $l \in \NN$, $n \in \NN$, along with orderings
  $j_1, \ldots, j_l \in \{1, \ldots, m\}$, and time step sizes
  $\delta_1, \ldots, \delta_l > 0$, and define $\cSol$ as in
  \eqref{eq:ham:spt:gen:alg}.  Then, according to this definition
  \eqref{eq:ham:spt:gen:alg}, $\cSol$ is both symplectic and
  reversible with respect to $R$ in the sense given in \cref{def:sym:Plec}.
\end{Theorem}

\begin{proof}
  According to \cref{prop:Ham:props:gen} each $\Xi^{(j_i)}$ is
  symplectic and is reversible with respect to $R$ for
  $i = 1, \ldots, m$.  Keeping in mind the palindromic structure of
  $\cSol$ in regards to reversibility, it is therefore clear that
  these two properties extend to $\cSol$.  The proof is complete.
\end{proof}

We now summarize our second class of sampling methods which are
derived from \cref{thm:Ham:sur:splitting} with \cref{prop:Ham:res:1}.
As previously with \cref{alg:gHMC:FD:1}, this class includes classical
formulations of RWMC, MALA and HMC as notable special cases in a
fashion which we make precise in \cref{sec:con:class:alg:fd} below.
Note also that a special case of \eqref{eq:ham:spt:gen:alg} yields the
type of `non-standard' splittings which proves to be desirable for the
setting of \eqref{ham:dyn} in \cref{sec:HMC:inf:dim},
\cref{subsubsec:pHMC} below.

\begin{algorithm}[H]
  \caption{(Palindromic Iterations of Hamiltonian Surrogates to sample
    $\mu(d\bq) = Z_\UPot^{-1} e^{-\UPot(\bq)} d\bq$)}
\begin{algorithmic}[1]\label{alg:gHMC:FD:2}
  \State Select the algorithm parameters:
  \begin{itemize}
  \item[(i)] The proposal distribution
    $\Vker(\bq, d \bv) = Z_{\VPot}(\bq)^{-1} e^{-\VPot(\bq,\bv)} d \bv \in \Pr(\RR^\fd)$
    for each $\bq \in \RR^\fd$.
  \item[(ii)] For some $j =1, \ldots, m$ determine a set of surrogate
    $\Ham_j: \RR^{2\fd} \to \RR$ and associated symplectic
    matrices $J_j$ yielding solution maps $\{\Xi^{(j)}\}_{t \in \RR}$
    defined according to $d \bqv /dt = J_j^{-1} \nabla \Ham_j(\bqv)$.
  \item[(iii)] Identify a linear involution $R$ such that
    $R J^{-1}_j R^* = - J^{-1}_j$ and $\Ham_j \circ R = \Ham_j$ for
    $j =1, \ldots, m$.
  \item[(iv)] An operation ordering
    $j_1, \ldots, j_l \in \{ 1, \ldots, m\}$ along with associated
    time step sizes $\delta_1, \delta_2, \ldots, \delta_l > 0$.
  \item[(v)] The number of iterations $n$.
  \end{itemize}
   \State Choose $\bq_0 \in \RR^\fd$.
    \For{$k \geq 0$}
        \State Sample $\bv_k \sim \Vker(\bq_k, d \bv)$.
        \State Propose 
        $(\lbq_{k+1}, \lbv_{k+1}) :=
        \left(  \Xi^{(j_1)}_{\delta_1} \circ \cdots \circ \Xi^{(j_l)}_{\delta_l} 
          \circ  \Xi^{(j_l)}_{\delta_l} \circ  \cdots \circ\Xi^{(j_1)}_{\delta_1}
           \right)^n (\bq_k, \bv_k)$.
      \State Set $\bq_{k+1} :=\lbq_{k+1}$ with probability
         $1 \wedge \left[\exp( - \Ham(R (\lbq_{k+1}, \lbv_{k+1}))
        + \Ham(\bq_k,\bv_k)) \right]$
      and otherwise take $\bq_{k+1} := \bq_{k}$.
    \EndFor
\end{algorithmic}
\end{algorithm}

\subsubsection{Case 3: Non-Separable Dynamics Via Implicit Integrators} 
\label{ssec:imp:int}

We turn to our final case where \eqref{approx:dyn:fd} is assumed to
have a `non-separable' form unsuitable for either of the formulations
previously considered in \cref{ssec:sep:dyn} or
\cref{ssec:ham:sub:dyn}.  This situation arises for example from the
consideration of position dependent kinetic energy terms in
\eqref{eq:Ham} as developed previously in
\cite{GirolamiCalderhead2011}.  Here, in this position dependent case,
we have
\begin{align}
  \label{eq:apx:SD:2}
  \fu(\bq,\bv) \approx (A^*)^{-1} \nabla_\bv \VPot(\bq, \bv),  
  \qquad  
  \fv(\bq,\bv) \approx -A^{-1}   \left( 
  \nabla_\bq (  \UPot(\bq) + \VPot(\bq,\bv) )  
  + Z_\VPot(\bq)^{-1} \nabla_\bq Z_\VPot(\bq)\right),
\end{align}
when, as above, \eqref{eq:Ham:dyn:comp:fm} is defined with
a matrix $J$ of the separated form given in \eqref{Ham:sep:fd}.
In this non-separable case, implicit numerical discretizations for
\eqref{approx:dyn:fd} provides a means of maintaining indispensable
structural properties, namely the reversibility and volume
preservation conditions, required by \cref{prop:Ham:res:1}.  Following
ideas from \cite{LeimkuhlerReich2004, Hairer2006book,
  GirolamiCalderhead2011}, we consider schemes starting from the
so-called \emph{Euler-B} and \emph{Euler-A} methods applied to
\eqref{approx:dyn:fd}.

Start with the \emph{Euler-B} scheme which is defined, 
for a fixed time step $\dT > 0$, number of iterations $n \in \NN$ 
and any initial point
$(\bq_0, \bv_0) \in \RR^{2\fd}$, by
$(\bq_n, \bv_n) := (\SB_{\dT})^n(\bq_0,\bv_0)$. 
Here each iteration step $(\bq, \bv) := \SB_{\dT}(\lbq,\lbv)$ 
is specified, for a given $\lbq,\lbv \in \RR^{\fd}$ 
through the following implicit system of
equations
\begin{align}\label{def:EulerB}
  \bq = \lbq + \dT \fu(\lbq, \bv),
   \qquad \bv = \lbv + \dT \fv(\lbq, \bv).
\end{align}
It is not difficult to check that the
adjoint of $\SB_{\dT}$, cf. \cref{def:adj:num:op}, is the
\emph{Euler-A method} which is defined implicitly for
a single step as $(\bq, \bv) := \SA_{\dT}(\lbq,\lbv)$
where
\begin{align}\label{def:EulerA}
  \bq = \lbq + \dT \fu(\bq, \lbv), \qquad
  \bv = \lbv + \dT \fv(\bq, \lbv),
\end{align}
for given $\lbq,\lbv \in \RR^{\fd}$.  Balancing these two schemes,
\eqref{def:EulerB} and \eqref{def:EulerA}, leads to the consideration
of a symmetric integrator called the \emph{generalized
  St\"ormer-Verlet} method.  Relative to the parameters $\dT > 0$ and
$n \geq 1$ we define
$\hat{S} = \hat{S}_{n,\dT} \coloneqq (\SA_{\dT/2} \circ
\SB_{\dT/2})^n$ so that, for any given $(\bq_0,\bv_0) \in \RR^{2\fd}$,
$(\bq_n, \bv_n) :=\hat{S}_{n,\dT}(\bq_0,\bv_0)$ is computed
inductively according to
\begin{align}\label{def:Sol:fd}
  &\bv_{m + 1/2} = \bv_m + \frac{\dT}{2} \fv(\bq_m, \bv_{m+1/2})
    \notag\\
  & \bq_{m+1} = \bq_m + \frac{\dT}{2}
       \left[ \fu(\bq_m, \bv_{m+1/2}) + \fu(\bq_{m+1}, \bv_{m+1/2}) \right]  \\
  & \bv_{m+1} = \bv_{m+1/2} + \frac{\dT}{2} \fv(\bq_{m+1}, \bv_{m+1/2}) \notag
\end{align}
for $m = 0, \ldots, n-1$.  Note carefully that \eqref{def:Sol:fd}
reduces to \eqref{def:leapfrog:int} in the special case when $\fv$
depends only on $\bq$, $\fu$ only on $\bv$, and we take
$\da =\dT/2, \db = \dT$. See also \cref{rmk:rel:alg} below for
further commentary around this point.

We next establish conditions on $\fu$ and $\fv$ which yield
desirable reversibility and volume-preservation properties for the
scheme \eqref{def:Sol:fd}. In the statement below,
$\Pi_1, \Pi_2: \RR^{2\fd} \to \RR^{2\fd}$ denote the projections onto the
first and second components, respectively, i.e.
\begin{align*}
  \Pi_1(\bq, \bv) = \bq, \quad \Pi_2(\bq, \bv) = \bv
  \quad \mbox{ for all } (\bq, \bv) \in \RR^{2\fd}.
\end{align*}

\begin{Theorem}\label{prop:HMC:Leap:Frog}
  Suppose that $\fu, \fv: \RR^{2\fd} \to \RR^\fd$ are $C^1$ functions
  and assume that, for some $\delta > 0$, the maps 
  $\SB_{\dT/2}, \SA_{\dT/2}$ specified implicitly from
  \eqref{def:EulerB}, \eqref{def:EulerA} are uniquely and globally defined
  and are $C^1$; cf.
  \cref{rmk:ex:unq:imp} below.  For any $n \geq 1$ 
  we consider the generalized
  St\"ormer-Verlet implicit integration scheme
  $\hat{S} = \hat{S}_{n,\dT} := (\SA_{\dT/2} \circ \SB_{\dT/2})^n$
  as given by \eqref{def:Sol:fd}.
  \begin{itemize}
  \item[(i)] Suppose that, for some linear invertible matrix $R$, the
    mapping $\f = (\fu, \fv)$ satisfies
  \begin{align}\label{cond:fu:fv}
    R \f(\Pi_1 \bz, \Pi_2 \tilde{\bz})
    = - \f (\Pi_1 \circ R \bz, \Pi_2 \circ R \tilde{\bz})
    \quad \mbox{ for all } \bz, \tilde{\bz} \in \RR^{2\fd}.
    \end{align}
    Then $\Sol$ is reversible with respect to $R$, in the sense of
    \cref{def:sym:Plec}, (i).  In particular, if $\fu$ and $\fv$ maintain
    \begin{align}\label{cond:fu:fv:flip}
      \fu(\bq, \bv) = - \fu(\bq, -\bv),
      \quad \fv(\bq, \bv) = \fv(\bq, -\bv)
      \quad \mbox{ for all } (\bq, \bv) \in \RR^{2\fd},
    \end{align}
    then $\Sol$ is reversible with respect to the `momentum-flip
    involution operation' \eqref{def:flip:R}.
    \item[(ii)] Now suppose that $\fu, \fv: \RR^{2\fd} \to \RR^{\fd}$ are such that
    \begin{align}\label{eq:symp:sur}
    \nabla_{\bq} \fu (\bq,\bv) + (\nabla_\bv \fv)^{*}(\bq,\bv) = 0
    \quad \text{ and } \quad
    \nabla_\bv \fu(\bq,\bv), \nabla_\bq \fv(\bq, \bv) 
    \text{ are symmetric matrices,}
    \end{align}
    for all $(\bq, \bv) \in \RR^{2\fd}$. Then $\Sol$ is symplectic 
    relative to the canonical form $J$ as in \eqref{eq:cls:J} in 
    the sense of \cref{def:sym:Plec}, (ii). 
  \item[(iii)] Finally consider the case where
    $\f(\bqv) = \tJ^{-1} \nabla \tHam(\bz)$, for each $\bz \in \RR^{2\fd}$ where
    $\tHam \in C^2$ and where $\tJ$ is of the form
   \begin{align}\label{part:tJ}
	\tJ = 
	\begin{pmatrix}
		0 & - E^{-1} \\
		E^{-1} & 0 
	\end{pmatrix}
   \end{align}
   for some invertible matrix $E \in \RR^{\fd\times \fd}$. Then,
   in this circumstance, $\Sol$ is volume-preserving, 
   i.e. $|\det \nabla \Sol(\bz)| = 1$ for all $\bz \in \RR^{2\fd}$.
 \end{itemize}
\end{Theorem}
\begin{proof}
  We begin with the first item and prove that $\Sol$ is $R$-reversible
  under assumption \eqref{cond:fu:fv}.  Since $\Sol$ is the $n$-fold
  composition of $\SA_{\dT/2} \circ \SB_{\dT/2}$, due to
  \cref{prop:inv:R:S}, (ii), it suffices to show that
  $\SA_{\dT/2} \circ \SB_{\dT/2}$ is $R$-reversible. Moreover, as we
  already observed above, it is direct to verify that the adjoint of
  the Euler-B scheme is the Euler-A scheme in sense given in
  \cref{def:adj:num:op}, (i). In other words we have that
  $(\SB_{\dT/2})^* = \SA_{\dT/2}$ and so, as observed in
  \cref{lem:adj:props}, (i), we infer that
  $\SA_{\dT/2} \circ \SB_{\dT/2}$ is symmetric.  Invoking
  \cref{lem:adj:props}, (ii) and \cref{prop:inv:R:S}, (i), it
  therefore suffices to show that
  \begin{align}\label{rev:SA:SB}
    R \circ \SA_{\dT/2} = \SA_{-\dT/2} \circ R,
    \,\, \mbox{ and } \quad R \circ \SB_{\dT/2} = \SB_{-\dT/2} \circ R.
  \end{align}
  Let us verify the statement in \eqref{rev:SA:SB} for
  $\SA_{\dT/2}$. Let $\lbz = (\lbq, \lbv) \in \RR^{2\fd}$ and denote
  $\bz = (\bq, \bv) \coloneqq \SA_{\dT/2}(\lbz)$. From
  \eqref{def:EulerA}, we have
  \begin{align*}
    \bz = \lbz+ \frac{\dT}{2} \f(\bq, \lbv)
    = \lbz + \frac{\dT}{2} \f(\Pi_1 \bz, \Pi_2 \lbz).
  \end{align*}
  Applying $R$ and invoking \eqref{cond:fu:fv}, it follows that
  \begin{align}\label{eq:Rz1}
    R \bz = R \lbz + \frac{\dT}{2} R \f(\Pi_1 \bz, \Pi_2 \lbz) 
    = R \lbz - \frac{\dT}{2} \f(\Pi_1 \circ R \bz, \Pi_2 \circ R \lbz).
  \end{align}
  On the other hand, denoting $\btz \coloneqq \SA_{-\dT/2}(R \lbz)$,
  we have again from \eqref{def:EulerA} that
  \begin{align}\label{eq:tz1}
    \btz = R \lbz - \frac{\dT}{2} \f(\Pi_1 \btz, \Pi_2 \circ R \lbz)
  \end{align}
  From \eqref{eq:Rz1}, \eqref{eq:tz1} we conclude by uniqueness of
  solutions for \eqref{def:EulerB} that $R \bz = \btz$ or in other
  words $R \circ \SA_{\dT/2} (\lbz) = \SA_{-\dT/2} \circ R (\lbz)$,
  for every $\lbz \in \RR^{2\fd}$ as desired. The proof for $\SB$ is
  analogous.  Finally, notice that if $R$ is the momentum-flip
  involution \eqref{def:flip:R}, then \eqref{cond:fu:fv} reduces to
  the requirement \eqref{cond:fu:fv:flip}. This finishes the proof of
  the first item.
	
  Turning to the second item we follow the approach in \cite[Chapter
  4.1]{LeimkuhlerReich2004} and show that
  $d\bq \wedge d \bv = d\lbq \wedge d\lbv$ for
  $(\bq,\bv) = \SA_{\dT/2} (\lbq,\lbv)$ and for
  $(\bq,\bv) = \SB_{\dT/2} (\lbq,\lbv)$.  Here $d$ is the exterior
  derivative and $\wedge$ is the wedge product. See the footnote in
  \cref{def:sym:Plec} above and reference
  e.g. \cite{LeimkuhlerReich2004}, \cite{loring2011introduction} for
  further details and proper definitions.  Starting with Euler-A we
  have, referring back to \eqref{def:EulerA}, that
  \begin{align*}
	   d\bq = d\lbq + \frac{\dT}{2} \nabla_\bq \fu(\bq, \lbv) d\bq
	   +\frac{\dT}{2} \nabla_\bv \fu(\bq, \lbv) d\lbv , 
	   \qquad 
	   d\bv = d\lbv +  \frac{\dT}{2} \nabla_\bq \fv(\bq, \lbv)  d\bq
	   + \frac{\dT}{2}\nabla_\bv \fv(\bq, \lbv) d \lbv.
  \end{align*}
  We now compute $d \bq \wedge d \bv$ by expanding first in $d\bv$
  then in $d \bq$ appropriately.  As illustrated in \cite[Chapter
  3.6]{LeimkuhlerReich2004}, we make use of the identity that
  $d \bw \wedge (A d \btw) = (A^*d \bw) \wedge d \btw$ for any
  $\bw, \btw \in \RR^\fd$ and any $\fd \times \fd$ matrix $A$ so that
  in particular $d \bw \wedge (A d\bw) = 0$ when $A$ is
  symmetric. With our standing assumption \eqref{eq:symp:sur} we find
  \begin{align*}
	  d \bq \wedge d \bv 
	  &=
            d \bq \wedge \left( d\lbv
            +  \frac{\dT}{2} \nabla_\bq \fv(\bq, \lbv)  d\bq
	   + \frac{\dT}{2}\nabla_\bv \fv(\bq, \lbv) d \lbv \right)
            = d \bq \wedge d\lbv
            + d\bq \wedge\left( \frac{\dT}{2}\nabla_\bv \fv(\bq, \lbv) d \lbv \right)\\
	   &=  d \lbq \wedge d\lbv  + 
             \left(\frac{\dT}{2} \nabla_\bq \fu(\bq, \lbv) d\bq
             +\frac{\dT}{2} \nabla_\bv \fu(\bq, \lbv) d\lbv\right) \wedge d \lbv
	        + d\bq \wedge\left( \frac{\dT}{2}\nabla_\bv \fv(\bq,
             \lbv) d \lbv \right)
             = d \lbq \wedge d\lbv,
  \end{align*}
  as desired in the first case.  Regarding Euler-B,
  \eqref{def:EulerB}, we have
  \begin{align*}
          d\bq = d\lbq + \frac{\dT}{2}  \nabla_\bq  \fu(\lbq, \bv) d \lbq
          + \frac{\dT}{2} \nabla_\bv  \fu(\lbq, \bv) d \bv,  
          \qquad 
          d\bv = d\lbv + \frac{\dT}{2}\nabla_\bq \fv(\lbq, \bv)d \lbq
          + \frac{\dT}{2} \nabla_\bv \fv(\lbq, \bv) d\bv.
  \end{align*}
  Similarly to the previous case but now expanding first in $d \bq$
  and then in $d \bv$ we have:
  \begin{align}
	  d \bq \wedge d \bv 
	  &= \left( d\lbq + \frac{\dT}{2}  \nabla_\bq  \fu(\lbq, \bv) d \lbq
		+ \frac{\dT}{2} \nabla_\bv  \fu(\lbq, \bv) d \bv \right) \wedge d \bv
             = d \lbq \wedge d\bv  + 
             \left(  \frac{\dT}{2}  \nabla_\bq  \fu(\lbq, \bv) d \lbq \right) \wedge d \bv\\
             &= d \lbq \wedge \left( d\lbv + \frac{\dT}{2}\nabla_\bq \fv(\lbq, \bv)d \lbq
		+ \frac{\dT}{2} \nabla_\bv \fv(\lbq, \bv) d\bv\right)  + 
             \left(  \frac{\dT}{2}  \nabla_\bq  \fu(\lbq, \bv) d \lbq \right) \wedge d \bv 
             = d \lbq \wedge d\lbv,
  \end{align}
  completing the proof of the second item.
	
  We address the final item (iii) showing in this case that $\Sol$ is
  volume-preserving by explicitly computing $\det \nabla \Sol$.  Since
  $\Sol$ is the composition of half-steps of $\SB$ and $\SA$, it
  suffices to verify that each of these integrators is
  volume-preserving. We show this only for $\SA$, since the proof for
  $\SB$ follows analogously. Taking
  $(\fu(\bz),\fv(\bz)) = \tJ^{-1} \nabla \tHam(\bz)$ in
  \eqref{def:EulerA} and differentiating
  $(\bq^*, \bv^*) = \SA_{\dT}(\lbq,\lbv)$ with respect to $\lbq$ and
  $\lbv$, it follows that
  \begin{align*}
		\begin{pmatrix}
			I & - \dT E \tHam_{\bv \bv} \\
			0 & I + \dT E \tHam_{\bq \bv}
		\end{pmatrix}
		\nabla \SA_{1,\dT}(\lbq,\lbv)
		=
		\begin{pmatrix}
			I + \dT E \tHam_{\bv \bq} & 0 \\
			- \dT E \tHam_{\bq \bq} & I
		\end{pmatrix},
  \end{align*}
  where $\tHam_{\bq \bq}$, $\tHam_{\bq \bv} $, $\tHam_{\bv \bq}$,
  $\tHam_{\bv \bv} \in \RR^{\fd \times \fd}$ are evaluated at
  $(\bq^*, \bv^*)$ and denote the matrices of second-order partial
  derivatives of $\tHam$ with respect to the variables $\bq$ and/or
  $\bv$. Therefore,
  \begin{align}\label{nabla:SA}
          \nabla \SA_{\dT}(\lbq,\lbv)=
		\begin{pmatrix}
                  I + \dT E \tHam_{\bv \bq}
                  - \dT^2 E \tHam_{\bv \bv} (I + \dT E \tHam_{\bq\bv})^{-1}
                  E \tHam_{\bq \bq}
                  & \dT E\tHam_{\bv \bv} (I + \dT E \tHam_{\bq \bv})^{-1} \\
			- \dT (I + \dT E \tHam_{\bq \bv})^{-1} E\tHam_{\bq \bq}
                        & (I + \dT E \tHam_{\bq \bv})^{-1}
		\end{pmatrix}.
  \end{align}
  Invoking the formula for the determinant of block matrices, namely
  \begin{align}
	\det
	\begin{pmatrix}
	A & B \\
	C & D
	\end{pmatrix} 
	=\det( A D - B D^{-1} C D),  \quad
            \text{ whenever } A, B, C, D \in \RR^{\fd \times \fd}
            \text{ and } D \text{ is invertible} 
  \end{align}
  (see e.g. \cite{silvester2000}), a direct calculation yields that
  $|\det \nabla \SA_{\dT}(\lbq,\lbv)| = 1$ for any
  $(\lbq,\lbv) \in \RR^{2\fd}$. The proof is now complete.
\end{proof}

Let us make several further technical remarks concerning
\cref{prop:HMC:Leap:Frog}.

\begin{Remark} \label{rmk:ex:unq:imp}
    It is natural to ask if there are reasonable conditions which may be
  placed on $\f := (\fu,\fv)$ and on $\dT >0$ in \eqref{def:Sol:fd}
  guaranteeing a unique solution to this implicitly defined scheme.
  Here, suppose, for example, that $\f$ satisfies a global Lipschitz
  condition, namely
  \begin{align}\label{eq:glo:lip}
    |\f(\tilde{\bz}) - \f(\bz)| \leq K | \tilde{\bz}- \bz|
    \quad
    \text{ for any } \tilde{\bz},\bz \in \RR^{\fd},
  \end{align}
  where $K > 0$ is an absolute constant independent of $\tilde{\bz},\bz$.
  Then, under \eqref{eq:glo:lip}, for any $\delta < K^{-1}$ the scheme
  \eqref{def:Sol:fd} is uniquely defined for any $n \geq 1$, as may be
  readily demonstrated using the Banach fixed point theorem.  Indeed,
  noting that \eqref{def:Sol:fd} is a composition of the Euler A and B
  subschemes, \eqref{def:EulerB} and \eqref{def:EulerA}, we may apply
  this fixed point argument to each of these substeps in turn.
  Regarding the Euler B step, \eqref{def:EulerB}, consider, for any
  $\bq, \bv, \btq, \btv \in \RR^\fd$,
  \begin{align*}
    F_\delta(\bq, \bv, \btq, \btv)
        := ( \btq + \dT \fu(\bq, \btv) , \btv + \dT \fv(\bq, \btv)).
  \end{align*}
  Under \eqref{eq:glo:lip} it is clear that, for any fixed
  $\btq, \btv \in \RR^\fd$ the map
  $(\bq, \bv) \mapsto F_\delta(\bq, \bv, \btq, \btv)$ is contractive
  on $\RR^{2\fd}$ so long as $\delta < K^{-1}$ and hence, for any
  given $\btq, \btv$ and $\delta$ we obtain a unique $(\bq^*, \bv^*)$
  such that $(\bq^*, \bv^*) = F_\delta(\bq^*, \bv^*, \btq, \btv)$.  In
  other words we have shown that \eqref{def:EulerB} must have a unique
  solution.  Of course precisely the same argument may be applied also
  to \eqref{def:EulerA}.  Finally we note that, under
  \eqref{eq:glo:lip} it is direct to show that $\SA_\delta$ and
  $\SB_\delta$ are $C^1$ functions.
\end{Remark}

\begin{Remark}\label{rmk:rev:ham:case}
  In the case where $\f(\bqv) = \tJ^{-1} \nabla \tHam(\bz)$ the
  reversibility condition \eqref{cond:fu:fv} in
  \cref{prop:HMC:Leap:Frog}, (i) can be characterized as follows.
  Given any linear invertible mapping $R: \RR^{2\fd} \to \RR^{2\fd}$,
  which we write in the block form
\begin{align*}
 R = \begin{pmatrix}
A & B \\
C & D
\end{pmatrix}, 
\quad A, B, C, D \in \RR^{\fd \times \fd},
\end{align*}
we suppose that
\begin{align}\label{cond:tHam}
  \tHam(\Pi_1 \circ R \bz, \Pi_2 \circ R \btz)
  = \tHam(\Pi_1 \bz, \Pi_2 \btz)
  \quad \mbox{ for all } \bz, \btz \in \RR^{2\fd},
\end{align}
and that
\begin{align}\label{cond:R:tJ}
  R \tJ^{-1} R_0^* = - \tJ^{-1}, \quad \mbox{ where } 
  R_0 \coloneqq 
  \begin{pmatrix}
    A & 0 \\
    0 & D
  \end{pmatrix}.
\end{align}
To see that these conditions yield \eqref{cond:fu:fv} observe that
using \eqref{cond:tHam} we obtain after a direct calculation that
$\nabla \tHam (\Pi_1 \bz, \Pi_2 \btz) = R_0^* \nabla \tHam(\Pi_1 \circ
R \bz, \Pi_2 \circ R \btz)$, for any $\bz, \btz \in \RR^{2\fd}$.
Thus
\begin{align*}
  R \f(\Pi_1 \bz, \Pi_2 \btz) = R \tJ^{-1} \nabla \tHam(\Pi_1 \bz, \Pi_2 \btz) 
  &= R \tJ^{-1} R_0^* \nabla \tHam (\Pi_1 \circ R \bz, \Pi_2\circ  R \btz)\\
  &= - \tJ^{-1} \nabla \tHam(\Pi_1 \circ R \bz, \Pi_2 \circ R \btz)
    = - \f (\Pi_1 \circ R \bz, \Pi_2 \circ R \btz)
\end{align*}
refer where the third equality follows from \eqref{cond:R:tJ}.  Note that
this condition \eqref{cond:tHam}, \eqref{cond:R:tJ} is
comparable to \eqref{eq:r:j:sym:cond}, \eqref{eq:H:r:sym:cond}
in \cref{prop:Ham:props:gen}.
\end{Remark}

\begin{Remark}
  In \cite[Theorem 3.3]{Hairer2006book}, it is shown that if $\tJ$ is
  of the form in \eqref{part:tJ} with $E = I$, then $\SA$ and $\SB$
  are in fact symplectic integrators with respect to $\tJ$. Namely,
  $(\nabla \SA_{\dT}(\bz))^* \tJ \nabla \SA_{\dT}(\bz) = \tJ$, for all
  $\bz \in \RR^{2\fd}$, and analogously for $\SB$. This can be
  verified directly from \eqref{nabla:SA} by taking $E = I$ or by
  repeating the approach given in \cref{prop:HMC:Leap:Frog}, (ii).
  However, the same does not seem to hold for a general invertible
  $E \in \RR^{\fd \times \fd}$.
\end{Remark}

Combining \cref{prop:HMC:Leap:Frog} with \cref{prop:Ham:res:1} we now
formulate the final algorithm of this section as:

\begin{algorithm}[H]
  \caption{(Surrogates for Non-separable Hamiltonians to sample
    $\mu(d\bq) = Z_\UPot^{-1} e^{-\UPot(\bq)} d\bq$)}
\begin{algorithmic}[1]\label{alg:gHMC:FD:3}
  \State Select the algorithm parameters:
  \begin{itemize}
  \item[(i)] The proposal distribution
    $\Vker(\bq, d \bv) = Z_{\VPot}(\bq)^{-1} e^{-\VPot(\bq,\bv)} d \bv \in \Pr(\RR^\fd)$
    for each $\bq \in \RR^\fd$.
  \item[(ii)] Surrogate functions $\fu, \fv: \RR^{2\fd} \to \RR^\fd$
    specified so that either \eqref{eq:symp:sur} holds or
    $(\fu(\bz), \fv(\bz)) = J^{-1} \nabla \Ham(\bz)$ where $J$ is as in
    \eqref{part:tJ}.
  \item[(iii)] Relative to $\fu, \fv$ identify a linear involution $R$
    maintaining \eqref{cond:fu:fv} (see in particular
    \eqref{cond:fu:fv:flip}).
  \item[(iv)] The time step $\dT > 0$ and number of iterations
    $n \geq 1$.
  \end{itemize}
   \State Choose $\bq_0 \in \RR^\fd$.
    \For{$k \geq 0$}
        \State Sample $\bv_k \sim \Vker(\bq_k, d \bv)$.
        \State Propose
        $(\lbq_{k+1}, \lbv_{k+1}) := \Sol_{n,\dT} (\bq_k, \bv_k)$
        where $\Sol_{n, \dT} = (\SA_{\dT/2} \circ \SB_{\dT/2})^n$ is
        defined implicitly via \eqref{def:Sol:fd}.
        \State Set $\bq_{k+1} :=\lbq_{k+1}$ with probability
        $1 \wedge \left[\exp( - \Ham(R (\lbq_{k+1}, \lbv_{k+1})) +
          \Ham(\bq_k,\bv_k)) \right]$ and otherwise take
        $\bq_{k+1} := \bq_{k}$.
        \EndFor
\end{algorithmic}
\end{algorithm}

\begin{Remark}\label{rmk:rel:alg}
  The methods described in \cref{alg:gHMC:FD:1}, \cref{alg:gHMC:FD:2}
  and \cref{alg:gHMC:FD:3} have quite a bit of overlapping scope and
  can be generalized in a number of immediate ways working from the
  frameworks developed here.  Regarding the scope of the methods
  notice, for example, that \cref{alg:gHMC:FD:1} is actually a special 
  case of \cref{alg:gHMC:FD:2} when  $\fu(\bv) = (A^*)^{-1} \nabla \tilde{\VPot}(\bv)$,
  $\fv(\bq) = -A^{-1} \nabla \tilde{\UPot}(\bq)$, for any
  $\tilde{\VPot}, \tilde{\UPot} \in C^1$ and any invertible $A$, or a special 
  case of \cref{alg:gHMC:FD:3} when, more generally, 
  $\fu = \fu(\bv)$, $\fv = \fv(\bq)$, with appropriate choices for the time steps.  In
  fact, as already observed above, \eqref{def:Sol:fd} reduces to
  \eqref{def:leapfrog:int} with $\da = \dT/2$ and $\db = \dT$ whenever
  $\fu$ only depends on $\bv$ and $\fv$ only on $\bq$.

  Regarding further generalizations we observe that
  \cref{alg:gHMC:FD:1} and \cref{alg:gHMC:FD:3} are easily extended to
  any palindromic splitting of $\Xi^{(1)}, \Xi^{(2)}$ or of $\SA, \SB$
  in the fashion of \eqref{eq:ham:spt:gen:alg} in
  \cref{alg:gHMC:FD:2}. Note that the efficacy of such extended
  palindromic settings have been explored in more depth recently in
  \cite{bou2018geometric}.  In a different direction we note that in
  all of the above algorithms we have identified conditions
  guaranteeing that the integrator is symplectic or at least volume
  preserving.  In principle, \cref{alg:gHMC:FD} allows for schemes
  $\cSol$ where $|\det(\nabla \cSol(\bqv))| \not \equiv 1$. However,
  at least within a naive set-up, the computational requirements
  around such gradient terms would appear to be lethal. Nevertheless,
  see the recent work \cite{levy2017generalizing}.
\end{Remark}

\begin{Remark}\label{rmk:sub:error}
  People who have applied HMC to complicated problems may have
  encountered scenarios where the implementation of the numerical
  integrator was incorrect, but the sampler still produced correct
  results.  The structure of \cref{alg:gHMC:FD}, \cref{alg:gHMC:FD:1},
  \cref{alg:gHMC:FD:2} and \cref{alg:gHMC:FD:3} (see also
  \cite{neal2011mcmc}) provides an explanation for this phenomenon --
  as long as the integrator and acceptance ratio are consistent with
  each other, the accept/reject step will correct for any errors in
  the propagation in the dynamical system, whether they be due to
  discretization errors or bugs in the implementation.
\end{Remark}

\subsection{The infinite-dimensional case}
\label{sec:HMC:inf:dim}

We turn now to introduce our second class of generalized trajectory
methods.  These methods, summarized in \cref{alg:method} below, are
adapted to sampling from measures which are defined on a separable
Hilbert space and which are absolutely continuous with respect to a
certain class of Gaussian base measures.  As mentioned above, our
approach here takes its inspiration from a preconditioned HMC method
introduced in \cite{Beskosetal2011} and later revisited in
\cite{beskos2013advanced,beskos2017geometric}.  Actually, as we show
below in \cref{subsec:app:Hilb}, the algorithms introduced in
\cite{Beskosetal2011, beskos2017geometric} along with the pCN and
$\infty$-MALA samplers from e.g. \cite{beskos2008,cotter2013mcmc} all
fall out as special parameter choices in our class of samplers
summarized in \cref{alg:method}.

Let us begin by making our infinite-dimensional setting precise.
Throughout this section we assume that the parameter space $\spq$ is a
separable Hilbert space, with associated norm and inner product
denoted by $|\cdot|$ and $\langle \cdot, \cdot \rangle$,
respectively. We take $\cB_{\spq}$ to signify the associated
$\sigma$-algebra of Borel subsets of $\spq$.  On $\spq$ we consider
target probability measures of the form
\begin{align}\label{Gibbs:meas}
  \mu(d \bq) = \frac{1}{Z} e^{-\Pot(\bq)} \mu_0 (d \bq), \quad
  Z = \int_X e^{-\Pot(\bq)} \mu_0 (d \bq),
\end{align}
for some potential function $\Pot: X \to \RR$ such that
$e^{-\Pot(\bq)}$ is $\mu_0$-integrable, and with $\mu_0$ given as a
centered Gaussian measure on $\spq$ with covariance operator
$\cC: \spq \to \spq$, i.e. $\mu_0 = \cN(0, \cC)$. 

Regarding the covariance structure for the Gaussian measure $\mu_0$,
we assume throughout what follows that $\cC$ is a trace-class,
symmetric and strictly positive definite operator. By the spectral
theorem, it follows that $X$ admits a complete orthonormal basis
$\{\e_i\}_{i \in \NN}$ consisting of eigenfunctions of $\cC$,
corresponding to a non-increasing sequence of positive eigenvalues
$\{\lambda_i\}_{i \in \NN}$.\footnote{Here we note that the
  trace-class condition amounts to requiring that
  $\sum_{j = 1}^\infty \lambda_j < \infty$.}  We thus define the
fractional powers $\cC^\gamma$ of $\cC$ for any $\gamma \in \RR$ as
\begin{align}\label{frac:C}
  \cC^\gamma \bq
  = \sum_{i =1}^\infty \lambda_i^\gamma \langle \bq, \e_i \rangle \e_i,
    \quad 
    \mbox{ for all } \bq \in \mbox{Dom}(\cC^\gamma),
\end{align}
where $\mbox{Dom}(\cC^\gamma)$ can be characterized as the set of all
$\bq \in \spq$ such that
$\sum_{i =1}^\infty \lambda_i^{2\gamma} \langle \bq, \e_i \rangle^2 <
\infty$ when $\gamma \leq 0$ while taking $\mbox{Dom}(\cC^\gamma)$ as
the dual relative to $X$ of $\mbox{Dom}(\cC^{-\gamma})$ when
$\gamma > 0$.  We refer the reader to
\cite{bogachev1998gaussian, DPZ2014} for further general background on
the wider theory of Gaussian measures on function spaces.

The preconditioned Hamiltonian Monte Carlo algorithm introduced in
\cite{Beskosetal2011} to sample from
$\mu(d\bq) = Z^{-1} e^{-\Pot(\bq)} \mu_0(d \bq)$ is based on the
following Hamiltonian function
\begin{align}\label{def:Ham:pHMC}
  \Ham(\bq, \bv) \coloneqq \frac{1}{2} |\cC^{-1/2} \bv|^2
  + \frac{1}{2} |\cC^{-1/2} \bq|^2 + \Pot(\bq).
\end{align}
The potential function $\Pot$ is assumed to be Fr\'echet
differentiable, with its Fr\'echet derivative denoted by
$D \Pot: \spq \to \spq$, where the Hilbert space $X$ is identified
with its dual. It thus follows that \eqref{def:Ham:pHMC} has an
associated (non-canonical) Hamiltonian dynamic given by
\begin{align}\label{ham:dyn}
    \frac{d\bq}{dt} = \bv,
    \quad
    \frac{d\bv}{dt} = -\bq - \cC D\Pot(\bq),
\end{align}
which formally corresponds to the more compact formulation
\eqref{eq:Ham:dyn:comp:fm} with the following choice of
``preconditioning'' operator
\begin{align}
  \label{eq:non:con:form:J:PC}
  J :=
  \begin{pmatrix}
    0&   - \cC^{-1}\\
    \cC^{-1}&     0
  \end{pmatrix}.
\end{align}
Under suitable assumptions on the potential function $\Pot$, one can
show that the Hamiltonian flow associated to \eqref{ham:dyn} holds
$\mu\otimes \mu_0$ as an invariant measure, see \cite[Theorem
3.1]{Beskosetal2011}. As a consequence, the Markov transition kernel
associated to the evolution of the $\bq$ variable in \eqref{ham:dyn},
defined as
\begin{align*}
  P^t(\bq_0,A) \coloneqq \Prb (\bq(t;\bq_0, \bv_0) \in A),
  \quad \bv_0 \sim \mu_0,
\end{align*}
for all $\bq_0 \in \spq$ and $A \in \cB_\spq$, holds the marginal of
$\mu\otimes \mu_0$ onto the $\bq$ variable, i.e. $\mu$, as an
invariant measure.

A crucial insight in \cite{Beskosetal2011} later revisited in
\cite{beskos2013advanced,beskos2017geometric}, is that each of the
components of a certain Strang splitting of \eqref{ham:dyn} given by
\begin{align}\label{split:ham:classic}
    \frac{d\bq}{dt} = 0, \quad \frac{d\bv}{dt} = - \cC D\Pot(\bq),
  \quad \text{ and } \quad
    \frac{d\bq}{dt} = \bv, \quad  \frac{d\bv}{dt} = -\bq,
\end{align}
yield maps $\Xi^{(1)}_t$, $\Xi^{(2)}_t$ such that
$\Xi^{(i)}_t(\mu_0 \otimes \mu_0)$ are absolutely continuous with
respect to $\mu_0 \otimes \mu_0$ for $i =1,2$ and any
$t \geq 0$.\footnote{In fact it is not hard to see that $\Xi^{(2)}_t$
  leaves $\mu_0 \otimes \mu_0$ fixed as is shown in the proof of
  \cref{thm:rev:HMC} below.}  Taking \eqref{def:Ham:pHMC} and the splitting
\eqref{split:ham:classic} of \eqref{ham:dyn} as a starting point
therefore suggests the following more general class of trajectory
methods.

Firstly we may replace $\mathcal{H}(\bq, \bv)$ in \eqref{def:Ham:pHMC}
with a more general class of non-separable Hamiltonians which allows
for a ``position''-dependent kinetic energy component, namely
\begin{align}
  \label{eq:non:sep:Ham}
  \Ham(\bq, \bv)
  \coloneqq \HmT(\bq, \bv) + \frac{1}{2} |\cC^{-1/2} \bv|^2
  + \frac{1}{2} |\cC^{-1/2} \bq|^2 + \Pot(\bq)
\end{align}
for a measurable function $\HmT: \spq \times \spq \to \RR$.  In view
of obtaining a well-defined probability measure on the extended state
space $\spq \times \spq$, we require the following integrability
condition
\begin{align}
  \label{eq:gibb:ID:cond:1}
  \int_{X \times X}\exp(- \HmT(\bq, \bv) - \Pot(\bq))
  \mu_0 \otimes \mu_0(d\bq, d\bv) < \infty.
\end{align}
We further assume that
\begin{align}
  \label{eq:gibb:ID:cond:2}
  \int_{X}\exp(- \HmT(\bq, \bv)) \mu_0(d\bv) = 1, \quad
  \mu_0 \text{ almost surely in } \bq.
\end{align}
This later assumption ensures that the position marginal of the Gibbs
measure $\cM$ associated to $\Ham$ given by \eqref{eq:non:sep:Ham} is
the desired target measure $\mu$ as in \eqref{Gibbs:meas}.  We can
therefore write, in the notational formulation of
\cref{thm:gen:rev:new},
\begin{align}\label{eq:M:gen:inf:form}
  \cM(d\bq, d \bv) = \Vker(\bq, d\bv) \mu(d \bq)
  \quad \text{ where } \Vker(\bq, d\bv) = \exp(- \HmT(\bq, \bv)) \mu_0(d\bv),
\end{align}
so that \eqref{eq:gibb:ID:cond:1} implies mutual absolute continuity
between $\cM$ and $\cM_0 := \mu_0 \otimes \mu_0$.\footnote{One can
  formally relate \eqref{eq:Ham}, \eqref{eq:gibbs:Msr} and
  \eqref{eq:non:sep:Ham}, \eqref{eq:M:gen:inf:form} as follows: We
  assume that
  $Z_\VPot(\bq) := \int \exp(- \HmT(\bq, \bv) - \frac{1}{2} |
  \cC^{-1/2} \bv|^2) d\bv \equiv 1$ so that \eqref{eq:Ham} reduces to
  \eqref{eq:non:sep:Ham} by setting
  $\VPot(\bq, \bv) = \HmT(\bq, \bv) + \frac{1}{2} | \cC^{-1/2} \bv|^2$
  and $\UPot(\bq) = \frac{1}{2} |\cC^{-1/2} \bq|^2 + \Pot(\bq)$.  Of
  course these relationships are purely formal when
  $\dim(\spq) = \infty$ since the Lebesgue measure is ill-defined in
  infinite dimensions.}

A Hamiltonian dynamic associated to the non-separable Hamiltonian in
\eqref{eq:non:sep:Ham} can be formally written as
\begin{align}\label{gen:Ham:infd}
	\frac{d\bz}{dt} = J^{-1} D \Ham (\bz),
\end{align}
for some general ``preconditioning'' operator $J$, and where $D \Ham$
denotes the Fr\'echet derivative of $\Ham$. Under appropriate
assumptions on $J$ and $\Ham$, and following a similar
finite-dimensionalization argument as in \cite{Beskosetal2011}, one
can show that such dynamic holds the probability measure $\cM$ from
\eqref{eq:M:gen:inf:form} as invariant. However, notwithstanding
\cref{ssec:imp:int} above, deriving a suitable numerical integrator
for \eqref{gen:Ham:infd} that is well-defined particularly in this
infinite-dimensional context is a nontrivial task.

An alternative consists in defining a suitable ``surrogate dynamics''
for \eqref{gen:Ham:infd} that can be more easily integrated. Here we
consider the following general and possibly gradient-free ``surrogate
dynamics'', namely
\begin{align}\label{mod:ham:dyn}
    \frac{d\bq}{dt} = \bv,
    \quad
    \frac{d\bv}{dt} = -\bq -  \agp (\bq),
\end{align}
where $\agp: \spq \to \spq$ is a measurable mapping.  In order to
maintain absolute continuity with respect to $\mu_0$, we assume that
$f(\bq)$ is in the Cameron-Martin space of $\mu_0$ for every
$\bq \in \spq$, or in other words
\begin{align}
  f(\bq) \in \mbox{Dom}(\cC^{-1/2})
  \quad \text{ for every } \bq \in \spq.
	\label{eq:f:CM:reg}
\end{align}
Notably, in the particular case that the kernel $\Vker$ from
\eqref{eq:M:gen:inf:form} is given by
$\Vker(\bq, \cdot) = \mu_0 = \cN(0, \cC)$ for every $\bq \in \spq$, as
in \eqref{def:Ham:pHMC}, we may view $\agp$ as a surrogate for
$\cC D \Pot$.  Then, following \eqref{split:ham:classic}, we consider
a numerical discretization of \eqref{mod:ham:dyn} by first splitting
the dynamics into
\begin{align}\label{split:ham:1}
    \frac{d\bq}{dt} = 0, \quad \frac{d\bv}{dt} = -  \agp (\bq)
\end{align}
and
\begin{align}\label{split:ham:2}
    \frac{d\bq}{dt} = \bv, \quad  \frac{d\bv}{dt} = -\bq.
\end{align}
Clearly, the solution of \eqref{split:ham:1} at any time $t$ starting
from $(\bq_0, \bv_0) \in \spq \times \spq$ is given by
\begin{align}\label{def:Xi1}
    \Xi_{\agp,t}^{(1)}(\bq_0, \bv_0) = \Xi_t^{(1)}(\bq_0, \bv_0)   
                     = (\bq_0, \bv_0 - t \agp(\bq_0)),
\end{align}
while the corresponding solution of \eqref{split:ham:2} is given by
\begin{align}\label{def:Xi2}
  \Xi_t^{(2)}(\bq_0, \bv_0)
  = (\cos(t)\bq_0 + \sin(t) \bv_0, - \sin(t) \bq_0 + \cos(t)\bv_0).
\end{align}

For any fixed time steps $\da > 0$ and $\db > 0$ for \eqref{def:Xi1}
and \eqref{def:Xi2}, respectively, and a fixed number $n \in \NN$ of
iterative steps, we consider a numerical integrator for
\eqref{mod:ham:dyn} starting from
$(\bq_0, \bv_0) \in \spq \times \spq$ given by the following
Strang-type splitting
\begin{align}\label{Verlet:int}
  \Sol (\bq_0, \bv_0)
  = \left( \Xi_{\da}^{(1)} \circ \Xi_{\db}^{(2)} \circ \Xi_{\da}^{(1)} \right)^n (\bq_0, \bv_0).
\end{align}
Here note that, with $R$ taken as the usual ``momentum-flip'' operator
defined in \eqref{def:flip:R}, and invoking \cref{prop:inv:R:S}.(i)
and \cref{lem:cons:symp:rev}.(i), it is direct to verify that
$S = R\circ \Sol$ is an involution.  For convenience, we also denote
by $\Si(\bq_0, \bv_0) = (\bq_i, \bv_i)$ the solution arising from $i$
steps taken in the process of \eqref{Verlet:int}, i.e.
\begin{align}\label{not:int:i}
  \Si(\bq_0,\bv_0) = (\bq_i, \bv_i)
  = \left( \Xi_{\da}^{(1)} \circ \Xi_{\db}^{(2)} \circ \Xi_{\da}^{(1)}
  \right)^i(\bq_0, \bv_0),
\end{align}
for $i = 1, \ldots, n$.

We notice that, in contrast to \eqref{ham:dyn}, the system
\eqref{mod:ham:dyn} for general $\agp$ may not generate a Hamiltonian
dynamics, and also may not hold $\mu \otimes \mu_0$ as an invariant
measure.  Similarly with a more general class of Gibbsian measures
$\cM$ as in \eqref{eq:non:sep:Ham}, \eqref{eq:M:gen:inf:form} it is
not clear how to select an $\agp$ so that the resulting flow
\eqref{mod:ham:dyn} holds $\cM$ as an invariant.\footnote{In this
  sense we may understand the algorithms introduced in
  \cite{beskos2017geometric} as themselves being surrogate methods.}
Nevertheless, as we now demonstrate, the class of numerical
integrators $\Sol$ defined by \eqref{Verlet:int} and by ``momentum
selection'' mechanisms $\Vker$ given in \eqref{eq:M:gen:inf:form}
allow for the identification of an accept-reject function $\harFn$ as
in \eqref{def:arFn:b} which generates a sampling algorithm of the form
\eqref{eq:ext:MH:ker:diff} that is reversible with respect to the
desired target measure $\mu$.

\begin{Theorem}\label{thm:rev:HMC}
  Let $\spq$ be a separable Hilbert space with $\cB_\spq$ denoting the
  $\sigma$-algebra of Borel subsets of $\spq$ and take
  $\cC: \spq \to \spq$ be a trace-class, symmetric and strictly
  positive definite operator. Let $\mu_0 = \cN(0, \cC)$ and suppose
  that $\Pot: \spq \to \RR$ is a measurable mapping with
  $e^{-\Pot(\cdot)} \in L^1(\mu_0)$.  Under these assumptions we
  consider a target probability measure
  $\mu(d\bq) = \frac{1}{Z} e^{-\Pot(\bq)} \mu_0(d\bq)$ with
  $Z :=\int_\spq e^{-\Pot(\bq)} \mu_0(d\bq)$.
  
  We now define a class of proposal Markov kernels as follows.  Fix
  any measurable $\agp: \spq \to \spq$ such that $\agp(\bq)$ is in the
  Cameron-Martin space of $\mu_0$ for every $\bq \in \spq$ namely so
  that $\agp$ is subject to \eqref{eq:f:CM:reg}.  We then let $\Sol$
  be the numerical integrator for the dynamics \eqref{mod:ham:dyn}
  defined in \eqref{Verlet:int} from \eqref{def:Xi1}, \eqref{def:Xi2}
  for any fixed discretization parameters $\da > 0$ and $\db > 0$, and
  any desired number of iteration steps $n \geq 1$.  Finally we fix
  any $\cM(d\bq,d \bv) := \Vker(\bq, d\bv) \mu(d\bq)$ such that the
  Markov kernel $\Vker$ has the form
  $\Vker(\bq, d\bv) = \exp(- \HmT(\bq, \bv)) \mu_0(d\bv)$ for any
  measurable $\HmT: \spq \times \spq \to \RR$ subject to integrability
  and normalization conditions \eqref{eq:gibb:ID:cond:1},
  \eqref{eq:gibb:ID:cond:2}. Relative to this data we consider
  proposal kernels of the form
  $Q(\bq, d \btq) = (\Pi_1 \circ \Sol(\bq, \cdot))^*\Vker(\bq,
  d\btq)$.
  
  Then, taking $R$ to be the ``momentum-flip'' operator as defined in
  \eqref{def:flip:R}, $S := R \circ \Sol$ is an involution and
  $(R \circ \Sol)^*\cM$ is absolutely continuous with respect to
  $\cM$. Furthermore, defining the function
  $\harFn: \spq \times \spq \to [0,1]$ as
\begin{align}\label{arFn:HMC:inf}
  \harFn(\bq_0, \bv_0)
  \coloneqq 1 \wedge
  \frac{d (R \circ \Sol)^*\cM}{d \cM}(\bq_0, \bv_0) ,
  \quad (\bq_0, \bv_0) \in \spq \times \spq,
\end{align}
the associated Markov kernel $P: \spq \times \cB_{\spq} \to [0,1]$
given as in \eqref{eq:ext:MH:ker:diff}, with these choices of $S$ and
$\Vker$, is reversible with respect to $\mu$.  Moreover, invoking the
notation from \eqref{not:int:i}, we have
\begin{align}
    &\frac{d (R \circ \Sol)^*\cM}{d \cM}(\bq_0, \bv_0) 
    \notag\\
    &\quad = \exp \left(  \Pot(\bq_0) + \HmT(\bq_0, \bv_0)  - \Pot(\bq_n) -
      \HmT(\bq_n, -\bv_n)
    - \frac{\da^2}{2} \left[ |\cC^{-1/2} \agp(\bq_0)|^2 - |\cC^{-1/2} \agp(\bq_n)|^2 \right] \right. 
    \notag\\
   &\qquad \qquad 
     \left.   + 2\da \sum_{i=1}^{n-1} \langle \cC^{-1/2}\bv_i, \cC^{-1/2}\agp(\bq_i) \rangle 
    + \da \left[ \langle \cC^{-1/2} \bv_0,  \cC^{-1/2}\agp(\bq_0)
     \rangle
     + \langle \cC^{-1/2}\bv_n, \cC^{-1/2} \agp (\bq_n) \rangle  \right] \right),
    \label{RN:RS:mu:mu0}
\end{align}
 for $\cM$-a.e. $(\bq_0, \bv_0) \in \spq \times \spq$.
\end{Theorem}

We summarize the class of algorithms to sample from measures of the
form $\mu(d\bq) \propto e^{-\Pot(\bq)} \mu_0(d\bq)$ for
$\mu_0 = N(0, \cC)$ resulting from \cref{thm:rev:HMC} as follows:
\begin{algorithm}[H]
\caption{Splittings of preconditioned dynamics for Gaussian-Hilbert space targets
  $\tfrac{e^{-\Phi(\bq)} }{Z} \mu_0(d \bq)$.}
\begin{algorithmic}[1]\label{alg:method}
    \State Select the algorithm parameters:
    \begin{itemize}
      \item[(i)] The form of $\HmT(\bq, \bv)$ for the
      proposal distribution $\Vker(\bq, d\bv) := \exp(- \HmT(\bq, \bv)) \mu_0(d\bv)$.
      \item[(ii)] The approximation $\agp$ of $D \Pot$. 
      \item[(iii)] The discretization time step parameters $\da > 0$,
        $\db > 0$ in the splittings \eqref{def:Xi1} and
        \eqref{def:Xi2} respectively.
      \item[(iv)] The number of iterative steps $n \in \mathbb{N}$ now
        defining $\Sol$ in \eqref{Verlet:int}.
    \end{itemize}
    \State Choose $\bq_0 \in X$
    \For{$k \geq 0$}
        \State Sample $\bv_k \sim \Vker(\bq_k, d\bv)$
        \State Propose $\lbq_{k+1} := \Projq \circ \Sol(\bq_k,\bv_k)$,
            with $\Sol$ as defined in \eqref{Verlet:int}
        \State Set $\bq_{k+1} :=\lbq_{k+1}$ with probability
                   $\harFn(\bq_{k}, \bv_{k})$ given by
                   \eqref{arFn:HMC:inf},
                   computed via \eqref{RN:RS:mu:mu0}, otherwise
                   take $\bq_{k+1} := \bq_{k}$.
    \EndFor
\end{algorithmic}
\end{algorithm}

We turn now to the proof of \cref{thm:rev:HMC}.  Before diving into
the details we attempt to provide some intuition.  The integrator
$S = R \circ \Sol$ in this case involves repeated applications of the
maps $\Xi_{\da}^{(1)}$, $\Xi_{\db}^{(2)}$, and $R$. The latter two
maps hold $\mu_0 \otimes \mu_0$ invariant, while the Radon-Nikodym
derivative associated with the pushforward of $\mu_0 \otimes \mu_0$ by
$\Xi_{\da}^{(1)}$ can be computed via the Cameron-Martin Theorem. We
then repeatedly apply the property \eqref{RN:push:comp} to compute the
Radon-Nikodym derivative associated with the pushforward of
$\mu_0 \otimes \mu_0$ by $S$, thus yielding the the desired acceptance
probability given by \eqref{arFn:HMC:inf}, \eqref{RN:RS:mu:mu0}.

\begin{proof}[Proof of \cref{thm:rev:HMC}]
  As usual we proceed by establishing conditions \ref{P1c}, \ref{P2c}
  in \cref{thm:gen:rev:new} for the elements $S =R \circ \Sol$ and
  $\cM$ specified according to \eqref{Verlet:int} and
  \eqref{eq:M:gen:inf:form}, respectively.  In the process of
  establishing \ref{P2c} we demonstrate \eqref{RN:RS:mu:mu0} by making
  suitable use of various identities from \cref{subsec:prelim:meas}.
  Regarding \ref{P1c}, from the definitions of $\Xi_t^{(1)}$ and
  $\Xi_t^{(2)}$ in \eqref{def:Xi1} and \eqref{def:Xi2}, it is not
  difficult to check that
  \begin{align}\label{eq:Inv:stp:GHS}
    (R \circ \Xi_t^{(1)})^2 = I
    \quad \text{ and } \quad
    (R \circ \Xi_t^{(2)})^2 = I
    \quad    \text{ for any } t \geq 0.
  \end{align}
  Hence, from
  \cref{prop:inv:R:S}, (i) and \cref{lem:cons:symp:rev}, (i) it follows
  that $(R \circ \Sol)^2 = I$, so that $S$ is indeed an involution and
  thus condition \ref{P1c} is verified.

  Turning to condition \ref{P2c} of \cref{thm:gen:rev:new} we
  introduce an intermediate measure $\cM_0 = \mu_0\otimes
  \mu_0$. Notice that, embedded in the assumption
  \eqref{eq:gibb:ID:cond:1}, we have that $\cM$ and $\cM_0$ must be
  mutually absolutely continuous and hence it follows that $S^*\cM$
  and $S^*\cM_0$ must also be mutually absolutely continuous.  Thus,
  in order to establish \ref{P2c}, it suffices to prove that
  $S^* \cM_0$ is absolutely continuous with respect to $\cM_0$.  This
  being true, notice moreover that, with the fact that $S$ is and
  involution and the aide of the identities \eqref{chain:rule:meas},
  \eqref{RN:pushfwd}, we have
  \begin{align}\label{RN:S:cM:cM0}
    \frac{d S^* \cM}{d \cM}(\bq, \bv)
    = \frac{d  \cM_0}{d \cM} (\bq, \bv)
        \frac{d \cM}{d \cM_0} (S(\bq, \bv)) 
        \frac{d S^* \cM_0}{d \cM_0} (\bq, \bv) 
  \end{align}
  for $\cM$-a.e. $(\bq, \bv) \in \spq \times \spq$.  Here, using
  \eqref{eq:RN:Der:form}, we have
  \begin{align}
  \label{eq:dsMM:p1}
    \frac{d \cM_0 }{d \cM}( \bq_0, \bv_0)
    = \exp \left(  \Pot(\bq_0) + \HmT(\bq_0, \bv_0) \right)
  \end{align}
  and similarly, invoking the notation introduced in \eqref{not:int:i},
  \begin{align}
    \label{eq:dsMM:p2}
    \frac{d \cM }{d \cM_0}( S( \bq_0, \bv_0))
      = \frac{d \cM }{d \cM_0}( \bq_n, -\bv_n)
      =\exp \left(  - \Pot(\bq_n) -\HmT(\bq_n, -\bv_n)  \right).
  \end{align}

  We now show that indeed $S^* \cM_0 \ll \cM_0$ and compute
  $d S^* \cM_0/ d\cM_0$ by breaking $S$ up into its constituent maps
  and making repeated usage of \eqref{RN:push:comp}.  Start by
  noticing that the mappings $\Xi_{\da}^{(1)}$ and $\Xi_{\db}^{(2)}$
  defined in \eqref{def:Xi1}-\eqref{def:Xi2} satisfy
  $(\Xi_{\delta_i}^{(i)})^*\cM_0 \ll \cM_0$, $i = 1,2$. Indeed, it is
  not difficult to check via e.g. the equivalence of characteristic
  functionals that the rotation mapping $\Xi_{\db}^{(2)}$ preserves
  the measure $\cM_0$, i.e. $(\Xi_{\db}^{(2)})^*\cM_0 = \cM_0$, so
  that in particular
\begin{align}\label{RN:Xi2}
    \frac{d (\Xi_{\db}^{(2)})^*\cM_0}{d \cM_0} (\bq, \bv) = 1 
    \quad \mbox{ for $\cM_0$-a.e. } (\bq,\bv) \in \spq \times \spq.
\end{align}
Regarding $\Xi_{\da}^{(1)}$ we denote by
$G: \spq \times \spq \to \spq$ the mapping corresponding to the second
component of $\Xi_{\da}^{(1)}$, i.e.
\begin{align*}
  G(\bq, \bv) \coloneqq \bv - \da \agp(\bq),
  \quad (\bq, \bv) \in \spq \times \spq,
\end{align*}
and thus obtain 
\begin{align*}
    (\Xi_{\da}^{(1)})^*\cM_0 (d \bq, d \bv) 
    = G(\bq, \cdot)^*\mu_0(d \bv) \mu_0 (d \bq).
\end{align*}
Now, for $\bv \sim \mu_0 = \cN(0,\cC)$ it follows that for each fixed
$\bq \in \spq$, $G(\bq, \bv) \sim \cN(- \da \agp(\bq), \cC)$. Hence,
for each $\bq \in X$,
$G(\bq, \cdot)^*\mu_0 = \text{Law}(G(\bq, \bv)) = \cN(- \da \agp(\bq),
\cC)$.  Under our assumption \eqref{eq:f:CM:reg}, namely that
$\agp(\bq) \in \mbox{Dom}(\cC^{-1/2})$, we have that $\mu_0$ and
$G(\bq, \cdot)^*\mu_0$ are mutually absolutely continuous and,
consequently, $(\Xi_{\da}^{(1)})^*\cM_0 \ll \cM_0$, with
\begin{align}\label{RN:Xi1}
    \frac{d (\Xi_{\da}^{(1)})^*\cM_0}{d \cM_0} (\bq, \bv)
    =
    \frac{d G(\bq, \cdot)^*\mu_0}{d \mu_0} (\bv) 
    = 
    \exp \left( - \da \langle \cC^{-1/2}\agp(\bq),  \cC^{-1/2}\bv \rangle
          - \frac{\da^2}{2} |\cC^{-1/2} \agp(\bq)|^2 \right) 
\end{align}
for $\cM_0$-a.e. $(\bq, \bv) \in \spq \times \spq$ (see
e.g. \cite[Theorem 2.23]{DPZ2014}).

From \eqref{RN:push:comp} and \eqref{RN:Xi2} and recalling the
notation \eqref{not:int:i}, we deduce that the mapping
$\Sstp \coloneqq \Xi_{\da}^{(1)} \circ \Xi_{\db}^{(2)} \circ
\Xi_{\da}^{(1)}$ satisfies $\Sstp^*\cM_0 \ll \cM_0$ and
\begin{align}\label{RN:Psi:mu0}
    \frac{d \Sstp^*\cM_0}{d \cM_0} (\bq, \bv) 
    =
  \frac{d (\Xi_{\da}^{(1)})^*\cM_0}{d \cM_0}
  ((\Xi_{\da}^{(1)} \circ \Xi_{\db}^{(2)})^{-1}(\bq, \bv)) 
    \frac{d (\Xi_{\da}^{(1)})^*\cM_0}{d \cM_0} (\bq, \bv), 
\end{align}
for $\cM_0$-a.e. $(\bq, \bv) \in \spq \times \spq$.  Consequently,
invoking \eqref{RN:push:comp} once again, we obtain that
$\Sol = \Sstp^n$ satisfies $\Sol^*\cM_0 \ll \cM_0$, with
\begin{align}\label{RN:Sol:mu0}
    \frac{d \Sol^*\cM_0}{d \cM_0} (\bq, \bv) 
    =
    \prod_{i=1}^n  \frac{d \Sstp^*\cM_0}{d \cM_0} ((\Sstp^{i-1})^{-1}(\bq, \bv)),
\end{align}
for $\cM_0$-a.e. $(\bq, \bv) \in \spq \times \spq$, where we set
$\Sstp^0 \coloneqq I$.

Finally, it is not difficult to check (e.g. via equivalence of
characteristic functionals) that the momentum-flip operator $R$,
defined as in \eqref{def:flip:R}, preserves the measure $\cM_0$, so
that
\[
  \frac{d R^*\cM_0}{d \cM_0}(\bq, \bv) = 1 \quad
  \mbox{ for $\cM_0$-a.e. } (\bq, \bv) \in \spq \times \spq.
\]
Therefore, with a third application of \eqref{RN:push:comp}, we deduce
that $(R \circ \Sol)^*\cM_0 \ll \cM_0$ and
\begin{align}\label{RN:R:Sol:mu0}
    \frac{d (R \circ \Sol)^*\cM_0}{d \cM_0} (\bq, \bv)
    &= 
    \frac{d \Sol^*\cM_0}{d \cM_0} (R^{-1}(\bq, \bv)) 
    \frac{d R^*\cM_0}{d \cM_0}(\bq, \bv) \notag \\
    &=
    \frac{d \Sol^*\cM_0}{d \cM_0} (R(\bq, \bv)),
\end{align}
for $\cM_0$-a.e. $(\bq, \bv) \in \spq \times \spq$, where in the last
equality we used that $R^2 = I$.

Moreover, from \eqref{RN:Sol:mu0} and \eqref{RN:R:Sol:mu0} it follows
that for $\cM_0$-a.e. $(\bq, \bv) \in \spq \times \spq$
\begin{align}\label{RN:R:Sol:mu0:b}
    \frac{d (R \circ \Sol)^*\cM_0}{d \cM_0} (\bq, \bv)
    &=
      \prod_{i=1}^n  \frac{d \Sstp^*\cM_0}{d \cM_0} ((\Sstp^{i-1})^{-1}
                             \circ R (\bq, \bv)) \notag \\
    &= \prod_{i=1}^n  \frac{d \Sstp^*\cM_0}{d \cM_0} (R \circ \Sstp^{i-1} (\bq, \bv)).
\end{align}
To justify the second equality, observe that from
\eqref{eq:Inv:stp:GHS}, \eqref{not:int:i} together with
\cref{prop:inv:R:S} and the fact that $R^2 = I$ we deduce that
$(\Sstp^{i-1})^{-1} \circ R = R \circ \Sstp^{i-1}$, for all
$i = 1, \ldots, n$. With similar logic we observe that
$(\Xi_{\da}^{(1)} \circ \Xi_{\db}^{(2)})^{-1} \circ R =
\Xi_{\da}^{(1)} \circ R \circ \Sstp$ and obtain, invoking the notation
in \eqref{not:int:i}, that for $i = 1, \ldots, n$ and
$(\bq_0, \bv_0) \in \spq \times \spq$
\begin{align*}
   (\Xi_{\da}^{(1)} \circ \Xi_{\db}^{(2)})^{-1} \circ R \circ  \Sstp^{i-1} (\bq_0, \bv_0) 
   =
   \Xi_{\da}^{(1)} \circ R \circ \Sstp^i (\bq_0, \bv_0)
   =
   \Xi_{\da}^{(1)} \circ R (\bq_i, \bv_i)
   =
   (\bq_i, - \bv_i - \da  \agp(\bq_i)).
\end{align*}
Therefore, with these observations \eqref{RN:R:Sol:mu0:b},
\eqref{RN:Psi:mu0} and then \eqref{RN:Xi1}, we obtain for $\cM_0$-a.e.
$(\bq_0, \bv_0) \in \spq \times \spq$
\begin{align*}
   &\frac{d (R \circ \Sol)^*\cM_0}{d \cM_0} (\bq_0, \bv_0)\\
    &\;=
    \prod_{i=1}^n \frac{d (\Xi_{\da}^{(1)})^*\cM_0}{d \cM_0} 
                         (\Xi_{\da}^{(1)} \circ R \circ \Sstp^i (\bq_0, \bv_0))
    \frac{d (\Xi_{\da}^{(1)})^*\cM_0}{d \cM_0} (R \circ  \Sstp^{i-1}(\bq_0, \bv_0)) \\
    &\;= \prod_{i=1}^n \exp \left( - \da \langle \cC^{-1/2}\agp(\bq_i), - \cC^{-1/2}( \bv_i + \da \agp(\bq_i))\rangle 
                                               - \frac{\da^2}{2} |\cC^{-1/2} \agp (\bq_i)|^2 \right)\\
                                                   &\qquad \qquad \qquad \qquad     
   \times \exp \left( - \da \langle \cC^{-1/2} \agp(\bq_{i-1}), - \cC^{-1/2}\bv_{i-1} \rangle  
     - \frac{\da^2}{2} |\cC^{-1/2} \agp (\bq_{i-1})|^2 \right) \\
    &\;= \exp \left( \frac{\da^2}{2} |\cC^{-1/2} \agp(\bq_n)|^2 - \frac{\da^2}{2} |\cC^{-1/2} \agp(\bq_0)|^2  
    + \da \sum_{i=1}^n \left[ \langle \cC^{-1/2} \agp(\bq_{i-1}), \cC^{-1/2} \bv_{i-1} \rangle
                                       + \langle \cC^{-1/2} \agp(\bq_i), \cC^{-1/2} \bv_i \rangle\right] \right).
\end{align*}
Thus, together with \eqref{RN:S:cM:cM0}, this identity combines with
\eqref{eq:dsMM:p1}, \eqref{eq:dsMM:p2} to conclude
\eqref{RN:RS:mu:mu0}.  The proof is now complete.
\end{proof}

\begin{Remark}\label{rmk:leapfrog:fails}
  As similarly pointed out in \cite{beskos2013advanced}, we emphasize
  that the proof of \cref{thm:rev:HMC} provides some additional
  clarification regarding the particular choice of numerical
  integrator for \eqref{ham:dyn} that allowed the derivation of
  $\infty$HMC in \cite{Beskosetal2011} as a well-defined algorithm in
  infinite dimensions. Indeed, if we choose instead the classical
  ``leapfrog'' integrator described in \eqref{def:leapfrog:int} within
  the finite-dimensional setting, and apply it to the
  infinite-dimensional dynamics \eqref{mod:ham:dyn}, we obtain the
  following splitting:
\[
  \frac{d\bq}{dt} = 0, \quad
  \frac{d\bv}{dt} = - \bq - \agp(\bq),
  \quad \mbox{ and } \quad
  \frac{d\bq}{dt} = \bv, \quad \frac{d\bv}{dt} = 0.
\]
The corresponding solution mappings for a given time $t \geq 0$ are
given respectively by
\[
    \Xi_t^{(1)}(\bq_0, \bv_0) = (\bq_0, \bv_0 - t \bq_0 - t
    \agp(\bq_0))
    \quad \mbox{ and } \quad
    \Xi_t^{(2)}(\bq_0, \bv_0) = (\bq_0 + t \bv_0, \bv_0). 
\]
Following the proof of \cref{thm:rev:HMC}, we notice that here the
measures $(\Xi_{\da}^{(1)})^*(\mu_0 \otimes \mu_0)$ and
$(\Xi_{\db}^{(2)})^*(\mu_0 \otimes \mu_0)$ can be written as
\[
  (\Xi_{\da}^{(1)})^*(\mu_0 \otimes \mu_0)(d\bq, d \bv)
  = G(\bq, \cdot)^* \mu_0 (d\bv) \mu_0(d\bq),
  \quad \mbox{ for } \,\, G(\bq, \bv) \coloneqq  \bv - t \bq - t \agp(\bq)
\]
and
\[
  (\Xi_{\db}^{(2)})^*(\mu_0 \otimes \mu_0)(d\bq, d \bv)
  = F(\cdot, \bv)^* \mu_0(d\bq) \mu_0(d\bv),
  \quad \mbox{ for } \,\, F(\bq, \bv) \coloneqq \bq + t \bv.
\]
Now, since $\mu_0 = \cN(0,\cC)$, it follows that
$G(\bq, \cdot)^* \mu_0 = \cN(- t \bq - t \agp(\bq), \cC)$ and
$F(\cdot, \bv)^* \mu_0 = \cN(t\bv, \cC)$. However,
$|\cC^{-1/2} ( \bq + \agp(\bq))|$ and $|\cC^{-1/2} \bv|$ are infinite
for $\mu_0$-a.e. $\bq \in \spq$ and $\mu_0$-a.e. $\bv \in \spq$,
respectively, and thus
$- t \bq - t \agp(\bq) \notin \mbox{Dom}(\cC^{-1/2})$ and
$t \bv \notin \mbox{Dom}(\cC^{-1/2})$ almost surely. This implies that
the Cameron-Martin formula (see \cite[Theorem 2.23]{DPZ2014}) cannot
be applied as done in \eqref{RN:Xi1}, a crucial step for the remainder
of the proof.

In summary, this shows that the validity of \cref{thm:rev:HMC} relies
on utilizing maps $\Xi_{t}^{(1)}$ and $\Xi_{t}^{(2)}$ for which their
corresponding pushforwards on the product of Gaussians, namely
$(\Xi_{t}^{(1)})^*(\mu_0 \otimes \mu_0)$ and
$(\Xi_{t}^{(2)})^*(\mu_0 \otimes \mu_0)$, with $\mu_0 = \cN(0,\cC)$,
are absolutely continuous with respect to the product
$\mu_0 \otimes \mu_0$. In \cref{thm:rev:HMC}, this is achieved with a
``rotation'' mapping, $\Xi_t^{(2)}$ as in \eqref{def:Xi2}, and a
mapping given by a ``shift'' within $\mbox{Dom}(\cC^{-1/2})$,
$\Xi_t^{(1)}$ as in \eqref{def:Xi1}.
\end{Remark}

\begin{Remark}[A generalized Langevin-type algorithm from
  \cref{thm:rev:HMC}]
  We notice that \cref{thm:rev:HMC} allows us to derive a similar
  generalization to infinite-dimensional Langevin-type algorithms for
  sampling from measures $\mu$ of the form \eqref{Gibbs:meas}, as
  introduced in \cite{beskos2008,cotter2013mcmc}. Here we only
  consider, for simplicity, the case when
  $\Vker(\bq, \cdot) = \cN(0,\cC)$, for every $\bq \in \spq$, with a
  suitable covariance operator $\cC$.

  We recall that such Langevin-type algorithms are defined for a given
  Fr\'echet differentiable potential function $\Pot: \spq \to \RR$ and
  covariance operator $\cC: \spq \to \spq$ as in \eqref{Gibbs:meas},
  through the following Langevin dynamic
\begin{align}\label{eq:Langevin:a} 
d\bq + \frac{1}{2} \cK(\cC^{-1} \bq + D \Pot(\bq)) dt =  \sqrt{\cK} dW
\end{align}
for some ``preconditioning'' operator $\cK: \spq \to \spq$ and a
cylindrical Brownian motion $W$ evolving on $\spq$ (see
e.g. \cite{DPZ2014} for the general setting of such stochastic
evolution equations). Under appropriate assumptions on the potential
$\Pot$, it follows that such dynamics holds
$\mu(d\bq) \propto e^{-\Pot(\bq)}\mu_0(d\bq)$ as an invariant measure,
and thus effective MCMC proposals can be generated by taking suitable
numerical discretizations of \eqref{eq:Langevin:a}
\cite{beskos2008,cotter2013mcmc}.

Similarly as in \eqref{mod:ham:dyn}, here we consider the following
more general dynamics
\begin{align}\label{mod:Langevin}
d\bq + \frac{1}{2} \cK\cC^{-1} \bq + \agp(\bq) dt =  \sqrt{\cK} dW,
\end{align}
for some measurable mapping $\agp: \spq \to \spq$.  
In practice we would expect $\agp$ to be taken as a suitable
approximation of $\cK D \Pot$. Choosing $\cK = \cC$ and taking a
semi-implicit Euler time discretization of \eqref{mod:Langevin} where
the linear part of the drift is approximated in a Crank-Nicolson
fashion, yields, for a fixed time step $\dT > 0$,
\begin{align}\label{disc:Langevin}
  \frac{\btq - \bq}{\dT} =
  - \frac{1}{2} \left( \frac{\bq + \btq}{2} + \agp(\bq) \right) +
  \frac{1}{\sqrt{\dT}} \bv,
  \quad \bv \sim \mu_0 = \cN(0,\cC).
\end{align}
Solving \eqref{disc:Langevin} for $\btq$ and denoting
$\ndt = (4 - \dT)/(4 + \dT)$ thus yields the following proposal map
\begin{align}\label{def:F:lolmcmc}
  F(\bq, \bv) \coloneqq \ndt \bq + \sqrt{1 - \ndt^2}
  \left( \bv - \frac{\sqrt{\dT}}{2} \agp(\bq)\right),
  \quad \bq \in \spq, \, \bv \sim \mu_0,
\end{align}
which then defines the proposal kernel
$Q(\bq, d\btq) = F(\bq, \cdot)^*\mu_0(d \btq)$.

To account for the bias introduced both replacing $\cK D \Pot$ with
the surrogate $\agp$\footnote{Here, in contrast to
  \eqref{eq:Langevin:a}, the dynamics \eqref{mod:Langevin} are not
  guaranteed to hold $\mu$ as an invariant measure in general.} and
the from numerical discretization the resulting proposal kernel $Q$
may be complemented with an appropriate accept-reject function $\arFn$
so as to yield a Markov transition kernel $P$ as in
\eqref{eq:kern:tierney} for which $\mu$ is invariant. Following the
approach by \cite{Tierney1998} recalled in \cref{sec:tierney}, such an
$\arFn$ can be obtained as in \eqref{eq:ar:tierney}, by computing the
Radon-Nikodym derivative $d \eta^\perp/ d \eta$ directly, for
$\eta(d\bq, d \btq) = F(\bq, \cdot)^*\mu_0(d \btq) \mu(d\bq)$ and
$\eta^\perp(d\bq, d \btq) = \eta(d\btq,d\bq)$, as done in
\cite{beskos2008}.

Alternatively, here we compute $d \eta^\perp/ d \eta$ by invoking
formula \eqref{RN:eta:S:inv:a} and showing that its right-hand side
follows as a particular case of \eqref{RN:RS:mu:mu0}. Indeed, notice
that for each fixed $\bq \in \spq$, the mapping
$F(\bq,\cdot): \spq \to \spq$ is invertible, with
\begin{align}\label{eq:F:inv:a}
  F(\bq, \cdot)^{-1}(\btq)
  = \frac{1}{\sqrt{1 - \ndt^2}}(\btq - \ndt \bq) + \frac{\sqrt{\dT}}{2} f (\bq).
\end{align}
Now, in view of \eqref{def:inv:S}, we compute
\begin{align}\label{eq:F:inv}
F(F(\bq, \bv), \cdot)^{-1}(\bq)
  &= \frac{1}{\sqrt{1-\ndt^2}}\left( \bq - \rho \Fmap(\bq,\bv) \right)
        + \frac{\sqrt{\dT}}{2} \agp(\Fmap(\bq,\bv))  \notag \\
  &= \frac{1}{\sqrt{1-\ndt^2}}\left[ \bq
       - \ndt\left( \ndt \bq + \sqrt{1-\ndt^2}\left( \bv
    -\frac{\sqrt{\dT}}{2} \agp(\bq)  \right) \right) \right]
    + \frac{\sqrt{\dT}}{2} \agp(\Fmap(\bq,\bv)) \notag \\
  &= \sqrt{1-\rho^2} \bq -\rho\bv + \rho \frac{\sqrt{\dT}}{2} \agp(\bq)
    + \! \frac{\sqrt{\dT}}{2} \agp (\Fmap(\bq,\bv)).
\end{align}
Thus, it follows by \cref{prop:inv:S} or direct calculation that
$S(\bq, \bv) = (F(\bq,\bv), F(F(\bq, \bv), \cdot)^{-1}(\bq))$ is an
involution in $\spq \times \spq$.

Moreover, denote $\hat{S} = R \circ S$, where $R$ is the
``momentum-flip'' map \eqref{def:flip:R}, so that
$S = R \circ \hat{S}$. Following an observation from
\cite{beskos2013advanced}, we notice that such $\hat{S}$ is the
particular one-step case of \eqref{Verlet:int} regarding
\cref{alg:method} for specific choices of $\da$, $\db$. Namely, taking
$\da := \sqrt{\dT}/2$ and $\db := \cos^{-1}\ndt$, we have
\begin{align}\label{split:inf:Lang}
  S = R \circ \hat{S}, \quad \hat{S}
  = \Xi_{\da}^{(1)} \circ \Xi_{\db}^{(2)} \circ \Xi_{\da}^{(1)},
\end{align}
where $\Xi_t^{(1)}$ and $\Xi_t^{(2)}$ are as defined in
\eqref{def:Xi1} and \eqref{def:Xi2}, respectively. It thus follows
from \eqref{RN:RS:mu:mu0} with $\cM = \mu \otimes \mu_0$ that for
$(\mu \otimes \mu_0)$-a.e. $(\bq, \bv) \in \spq \times \spq$
\begin{align}
&\frac{d S^* (\mu \otimes \mu_0)}{d (\mu \otimes \mu_0)} (\bq, \bv)
    =  \exp  \left(   \Pot(\bq) - \Pot(\btq) -
                \frac{\dT}{8} \left[ |\cC^{-1/2} \agp(\bq)|^2
                - |\cC^{-1/2} \agp(\btq)|^2 \right] \right.
\notag \\
&\qquad \qquad \qquad \qquad \qquad \qquad \qquad
            \left. + \frac{\sqrt{\dT}}{2}
            \left[ \langle \cC^{-1/2} \bv, \cC^{-1/2} \agp(\bq)\rangle
            + \langle \cC^{-1/2} \btv, \cC^{-1/2} \agp(\btq) \rangle \right] \right),
\label{RN:S:mu:mu0:lolmcmc}
\end{align}
where
$(\btq, \btv) = \Xi_{\da}^{(1)} \circ \Xi_{\db}^{(2)} \circ
\Xi_{\da}^{(1)} (\bq, \bv) = R \circ S (\bq, \bv)$, so that
\begin{align}\label{def:q1:v1}
  \btq = F(\bq, \bv), \quad \mbox{ and }
  \btv = - F(F(\bq, \bv), \cdot)^{-1}(\bq),
\end{align}
with $F$ as defined in \eqref{def:F:lolmcmc}. Plugging $\btv$ as given
in \eqref{def:q1:v1} and \eqref{eq:F:inv} into
\eqref{RN:S:mu:mu0:lolmcmc}, we obtain after rearranging terms that
for $(\mu \otimes \mu_0)$-a.e. $(\bq, \bv) \in \spq \times \spq$
\begin{multline}\label{RN:lolmcmc}
\frac{d S^* (\mu \otimes \mu_0)}{d (\mu \otimes \mu_0)} (\bq, \bv)
= 
\exp \left(  \Pot(\bq) - \Pot(\btq) - \frac{\dT}{8} 
\left[ |\cC^{-1/2} \agp(\bq)|^2 + |\cC^{-1/2} \agp(\btq)|^2 \right]
+ \frac{\sqrt{\dT}}{2} \langle \cC^{-1/2} \bv, \cC^{-1/2} \agp(\bq)\rangle \right. \\
\left. - \frac{\sqrt{\dT}}{2} \left\langle \cC^{-1/2} (\sqrt{1 -
      \ndt^2} \bq - \ndt \bv) 
+ \ndt \frac{\sqrt{\dT}}{2} \cC^{-1/2} \agp(\bq) , \cC^{-1/2} \agp(\btq) \right\rangle \right).
\end{multline}
In view of \eqref{RN:eta:S:inv}, we plug
$\bv = F(\bq, \cdot)^{-1}(\btq)$ as given in \eqref{eq:F:inv:a} into
\eqref{RN:lolmcmc} to obtain that
\begin{align}\label{alpha:lolmcmc}
  \arFn(\bq, \btq) \coloneqq
  1 \wedge \frac{d S^* (\mu \otimes \mu_0)}
  {d (\mu \otimes \mu_0)} (\bq, F(\bq, \cdot)^{-1}(\btq)) 
= 1 \wedge \frac{\beta(\btq, \bq)}{\beta(\bq, \btq)}, 
\end{align}
where 
\[
  \beta(\bq, \btq) \coloneqq
  \exp \left( - \Pot(\bq) - \frac{\dT}{8} |\cC^{-1/2} \agp(\bq)|^2
    - \frac{\sqrt{\dT}}{2} 
\left\langle \frac{\cC^{-1/2}(\btq - \rho \bq)}{\sqrt{1 - \rho^2}}, 
       \cC^{-1/2} \agp(\bq) \right\rangle \right),
\]
for $\eta$-a.e. $(\bq, \btq) \in \spq \times \spq$, where we recall
that $\eta(d \bq, d\btq) = F(\bq, \cdot)^*\mu_0(d \btq) \mu(d\bq)$. We
notice that formula \eqref{alpha:lolmcmc} for the accept-reject
function concurs with \cite{beskos2008,cotter2013mcmc} in the
particular case $\agp = \cC D \Pot$.
\end{Remark}

\section{Connections with the classical algorithms}
\label{sec:class:ex}

In this final section we make explicit how a variety of established
MCMC method can be derived from \cref{thm:gen:rev:new}.  Here it is
notable and interesting that many of the algorithms discussed here can
be derived by multiple, \emph{non-equivalent}, applications of
\cref{thm:gen:rev:new}; namely different choices of $S$, $\Vker$ in
\eqref{eq:ext:MH:ker:diff} may ultimately lead to the same sampling
kernel.  In fact the machinery introduced in \cref{sec:tierney},
\cref{sec:approx:ham:methods} plays an important role here,
identifying and/or clarifying various connections between a number of
popular sampling methods.

First in \cref{sec:con:class:alg:fd} we consider methods defined for
sampling continuous distributions on $\RR^\fd$.  Further on in
\cref{subsec:app:Hilb} we consider the Hilbert space methods
introduced more recently in e.g. \cite{beskos2008,cotter2013mcmc,
  Beskosetal2011}.

\subsection{Finite-dimensional methods}
\label{sec:con:class:alg:fd}

For all examples in this subsection, we assume a finite-dimensional
state space $\spq = \RR^\fd$, and consider a continuously distributed
target measure $\mu(d \bq) = p(\bq) d\bq$.  We often write $\mu$ in
the potential form \eqref{eq:trg:Msr:fnt:dim} convenient for
Hamiltonian methods, i.e.  $\mu(d\bq) \propto e^{-\UPot(\bq)} d\bq$
for some suitably regular potential function $\UPot: \RR^\fd \to \RR$.

\subsubsection{The classical Metropolis-Hastings formulation}

Start first with the situation when we consider a general continuously
distributed Markov kernel $Q(\bq, d \btq) := q(\bq, \btq) d \btq$,
where $q: \RR^{\fd} \times \RR^{\fd} \to \RR^+$ is a measurable mapping such that
$\int q( \bq, \btq) d \btq = 1$ for any $\bq \in \RR^\fd$.   Recalling
the identity \eqref{eq:RN:Der:form}, this is a special case of 
Tierney's formulation \cite{Tierney1998}.  Therefore, precisely
as a special case of \cref{rmk:flip:it:baby}, we derive the classical
Hastings ratio \cite{hastings1970monte}, i.e. \eqref{eq:hastings:A:R},
from \cref{thm:gen:rev:new} as follows: take $S(\bq,\bv) = (\bv, \bq)$
which is clearly a volume preserving involution.  Set $\Vker = Q$ so
that $\cM(d \bq, d \bv) = p(\bq) q(\bq,\bv) d\bq d\bv$.  Thus,
comparing \eqref{def:arFn:b} with \eqref{eq:push:forward:den} we
obtain \eqref{eq:hastings:A:R}.  Here, strictly speaking, we would
need that both $p > 0$ and $q > 0$ almost everywhere to obtain the
condition \ref{P2c} in \cref{thm:gen:rev:new}, but see
\cref{rem:noabscont}.  

Note that, with this choice for $S$, $\Vker$,
we may view \cite{hastings1970monte} as a special case of the
Metropolis-Hastings-Green algorithm, \cite{green1995reversible},
discussed above in \cref{rmk:alg:Rn:Preliminary}. 
Also, it is amusing to observe that 
the classical Hastings ratio, \eqref{eq:hastings:A:R}, falls out as special case of
\cref{alg:gHMC:FD:1}.  To see this write the proposal density $q$ in
the potential form
$q(\bq,\bv) = Z_\VPot(\bq)^{-1} e^{-\VPot(\bq,\bv)}$ and set
$\Vker = Q$.  In \eqref{def:leapfrog:int}, we select $\fu(\bv) = \bv$,
$\fv(\bq) = - \bq$ and set $\da = \db = 1$.  Similar connections
between each of the other algorithms in \cref{sec:approx:ham:methods}
and the acceptance probability \eqref{eq:hastings:A:R} can be drawn as well.

\subsubsection{Random Walk Monte Carlo}

Perhaps the simplest Metropolis-Hastings method is the Random Walk
Monte Carlo (RWMC) algorithm: Take a proposal map
$F: \RR^\fd \times \RR^\fd \to \RR^\fd$ given by
\begin{equation}
F(\bq, \bv) := \bq + \bv, \quad \bv \sim \msv,
\label{eq:RWMC:prop}
\end{equation}
where
$\msv(d \btq) = q(\btq) d\btq = Z^{-1}_\VPot e^{- \VPot(\btq)}d \btq$
is some jumping probability distribution on $\RR^\fd$. With this
definition, \eqref{eq:push:forward:den} and noting that
$F(\bq, \cdot)^{-1}(\btq) = \btq - \bq$ for any fixed
$\bq \in \RR^\fd$, the associated proposal kernel takes the form
\begin{align}
  \label{eq:rwmc:tf}
  Q(\bq, d\btq) :=F(\bq, \cdot)^* \nu(d \btq)
  = q(\btq - \bq) d\btq
  = Z^{-1}_\VPot e^{- \VPot(\btq- \bq)}d \btq.
\end{align}
Clearly this case falls under the rubric of the usual Hastings ratio,
\eqref{eq:hastings:A:R}, which translates to
\begin{align}
    \label{eq:rwmc:tf:ar}
  \alpha(\bq, \btq)
  := 1 \wedge \frac{p(\btq) q(\bq - \btq)}{p(\bq) q(\btq -\bq) }
  = 1 \wedge \exp( \UPot(\bq) - \UPot(\btq) + \VPot(\btq- \bq)
  - \VPot(\bq - \btq));
\end{align}
cf., e.g., \cite[Equation
3.38]{kaipio2005statistical} or \cite[Equation
11.2]{gelman2014bayesian}.
Typically, $\nu$ is taken to be a zero-mean
gaussian distribution $\mathcal{N}(0,\cC)$, for some covariance matrix
$\cC$ on $\RR^{\fd \times \fd}$.  Note that, in this particular
symmetric case and in general, the terms involving $q$ cancel in
\eqref{eq:hastings:A:R} which is the classical Metropolis case.

Let us observe that the machinery in \cref{subsec:inv:F} reveals a 
slightly different $S$ and
$\Vker$ which yield the same RWMC algorithm from
\cref{thm:gen:rev:new}.
Comparing \eqref{eq:rwmc:tf} with \eqref{def:inv:S} 
and using that $F(\bq, \cdot)^{-1}(\btq) = \btq - \bq$ we obtain
the involution $\Smap(\bq,\bv):=(\bq+\bv,-\bv)$ and then set
$\Vker(\bq, d \bv) := \nu(d\bv)$.  Thus, according to
\eqref{eq:Gen:HMC:AR}, it follows that
\begin{align*}
  \harFn(\bq, \bv)
  = 1 \wedge \frac{d \Smap^*(\msq \otimes \msv)}{d (\msq \otimes \msv)} (\bq, \bv)
  = 1 \wedge \frac{p(\bq+\bv) q(- \bv)}{p(\bq) q(\bv) }
  = 1 \wedge \exp( - \UPot(\bq + \bv) - \VPot(-\bv) + \UPot(\bq) + \VPot(\bv)),
\end{align*}
which thus recovers $\arFn$ by invoking \eqref{RN:eta:S:inv}. Finally
note that this choice of $S$ and $\Vker$ turns out to be a special
case of \cref{alg:gHMC:FD:1}. To see this take $\fu := 0$ (and/or
$\delta_1 = 0$), set $\fv( \bv):= \bv$, $\delta_2 =1$ and choose
$n=1$. According to \eqref{def:leapfrog:int}, we have $S = R \circ \cSol$ with
$R$ the usual momentum-flip involution as in \eqref{def:flip:R}.

\subsubsection{The MALA Algorithm}

As introduced in \cite{besag1994comments,roberts1996exponential}, the
MALA algorithm uses as its starting point a numerical resolution of
the Langevin dynamic
\begin{align}
\label{eq:LA:fin:dim}
d \bq + \frac{1}{2} \nabla \UPot( \bq) dt = dW,
\end{align}
so chosen in order that the target measure
$\mu(d\bq) \propto e^{-\UPot(\bq)} d\bq$ is an invariant.  Here $W$ is an
$\fd$-dimensional Brownian motion so that \eqref{eq:LA:fin:dim} is an
It\^{o} stochastic differential equation evolving on $\RR^\fd$.  By
taking an explicit Euler numerical discretization of
\eqref{eq:LA:fin:dim} with time step $\delta^2$, we obtain the proposal
kernel $Q(\bq, d \btq) := F(\bq,\cdot )^* \nu( d \btq)$ where
$\nu(d \bv) = (2 \pi )^{-\fd/2} e^{-\frac{1}{2}|\bv|^2} d\bv $,
i.e. normally distributed with unit covariance, and
\begin{align}
F(\bq, \bv ) :=
\bq - \frac{\delta^2}{2}\nabla \UPot( \bq)
+\delta  \bv.
\label{eq:MALA:prop:comp}
\end{align}
A corresponding acceptance probability $\arFn$ can be directly obtained as in 
\eqref{eq:ar:tierney}, i.e. $\arFn(\bq, \btq) = 1 \wedge d \eta^\perp/d \eta(\bq, \btq)$, with
\[
	\eta(d\bq, d \btq) = F(\bq, \cdot)^*\nu(d\btq) \mu(d\bq) 
	\propto \exp \left( - \UPot( \bq) - \frac{1}{2\delta^2} \left|\btq - \bq
	+ \frac{\delta^2}{2}\nabla \UPot( \bq) \right|^2 \right)
	d \bq d \btq,
\]
so that 
\begin{align}
\arFn(\bq, \btq) =
1 \wedge
\exp\left( -\UPot( \btq)
- \frac{1}{2\delta^2} | \bq - \btq
+ \frac{\delta^2}{2}\nabla  \UPot(\btq)|^2
+ \UPot( \bq)
+ \frac{1}{2\delta^2} |\btq - \bq
+ \frac{\delta^2}{2}\nabla \UPot( \bq)|^2 \right).
\label{eq:AR:MALA}
\end{align}

Alternatively, we notice that under this setting the MALA algorithm
can also be seen as a particular case of \cref{alg:method:abs}
associated to \cref{thm:gen:rev:new}, by defining an involution map
$S$ as in \cref{prop:inv:S} and invoking \cref{thm:Tierney},
(ii). Indeed, observe that, for each fixed $\bq$, $F(\bq, \cdot)$ is
invertible in $\bv$ with
\begin{align*}
F(\bq, \cdot)^{-1}(\btq)
= \frac{1}{\delta} \left(\btq - \bq
+ \frac{\delta^2}{2}\nabla \UPot( \bq)\right).
\end{align*}
Invoking \cref{prop:inv:S}, one obtains the desired involution $S$
from \eqref{def:inv:S} as
\begin{align}
S(\bq, \bv)
= \left( \bq - \frac{\delta^2}{2}\nabla \UPot( \bq) +\delta  \bv,
\frac{\delta}{2}\nabla \UPot( \bq) -  \bv +
\frac{\delta}{2}
\nabla\UPot\left(  \bq - \frac{\delta^2}{2}\nabla \UPot( \bq) +\delta  \bv\right)\right).
\label{eq:MALA:involution}
\end{align}
The acceptance probability $\harFn$ from \cref{thm:gen:rev:new},
i.e.
$\harFn(\bq, \bv) = 1 \wedge d S^*(\mu \otimes \nu)/d (\mu \otimes
\nu)(\bq, \bv)$, is given particularly in this finite-dimensional
context by formula \eqref{eq:Gen:HMC:AR}, with
$p(\bq) \propto e^{-\UPot(\bq)}$ and
$q(\bq,\bv) \propto e^{-\frac{1}{2}|\bv|^2}$. However, computing
$\nabla S$ directly from the definition in \eqref{eq:MALA:involution}
might turn out to be an impractical task.

Instead, we notice that MALA is a special case of \cref{alg:gHMC:FD:1}
and that $S$ in \eqref{eq:MALA:involution} can be written as
$R \circ \cSol_{1,\dT}$, where $R$ is the velocity-flip involution
$R(\bq, \bv) = (\bq, - \bv)$ and $\cSol_{1,\dT}$ is the numerical
integrator defined in \eqref{def:leapfrog:int} with $n=1$,
$2\da = \db = \delta$ and applied to the Hamiltonian dynamic
\eqref{eq:Ham:dyn:comp:fm} with canonical $J$ as in \eqref{eq:cls:J}
in $\RR^{2\fd}$ corresponding to the separated Hamiltonian
$\Ham(\bq, \bv) = \UPot(\bq) + \frac{1}{2} |\bv|^2$,
$(\bq, \bv) \in \RR^{2\fd}$, leading us to choose $\fu(\bv) = \bv$ and
$\fv(\bq) = -\nabla \UPot(\bq)$.  It thus follows from
\cref{cor:sep:dyn} (or alternatively \cref{prop:Ham:props:gen}, (i))
that $S$ is volume-preserving, so that
$|\det \nabla S(\bq, \bv)| = 1$, for all $(\bq, \bv) \in \RR^{2\fd}$.

Now from \eqref{eq:Gen:HMC:AR}, or equivalently \eqref{eq:ar:let:it:flip}, we have
\begin{align*}
\harFn(\bq, \bv) = 1 \wedge e^{-\Ham(S(\bq,\bv)) + \Ham(\bq, \bv)}=
1 \wedge
\exp\left( - \UPot( F(\bq, \bv))
- \frac{1}{2} |F(F(\bq,\bv), \cdot)^{-1}(\bq)|^2
+\UPot(\bq) + \frac{1}{2} |\bv|^2\right).
\end{align*}
Invoking \eqref{RN:eta:S:inv}, we thus recover $\arFn$ as in
\eqref{eq:AR:MALA}. This shows that in this setting MALA is a
particular case of \cref{alg:method:abs}, or equivalently
\cref{alg:gHMC:FD:1}, corresponding to the involution $S$ as given in
\eqref{eq:MALA:involution} and the kernel
$\Vker(\bq, d\bv) = \nu(d\bv) \propto e^{-\frac{1}{2}|\bv|^2} d\bv$.

\subsubsection{Hamiltonian Monte Carlo}\label{subsubsec:fd:HMC}

In the classical HMC algorithm from
\cite{duane1987hybrid,neal2011mcmc}, one takes $\VPot$ in the
Hamiltonian \eqref{eq:Ham} as in \eqref{eq:typ:sample:pot},
i.e. $\VPot(\bv) = \frac{1}{2} \langle M^{-1} \bv, \bv\rangle$, for a
so that the Markov kernel $\Vker(\bq,\cdot)$ is the
($\bq$-independent) $\RR^\fd$-valued gaussian $\nu = \cN(0,M)$. Here,
$M$ is a positive-definite mass matrix in $\RR^{\fd \times \fd}$ which
is an algorithmic parameter to be chosen. The corresponding
Hamiltonian dynamics is given as in \eqref{eq:Ham:dyn:comp:fm} with
the standard choice of matrix $J$ from \eqref{eq:cls:J}, thus written
as
\begin{align}\label{eq:trad:ham}
	\frac{d\bq}{dt} = M^{-1}\bv, \quad \frac{d\bv}{dt} = - \nabla \UPot(\bq).
\end{align}
A typical choice of numerical integrator $\Sol$ for such dynamics is
given by the leapfrog integrator
$\cSol_{n,\dT} = ( \Xi^{(1)}_{\dT/2} \circ \Xi^{(2)}_{\dT} \circ
\Xi^{(1)}_{\dT/2} )^n$ defined as in \eqref{def:leapfrog:int} for a
given time step $2 \da = \db = \dT > 0$, number of iterations
$n \in \NN$ and taking
$\fu(\bv) = M^{-1} \bv, \fv(\bq) = - \nabla \UPot(\bq)$.  Here note
that $T := \dT \cdot n$ is understood as the total integration time of
our numerical approximation of \eqref{eq:trad:ham}.  A direct
calculation (or see \cref{prop:Ham:props:gen}, \cref{cor:sep:dyn})
yields that each of the solution maps $\Xi^{(1)}_{\dT/2}$ and
$\Xi^{(2)}_{\dT}$ are reversible with respect to the ``momentum''-flip
involution $R(\bq, \bv) = (\bq, -\bv)$. Hence, by \cref{prop:inv:R:S},
(i), $\Sol_{n,\dT}$ is reversible with respect to $R$, so that
$S \coloneqq R \circ \cSol_{n,\dT}$ is an involution. The classical
HMC algorithm thus follows as a special case of \cref{alg:gHMC:FD:1}
under these choices.

Other more recent versions of HMC can also be seen to be special cases
of \cref{alg:gHMC:FD:1}. For example, in the relativistic HMC method
from \cite{lu2017relativistic} the kinetic portion $\VPot$ of the
Hamiltonian in \eqref{eq:Ham} is given as
\begin{align}
\label{eq:rel:mom:pot}
\VPot(\bv) = m c^2 \left( \frac{| \bv |^2}{m^2c^2} +1  \right)^{1/2}
\end{align}
up to the user-determined algorithmic parameters $m,c >0$,
In the physical analogy drawn here these parameters correspond to ``mass'' and
the ``speed of light'', respectively. To connect back to \cref{alg:gHMC:FD:1}
we keep the same choices for
$\Vker$, $R$ as the classical case and we now take $\fu(\bv) = \VPot(\bv)$
as in \eqref{eq:rel:mom:pot} leaving $\fv(\bq) = - \nabla \UPot(\bq)$ 
to once again define $\Sol= ( \Xi^{(1)}_{\dT/2} \circ \Xi^{(2)}_{\dT} \circ
\Xi^{(1)}_{\dT/2} )^n$ according to the associated leap-frog steps $\Xi^{(1)}$,
$ \Xi^{(2)}$.

Another version is the Riemannian manifold Hamiltonian Monte Carlo
(RMHMC) introduced in \cite{GirolamiCalderhead2011}. Here one
considers $\VPot$ in \eqref{eq:Ham} as the negative log-density of a
normal distribution with ``position''-dependent covariance
matrix. More specifically,
\[
	\VPot(\bq, \bv) = \frac{1}{2} \langle M(\bq)^{-1} \bv, \bv \rangle, 
\]
so that $\Vker(\bq,\cdot) = \cN(0, M(\bq))$. The corresponding
Hamiltonian dynamics is written as in \eqref{eq:Ham:dyn:comp:fm} with
$J$ taken in the canonical form as in \eqref{eq:cls:J}. The standard
numerical integrator in this case is given by the implicit mapping
$\Sol_{n,\dT}$ defined in \eqref{def:Sol:fd}, for given choices of
time step $\dT >0$ and number of steps $n \in \NN$ and $\fu, \fv$
taken as equalities in \eqref{eq:apx:SD:2} and note that here
$Z_{\VPot}(\bq) = \frac{1}{2}\ln(2\pi) + \frac{\fd}{2} \ln (\det
M(\bq))$. From \cref{prop:HMC:Leap:Frog}, it follows that
$\Sol_{n,\dT}$ is symplectic (and hence volume preserving) as well as
reversible with respect to the momentum-flip involution $R$ in
\eqref{def:flip:R}. Therefore, the RMHMC method follows as a special
case of \cref{alg:gHMC:FD:3} under this setting.

\subsection{Hilbert space methods}
\label{subsec:app:Hilb}

As in the setting from \cref{sec:HMC:inf:dim}, here we consider
examples where $\spq$ is a separable Hilbert space, and the target
measure $\mu$ is of the form \eqref{Gibbs:meas}. Namely,
$\mu(d\bq) \propto e^{-\Pot(\bq)}\mu_0(d\bq)$, with
$\mu_0 = \cN(0,\cC)$, for a suitable potential function
$\Pot: \spq \to \RR$ and covariance operator $\cC: \spq \to \spq$
which is symmetric, positive and trace-class.  In particular
we will see that all of the algorithms presented in this subsection
fall out as special case of \cref{alg:method} under suitable parameter
choices.

\subsubsection{The Preconditioned Crank-Nicolson (pCN) algorithm}
\label{subsubsec:pCN}

The preconditioned Crank-Nicolson MCMC algorithm
\cite{beskos2008,cotter2013mcmc} is derived from the dynamics
\eqref{mod:Langevin} in the particular case when $\agp = 0$ and
$\cK = \cC$, so that
\[
d\bq = - \frac{1}{2} \bq dt + \sqrt{\cC} d W,
\]
which defines an Ornstein-Uhlenbeck process, and holds
$\mu_0 = \cN(0,\cC)$ as an invariant measure.

Following a similar derivation as in
\eqref{disc:Langevin}-\eqref{def:F:lolmcmc}, the pCN proposal map is
written as
\begin{align}\label{def:F:pCN}
  F(\bq, \bv) = \ndt \bq + \sqrt{1 - \ndt^2} \bv,
  \quad \bq \in \spq, \, \bv \sim \mu_0 = \cN(0,\cC),
\end{align}
namely we can write the proposal kernel in the form \eqref{ass:kern:tierney}
with this $F$ and $\Vker(\bq,d\bv) = \mu_0(d\bv)$.
Here we recall that $\ndt \coloneqq (4 - \dT)/(4 + \dT)$ for a fixed
time step $\dT > 0$.

As in \cite{Tierney1998}, a suitable accept-reject function is given
as $\arFn(\bq, \btq)= 1 \wedge d \eta^\perp/ d \eta(\bq, \btq)$, for
$\bq, \btq \in \spq$, where
$\eta(d \bq, d\btq) = F(\bq, \cdot)^*\mu_0(d\btq)\mu(d\bq)$ and
$\eta^\perp(d\bq, d\btq) = \eta(d\btq, d \bq)$, with $F$ as in
\eqref{def:F:pCN}, $\mu_0 = \cN(0,\cC)$ and $\mu$ as in
\eqref{Gibbs:meas}. Such $\arFn$ can be computed by noticing that
\[
  \frac{d\eta^\perp}{d \eta}(\bq, \btq) = \exp[\Pot(\bq) - \Pot(\btq)]
  \frac{d \eta_0^\perp}{d \eta_0}(\bq, \btq),
\]
where
$\eta_0(d\bq, d \btq) \coloneqq F(\bq, \cdot)^*\mu_0(d\btq)
\mu_0(d\bq)$. Further, it is not difficult to check, e.g. via
equivalence of characteristic functionals, that
$\eta_0^\perp = \eta_0$, so that
$d \eta_0^\perp/d\eta_0 (\bq, \btq)= 1$ for $\eta_0$-a.e.
$(\bq, \btq) \in \spq \times \spq$ and thus \be\label{arFn:pCN}
\arFn(\bq, \btq) = 1 \wedge \exp[\Pot(\bq) - \Pot(\btq)] \ee
(cf. e.g. \cite{beskos2008,cotter2013mcmc,beskos2017geometric}).

As an alternate route, $\arFn$ can also be computed by invoking
formula \eqref{RN:eta:S:inv} with $S$ given by \cref{prop:inv:S}, as
done in \eqref{alpha:lolmcmc}. Indeed, taking $f = 0$ in
\eqref{alpha:lolmcmc} yields immediately \eqref{arFn:pCN}. Regarding
the connection to \cref{alg:method}, from
\eqref{split:inf:Lang}, here we notice that the associated involution
$S$ derived from \eqref{def:inv:S} can be written as
$S = R \circ \Xi_{\db}^{(2)}$, where $R$ is the momentum-flip operator
\eqref{def:flip:R}, $\db > 0$ is such that $\ndt = \cos \db$, and
$\Xi_t^{(2)}$ is as defined in \eqref{def:Xi2}.

\subsubsection{The infinite-dimensional Metropolis-adjusted
  Langevin algorithm ($\infty$MALA)}
\label{subsubsec:infMALA}

The infinite-dimensional Metropolis-adjusted Langevin algorithm
($\infty$MALA) \cite{beskos2008,cotter2013mcmc} is derived from the
Langevin dynamic \eqref{eq:Langevin:a} via the numerical
discretization \eqref{disc:Langevin}. In this case, for a fixed time
step $\dT > 0$, it follows from \eqref{def:F:lolmcmc} with
$\agp = \cC D \Pot$ that the proposal map is given by
\begin{align}\label{def:F:mala}
  F(\bq, \bv) \coloneqq \ndt \bq + \sqrt{1 - \ndt^2}
  \left( \bv - \frac{\sqrt{\dT}}{2} \cC D \Pot(\bq)\right),
  \quad \bq \in \spq, \, \bv \sim \mu_0,
\end{align}
where we recall that $\ndt = (4 - \dT)/(4 + \dT)$. Together with the
decomposition in \eqref{split:inf:Lang}, it thus follows that
$\infty$MALA is the particular case of \cref{alg:method} with the
specific choices $\HmT = 0$ (i.e. $\Vker(\bq, \cdot) = \mu_0$ for all
$\bq \in \spq$), $\agp = \cC D \Pot$, $\da := \sqrt{\dT}/2$,
$\db := \cos^{-1}\ndt$, and $n = 1$.

Moreover, from \eqref{alpha:lolmcmc} it follows that an accept-reject
function for $\infty$MALA is given by
\begin{align}\label{arFn:mala}
    \arFn(\bq, \btq) = 1 \wedge \frac{\beta(\btq, \bq)}{\beta(\bq, \btq)},
\end{align}
where
\[
    \beta(\bq, \btq) \coloneqq \exp \left( - \Pot(\bq) - \frac{\dT}{8}
      |\cC^{1/2} D \Pot(\bq)|^2
      - \frac{\sqrt{\dT}}{2}
      \left\langle \frac{\btq - \rho \bq}{\sqrt{1 - \rho^2}},
        D \Pot(\bq) \right\rangle \right),
\]
for $\eta$-a.e. $(\bq, \btq) \in \spq \times \spq$. Here,
$\eta(d\bq, d \btq) = F(\bq, \cdot)^*\mu_0(d\btq)\mu(d\bq)$ with $F$
as given in \eqref{def:F:mala}.

\subsubsection{The infinite-dimensional Hamiltonian Monte Carlo
  ($\infty$HMC) algorithm}
\label{subsubsec:pHMC}

The infinite-dimensional Hamiltonian Monte Carlo algorithm introduced
in \cite{Beskosetal2011} is the particular case of \cref{alg:method}
when $\HmT = 0$ (i.e. $\Vker(\bq, \cdot) = \mu_0$ for all
$\bq \in \spq$), $\agp = \cC D \Pot$, and when the discretization
parameters $\da$, $\db$ from \eqref{Verlet:int} are defined as
$\da := \dT/2$ and $\db := \dT$ for some fixed time step $\dT >
0$. From \eqref{RN:RS:mu:mu0}, it thus follows that an acceptance
probability in this case is given by
\begin{align*}
  \harFn(\bq_0, \bv_0)
  &= 1 \wedge \frac{d (R \circ \hat{S})^*(\mu \otimes \mu_0)}
    {d (\mu \otimes \mu_0)} (\bq_0, \bv_0) \\
    &= 1 \wedge \exp\left( \Pot(\bq_0) - \Pot(\bq_n)
      - \frac{\dT^2}{8} \left[ |\cC^{1/2} D \Pot(\bq_0)|^2
      - |\cC^{1/2} D\Pot(\bq_n)|^2 \right] \right.
    \\
  & \qquad \qquad \qquad \left.
    + \dT \sum_{i=1}^{n-1} \langle \bv_i, D\Pot(\bq_i) \rangle 
    + \frac{\dT}{2} \left[ \langle \bv_0, D\Pot(\bq_0) \rangle
    + \langle \bv_n, D \Pot (\bq_n) \rangle  \right] \right),
\end{align*}
where $R$ is the ``momentum-flip'' map \eqref{def:flip:R}, and we
invoked similar notation as in \eqref{not:int:i}.

A geometric version of this $\infty$HMC algorithm that considers a
certain position-dependent kinetic energy as in \eqref{eq:non:sep:Ham}
was developed in \cite{beskos2017geometric}. Specifically, the kernel
$\Vker(\bq,\cdot)$ at each $\bq \in \spq$ is taken to be a Gaussian
$\cN(0,\cK(\bq))$, for a position-dependent covariance operator
$\cK(\bq):\spq \to \spq$ that is assumed to be trace-class, symmetric
and strictly positive definite. Here one assumes additionally that
$\text{Im}(\cK(\bq)^{1/2}) = \text{Im}(\cC^{1/2})$, and that
$ (\cC^{-1/2}\cK(\bq)^{1/2})(\cC^{-1/2}\cK(\bq)^{1/2})^* - I$ is a
Hilbert-Schmidt operator on $\spq$, for $\mu_0$-a.e. $\bq \in
\spq$. Under these assumptions, it follows from the Feldman-Hajek
theorem that $\Vker(\bq, \cdot) = \cN(0, \cK(\bq))$ and
$\mu_0 = \cN(0,\cC)$ are mutually absolutely continuous for
$\mu_0$-a.e. $\bq \in \spq$ (see e.g. \cite{DPZ2014}), so that $\HmT$
in \eqref{eq:M:gen:inf:form} is well-defined. The authors then
consider the following surrogate dynamic from the full corresponding
Hamiltonian system
\begin{align*}
  \frac{d\bq}{dt} = \bv, \quad \frac{d\bv}{dt}
  = - \cK(\bq) [\cC^{-1} \bq + D \Pot(\bq)],
\end{align*}
which is thus a particular case of the general system
\eqref{mod:ham:dyn} with
\begin{align}\label{agp:geom:inf:HMC}
    \agp(\bq) = \cK(\bq) [(\cC^{-1} - \cK^{-1}(\bq)) \bq + D \Pot(\bq)].
\end{align}
Due to the fixed assumptions on $\cK(\bq)$ and $\cC$, it follows that
such $\agp$ satisfies \eqref{eq:f:CM:reg}. Proceeding with a numerical
splitting as in \eqref{split:ham:1}-\eqref{split:ham:2},
\eqref{Verlet:int}, we obtain that this geometric version of the
$\infty$HMC algorithm is a particular case of \cref{alg:method} with
the choices $\Vker(\bq, \cdot) = \cN(0,\cK(\bq))$, $\agp$ as in
\eqref{agp:geom:inf:HMC}, $\da = \dT$ and $\db = \dT/2$ for a given
time step $\dT > 0$, and any number of iterative steps $n \in \NN$.

\section*{Acknowledgements}

Our efforts are supported under the grants DMS-1816551 (NEGH) and
DMS-2009859 (CFM).  This work was partially conceived over the course
of two research visits: by CFM to New Orleans, LA in January 2020 and
by NEGH and CFM to Blacksburg, VA in February 2020.  We are grateful
to the Tulane and Virginia Tech math departments for hosting these
productive visits.  

We would like to thank a number of our colleagues for extensive and
invaluable feedback surrounding this work: Geordie Richards for
additional references concerning recent developments in the study of
Gibbs measures for dispersive PDEs, Andrew Holbrook for extensive
references from the machine learning and statistics literature and
specifically for pointing us to the Metropolis-Hastings-Green
algorithm as well as a body of recent work on surrogate trajectory
methods, and Gideon Simpson for suggesting the extension of
\cref{thm:gen:rev:new} to include a more general proposal space $\spv$
in the formulation found in the second draft of this manuscript.
Finally we would like to thank Radford Neal, Michael Betancourt,
Shiwei Lan and Christophe Andrieu for further inspiring conversations
that helped us contextualize this work.

\begin{footnotesize}
\addcontentsline{toc}{section}{References}
\bibliographystyle{alpha}
\bibliography{./bib}
\end{footnotesize}

\newpage
\begin{multicols}{2}
\noindent
Nathan E. Glatt-Holtz\\ {\footnotesize
Department of Mathematics\\
Tulane University\\
Web: \url{http://www.math.tulane.edu/~negh/}\\
Email: \url{negh@tulane.edu}} \\[.2cm]

\noindent Justin Krometis\\
{\footnotesize
Advanced Research Computing\\
Virginia Tech\\
Web: \url{https://www.arc.vt.edu/justin-krometis/}\\
Email: \href{mailto:jkrometis@vt.edu}{\nolinkurl{jkrometis@vt.edu}}} \\[.2cm]

\columnbreak

\noindent Cecilia F. Mondaini \\
{\footnotesize
  Department of Mathematics\\
  Drexel University\\
  Web: \url{https://drexel.edu/coas/faculty-research/faculty-directory/mondaini-cecilia/}\\
  Email: \url{cf823@drexel.edu}}\\[.2cm]
 \end{multicols}

\end{document}